\documentclass[letterpaper, reqno,11pt]{amsart}
\usepackage[margin=1.0in]{geometry}
\usepackage{color,latexsym,amsmath,amssymb, amsthm, graphicx,subfigure,revsymb, enumerate,xcolor,mathtools}

 \usepackage{todonotes}

\usepackage{amsmath}\allowdisplaybreaks % symbols like equation*
\numberwithin{equation}{section}

\renewcommand{\dim}{\mathrm{dim}\,}
\newcommand{\diam}{\mathrm{diam}}
\newcommand{\dist}{\mathrm{dist}}
\newcommand{\sub}{\subset}

\newcommand{\R}{\mathbb{R}}

\newcommand{\RR}{\mathbb{R}}

\newcommand{\ZZ}{\mathbb{Z}}

\newcommand{\eps}{\varepsilon}
\newcommand{\bs}{\backslash}

\newcommand{\Ga}{\alpha}
\newcommand{\Ge}{\varepsilon}
\newcommand{\Gl}{\lambda}
\newcommand{\Gs}{\sigma}

\DeclarePairedDelimiter{\norm}{\lVert}{\rVert}

\newtheorem{thm}{Theorem}
\numberwithin{thm}{section}
\newtheorem{conj}[thm]{Conjecture}
\newtheorem{prop}[thm]{Proposition}
\newtheorem{defn}[thm]{Definition}
\newtheorem{lem}[thm]{Lemma}

\newtheorem{cor}[thm]{Corollary}
\theoremstyle{remark}
\newtheorem{rem}[thm]{Remark}

\xdefinecolor{darkgreen}{RGB}{0, 153, 0}

\begin{document}

\title[Restricted projection to lines in $\mathbb{R}^3$]{A Furstenberg-type problem for circles, and a Kaufman-type restricted projection theorem in $\mathbb{R}^3$}
\author{Malabika Pramanik, Tongou Yang and Joshua Zahl}
\address[Malabika Pramanik]{Department of Mathematics, The University of British Columbia, Vancouver, B.C. Canada V6T 1Z2}
\email{malabika@math.ubc.ca}

\address[Tongou Yang]{Department of Mathematics, The University of British Columbia, Vancouver, B.C. Canada V6T 1Z2\\
Current Address:  Department of Mathematics, The University of California, Los Angeles, CA, United States 90095 
}

\email{tongouyang@math.ucla.edu}

\address[Joshua Zahl]{Department of Mathematics, The University of British Columbia, Vancouver, B.C. Canada V6T 1Z2}
\email{jzahl@math.ubc.ca}

%\date{today}

\begin{abstract}
    We resolve a conjecture of F\"{a}ssler and Orponen on the dimension of exceptional projections to one-dimensional subspaces indexed by a space curve in $\mathbb R^3$. We do this by obtaining sharp $L^p$ bounds for a variant of the Wolff circular maximal function over fractal sets for a class of $C^2$ curves related to Sogge's cinematic curvature condition. A key new tool is the use of lens cutting techniques from discrete geometry.
\end{abstract}

\maketitle

%\tableofcontents

\section{Introduction}
In 1954, Marstrand \cite{Marstrand} proved that if $Z\subset\RR^2$ is Borel, then $\dim(Z\cdot v)=\min(\dim Z, 1)$ for a.e.~$v\in S^1$. Here and throughout, ``$\dim$'' will mean Hausdorff dimension. This was sharpened by Kaufman \cite{Kaufman} in 1968, who bounded the size of the set of exceptional directions for which the above inequality fails.
\begin{thm}[Kaufman]\label{KaufmanThm}
Let $Z\subset\RR^2$ be Borel, and let $0\leq s < \min\{\dim Z, 1\}$. Then
\begin{equation}\label{KauffmanPlane}
\dim\{ \theta\in S^1\colon \dim(Z\cdot\theta)<s\}\leq s.
\end{equation}
\end{thm}
Marstrand and Kaufman's results have been generalized to higher dimensions and other settings, and they are the foundation of projection theory, which is an active area in geometric measure theory. See \cite{Mattila} for an introduction to the topic, or \cite{FalFraJin} for a recent survey of the area.

In this paper, we will be interested in analogues of the above results for Borel sets $Z\subset\RR^3$. As a starting point, one analogue of Marstrand's theorem says that $\dim(Z\cdot v)=\min(\dim Z, 1)$ for a.e.~$v\in S^2$. There are also analogues of Kaufman's theorem; for example, if $0\leq s < \min\{\dim Z, 1\}$ then
\begin{equation}\label{dimZLeq1}
\dim\{ v\in S^2\colon \dim(Z\cdot v)<s\}\leq 1+ s.
\end{equation}
% if $\dim Z\leq 1$, and
% \begin{equation}\label{dimZGeq1}
% \dim\{ v\in S^2\colon \dim(Z\cdot v)<s\}\leq 2 -\dim Z+ s,
% \end{equation}
%if $\dim Z\geq 1$;
See \cite{Mattila1975} for details. In particular, \eqref{dimZLeq1} says that the set of $v\in S^2$ for which $\dim(Z\cdot v)=0$ has dimension at most 1, and this can be sharp---for example, if $Z$ is a subset of the $z$-axis then $Z\cdot v = \{0\}$ for each $v$ in the great circle $S^2 \cap \{z=0\}$. It is possible that stronger bounds hold if we restrict the set of projection directions to a subset of $S^2$ that does not concentrate near great circles. In this direction, F\"assler and Orponen \cite{FasslerOrponen} conjectured the following restricted Marstrand-type estimate.
\begin{conj}\label{conj_FasslerOpronen}
 Let $I\subset\RR$ be a compact interval and let $\gamma:I\to  S^2$ be a $C^2$ curve that satisfies the ``escaping great circle'' condition
\begin{equation}\label{eqn_torsion}
    \operatorname{span}\big\{\gamma(\theta),\dot\gamma(\theta),\ddot\gamma(\theta)\big\}=\RR^3\quad\textrm{for all}\  \theta\in I.
\end{equation}
Let $Z\subset \RR^3$ be analytic. Then for almost every $\theta\in I$, $\dim(\gamma(\theta)\cdot Z) = \min(\dim Z, 1)$.
\end{conj}
In \cite{FasslerOrponen}, F\"assler and Orponen proved Conjecture \ref{conj_FasslerOpronen} in the special case $\dim(Z) \leq 1/2$. In \cite{He}, He considered the related problem of estimating the size of the set where $\dim (\gamma(\theta)\cdot Z)$ is close to $(\dim Z)/3$. In \cite{KOV2021}, K\"aenm\"aki, Orponen, and Venieri used ideas from incidence geometry, and specifically a circle tangency bound due to Wolff \cite{Wolff2000} to prove Conjecture \ref{conj_FasslerOpronen} (for all $Z$) in the special case $\gamma(\theta) = \frac{1}{\sqrt 2}(\cos \theta, \sin \theta, 1)$. Our contribution is the following Kaufman-type sharpening of Conjecture \ref{conj_FasslerOpronen}.

\begin{thm}\label{thm_KaufmanRestrictedProjThm}
Let $I\subset\RR$ be a compact interval and let $\gamma:I\to  S^2$ be a $C^2$ curve that satisfies the ``escaping great circle'' condition \eqref{eqn_torsion}. Let $Z\subset\RR^3$ be analytic and let $0\leq s < \min\{\dim Z, 1\}$. Then
\begin{equation}\label{eqn_boundOnDimTheta}
\dim\{ \theta\in I\colon \dim(Z\cdot\gamma(\theta))<s\}\leq s.
\end{equation}
In particular, Conjecture \ref{conj_FasslerOpronen} is true.
\end{thm}
\emph{Added November 14, 2022:} There have been several exiting recent developments related to Theorem \ref{thm_KaufmanRestrictedProjThm}. First, Gan, Guth and Maldague \cite{GGM} have also recently and independently resolved Conjecture \ref{conj_FasslerOpronen}, using a different method related to decoupling. Second, Gan, Guo, Guth, Harris, Maldague and Wang \cite{GGGHMW} solved a closely related conjecture of F\"{a}ssler and Orponen, where projections onto one-dimensional subspaces of $\RR^3$ are replaced by projections onto two-dimensional subspaces. Finally, Harris \cite{Harris2022} proved that if $\dim Z>1$, then $\gamma(\theta)\cdot Z$ has positive Lebesgue measure for almost every $\theta$.

%%%%%%%%%
%%%%%%%%%
%%%%%%%%%

\subsection{A Furstenberg-type problem for circles}
Theorem \ref{thm_KaufmanRestrictedProjThm} follows from an $L^p$ bound for a variant of Wolff's circular maximal function, where circles have been replaced by a class of $C^2$ curves that are the graphs of we call ``cinematic functions,'' and the domain of integration is restricted to a special type of fractal set. To state this result, we will introduce several definitions. First, we recall the definition of a $(\delta,\alpha;C)_1$-set from \cite{KatzTao}.

\begin{defn}\label{defnDeltaAlphaCSet}
Let $0<\alpha\leq 1$, $\delta >0$, $C\geq 1$. A set $E\subset[0,1]$ is called a $(\delta,\alpha;C)_1$-set if
\begin{enumerate}[(a)]
\item $E$ is a union of intervals of length $\delta$.
\item For every interval $I\subset[0,1]$, we have the Frostman-type non-concentration condition
\begin{equation}\label{nonConcentrationKatzTao}
 |E\cap I| \leq C\delta^{1-\alpha}|I|^\alpha.
\end{equation}
\end{enumerate}
\end{defn}
\noindent Since \eqref{nonConcentrationKatzTao} always holds for intervals of length $\leq\delta$, every $(\delta,\alpha;C)_1$-set is also a $(\delta,\beta;C)_1$-set if $\beta\geq \alpha$. We will be concerned with certain subsets of the plane, which are fibered products of $(\delta,\alpha;C)_1$-sets.
\begin{defn}%[Quasi-product]
Let $0<\alpha,\beta\leq 1$, $\delta >0$, and $C\geq 1$. A set $E\subset[0,1]^2$ is called a $(\delta,\alpha;C)_1\times (\delta,\beta;C)_1$ quasi-product if $E$ is Borel and can be expressed in the form
\[
E = \bigcup_{a\in A} \{a\}\times B_a,
\]
where $A$ is a $(\delta,\alpha;C)_1$-set and $B_a$ is a $(\delta,\beta;C)_1$-set for each $a\in A$.
\end{defn}

Our main result is a maximal function estimate, where the domain of integration is a quasi-product. Our next task is to introduce the class of objects that can be analyzed using our maximal function estimate. The following definition was inspired by Sogge's notion of cinematic curvature \cite{Sogge}, which was considered by Kolasa and Wolff \cite{KolasaWolff} in the context of the Wolff circular maximal function.
\begin{defn}\label{cinematicFunctionsDef}
Let $I\subset\RR$ be a compact interval and let $\mathcal{F}\subset C^2(I)$. We say $\mathcal{F}$ is a \emph{family of cinematic functions}, with \emph{cinematic constant} $K$ and \emph{doubling constant} $D$, if the following conditions hold.
\begin{enumerate}
\item $\mathcal{F}$ has diameter at most $K$ (here and in what follows, we use the usual metric on $C^2(I)$).
\item  $\mathcal{F}$ is a doubling metric space, with doubling constant at most $D$.
\item For all $f,g\in \mathcal{F},$ we have
\begin{equation}\label{cinematicFunctionCondition}
\inf_t\big( |f(t)-g(t)| + |f'(t)-g'(t)| + |f''(t)-g''(t)|\big) \geq K^{-1} \Vert f-g\Vert_{C^2(I)}.
\end{equation}
\end{enumerate}
\end{defn}

With these definitions, we can now state the following variant of Wolff's circular maximal function theorem.
\begin{thm}\label{thm_wolffAnalogue_maximal_general}
Let $\eps>0$, $0<\alpha\leq\zeta\leq 1,$ and $D,K\geq 1$. Then the following is true for all $\delta>0$ sufficiently small (depending on $D,K,\eps$ only). Let $\mathcal{F}$ be a family of cinematic functions, with cinematic constant $K$ and doubling constant $D$. Let $E$ be a $(\delta,\alpha;\delta^{-\eps})_1\times(\delta,\alpha;\delta^{-\eps})_1$ quasi-product.

Let $F\subset\mathcal{F}$ be a set of functions that satisfies the Frostman-type non-concentration condition
\begin{equation}\label{frostmanConditionOnF}
\#(F \cap B) \leq  \delta^{-\eps}(r/\delta)^\zeta\quad\textrm{for all balls}\ B\subset C^2(I)\ \textrm{of radius}\ r\geq\delta.
\end{equation}

Then
\begin{equation}\label{eqn_L32Bd_general}
 \int_E\Big( \sum_{f\in F}{\bf 1}_{f^\delta}\Big)^{3/2} \leq \delta^{2-\alpha/2-\zeta/2-C\eps}(\# F),
\end{equation}
where $C=C(D)$ depends only on $D$, and $f^\delta$ is the $\delta$-neighborhood of the graph of $f$.
\end{thm}
\medskip
\noindent \emph{Remarks.}
\begin{enumerate}[{\bf 1.}]
\item Since $[0,1]^2$ is a $(\delta,1;1)_1\times (\delta,1;1)_1$ quasi-product and $I, \mathcal{F}$ are bounded, as a special case of Theorem \ref{thm_wolffAnalogue_maximal_general} we obtain the estimate
\begin{equation}\label{allOfR2}
\Big\Vert \sum_{f\in F}{\bf 1}_{f^\delta}\Big\Vert_{3/2} \leq \delta^{-C\eps}(\delta\# F)^{2/3},
\end{equation}
where $F$ satisfies \eqref{frostmanConditionOnF} for $\zeta=1$, and the order of parameters are the same as in the statement of Theorem \ref{thm_wolffAnalogue_maximal_general}.
\medskip

\noindent
\item \label{Special_case_circle} As a special case of \eqref{allOfR2}, Theorem \ref{thm_wolffAnalogue_maximal_general} generalizes Wolff's $L^{3/2}$ circular maximal function bound from \cite{WolffKakeyacircles}. The (upper) halves of circles with centers in a small neighborhood of the origin and radii between 1 and 2 form a family of cinematic functions $\mathcal{C}$ over the interval $[-1/2, 1/2]$. If we paramaterize a circle by it's center-radius pair $(x,y,r)$, then a set $C\subset\mathcal{C}$ obeys the non-concentration condition \eqref{frostmanConditionOnF} if the corresponding collection of centre-radius pairs $P_C\subset\RR^3$ obeys a corresponding Frostman-type non-concentration condition; namely,
%the collection Since (quarter) circles form a family of cinematic functions, if $\mathcal{C}$ is a set of circles of the form $(x-x_0)^2 + (y-y_0)^2 = r_0^2$ whose corresponding center-radius pairs $(x_0,y_0,r_0)$ are $\delta$-separated and satisfy the the Frostman-type non-concentration condition
\begin{equation}\label{frostmanForCircles}
\#(P_C \cap B) \leq  \delta^{-\eps}(r/\delta)\quad\textrm{for all balls}\ B\subset \RR^3\ \textrm{of radius}\ r\geq\delta.
\end{equation}
Thus if $C$ is a set of circles with centers in $[0,1]^2$ and radii between 1 and 2, whose center-radius pairs obey \eqref{frostmanForCircles}, then Theorem \ref{thm_wolffAnalogue_maximal_general} says that
\begin{equation}
\Big\Vert \sum_{c\in C}{\bf 1}_{c^\delta}\Big\Vert_{3/2} \leq \delta^{-C\eps}(\delta\# C)^{2/3}.
\end{equation}
In \cite{WolffKakeyacircles}, Wolff proved \eqref{thm_wolffAnalogue_maximal_general} in the special case where the circles in $C$ had $\delta$-separated radii. 
\medskip

\item Theorem \ref{thm_wolffAnalogue_maximal_general} generalizes the third author's variable coefficient Wolff circular maximal function bound \cite{Zahl2012} from curves given by a $C^\infty$ defining function to curves given by a $C^3$ defining function (the latter is the minimal amount of regularity required to state Sogge's cinematic curvature condition for the defining function). A precise statement is slightly technical, and is given in Appendix \ref{cinematicCurvatureAppendix}.

\medskip

\item  Theorem \ref{thm_wolffAnalogue_maximal_general} has proved useful in other contexts beyond the proof of Theorem \ref{thm_KaufmanRestrictedProjThm}. For example, Wang and the third author \cite{WangZahl} used Theorem \ref{thm_wolffAnalogue_maximal_general} to prove a special case of the Kakeya conjecture for sticky Kakeya sets in $\RR^3$, and Katz, Wu, and the third author \cite{KatzWuZahl} used Theorem \ref{thm_wolffAnalogue_maximal_general} to prove a special case of the Kakeya conjecture for $SL_2$ Kakeya sets. Chang, Dosidis, and Kim \cite{CDK} used a corollary of Theorem \ref{thm_wolffAnalogue_maximal_general} to analyze a Nikodym-type maximal function associated to spheres in $\RR^n$. F\"assler and Orponen \cite{FasslerOrponen2} used Theorem \ref{thm_wolffAnalogue_maximal_general} to study vertical projections in the first Heisenberg group. We remark that in \cite{WangZahl}, it is crucial that Theorem \ref{thm_KaufmanRestrictedProjThm} holds for $C^2$ curves; previous maximal function bounds are insufficient.
\end{enumerate}

\noindent Theorem \ref{thm_wolffAnalogue_maximal_general} will be discussed further in Section \ref{thmthm_wolffAnalogue_maximal_generalSketchSection}, and the complete proof will be given in Section \ref{proofOfMaximalThmSection}.

%%%%%%%%%
%%%%%%%%%
%%%%%%%%%

\subsection{From maximal functions to restricted projections}\label{maximalToProjSubSec}
We will briefly sketch the proof of Theorem \ref{thm_KaufmanRestrictedProjThm} and describe the connection between Theorems \ref{thm_KaufmanRestrictedProjThm} and \ref{thm_wolffAnalogue_maximal_general}. For the purpose of this proof sketch, we will ignore many technical details. A complete proof is given in Section \ref{sec_proj_thm}. Let $Z\subset\RR^3$ be Borel with $\dim Z = \zeta\leq 1$, and let $0\leq s < \dim Z$. Let $\Theta = \{ \theta\in I\colon \dim(Z\cdot\gamma(\theta))< s\}$, and suppose for contradiction that $\dim\Theta=\alpha>s$.

In \cite{KOV2021}, K\"aenm\"aki, Orponen, and Venieri observed that this problem could be transformed to a question about curves in the plane. For each $z\in\RR^3$, define the plane curve
\begin{equation}\label{defnGammaZ}
\Gamma_z = \{(\theta, z\cdot \gamma(\theta) : \theta \in I\}.
\end{equation}
Then for each $\theta\in\Theta$, the vertical line $L_\theta = \RR\times\{\theta\}$ satisfies
$
\dim \big(L_\theta \cap \bigcup \Gamma_z \big) <s.
$

Next, we discretize the arrangement at a small scale $\delta>0$. The idea (ignoring some technical details) is that we replace $Z$ by a $\delta$-separated subset $Z_\delta$ of cardinality roughly $\delta^{-\zeta}$ that satisfies a Frostman-type non-concentration condition, and we replace $\Theta$ by a $(\delta, \alpha; \delta^{-\eps})_1$-set, which we denote by $\Theta_\delta$.  For each $z\in Z_\delta$, let $\Gamma^\delta_z$ be the $\delta$-neighborhood of the curve $\Gamma_z$. Then for most $\theta\in\Theta_\delta$, $L_\theta \cap \bigcup_{z\in Z_\delta}\Gamma^\delta_z$ will have one-dimensional Lebesgue measure at most $\delta^{1-\eta-\eps}$, and it will be contained in a $(\delta,s,\delta^{-\eps})_1$-set. Thus by H\"older's inequality we have that
\begin{equation}\label{big32Norm}
\int_{\Theta_\delta\times\RR}\Big(\sum_{z \in Z_\delta}{\bf 1}_{\Gamma^\delta_z}\Big)^{3/2} \geq \delta^{2+\frac{1}{2}s-\frac{3}{2}\zeta-\alpha+\eps}.
\end{equation}
Most of the integrand in \eqref{big32Norm} is supported on a $(\delta,\alpha;\delta^{-\eps})_1\times (\delta,\alpha;\delta^{-\eps})_1$ quasi-product. We can use the ``escaping great circle'' condition \eqref{eqn_torsion} to show that the set of curves $\{\Gamma^\delta_z\colon z\in Z_\delta\}$ are the graphs of functions from a family of cinematic functions, and the Frostman-type non-concentration condition on $Z_\delta$ implies the analogous bound \eqref{frostmanConditionOnF}. Thus we can apply Theorem \ref{thm_wolffAnalogue_maximal_general} to upper-bound \eqref{big32Norm}. Comparing \eqref{eqn_L32Bd_general} and \eqref{big32Norm} (and recalling that $Z_\delta$ has cardinality roughly $\delta^{-\zeta}$), we conclude that $\alpha\leq s$. This gives us our desired contradiction and completes the proof.

\medskip

\noindent\emph{Remarks}.
\begin{enumerate}[\bf 1.]
\item The main challenge when executing the proof strategy described above is that we must discretize the set $Z$, the set $\Theta$, and the sets $L_\theta \cap \bigcup \Gamma_z$ at a common scale $\delta$. We refer the reader to Section \ref{sec_proj_thm} for details.

\medskip
\item Our bound \eqref{eqn_L32Bd_general} in Theorem \ref{thm_wolffAnalogue_maximal_general} becomes stronger as the dimension $\alpha$ of the quasi-product $E$ decreases. With this sensitivity to $E$, Theorem \ref{thm_wolffAnalogue_maximal_general} is related to the Furstenberg set problem for curves; see \cite{Liu} for a related result on this topic. If we replace the domain of integration in \eqref{eqn_L32Bd_general} by $\RR^2$ or $[0,1]^2$, then it is possible for the LHS of \eqref{eqn_L32Bd_general} to be as large as $\delta \#F$. On the other hand, it is easy to verify that if the RHS of \eqref{eqn_L32Bd_general} were replaced by the weaker estimate $\delta \#F$, then this would not be sufficient to prove the Kaufman-type estimate from Theorem \ref{thm_KaufmanRestrictedProjThm} (though it would still be strong enough to prove Conjecture \ref{conj_FasslerOpronen}).
\end{enumerate}

%%%%%%%%%
%%%%%%%%%
%%%%%%%%%

\subsection{Proof sketch of Theorem \ref{thm_wolffAnalogue_maximal_general}}\label{thmthm_wolffAnalogue_maximal_generalSketchSection}
We conclude the introduction with a discussion of the main ideas in Theorem \ref{thm_wolffAnalogue_maximal_general} and a brief sketch the proof. In \cite{WolffKakeya}, Wolff proved that for each $\eps>0$ and all $\delta>0$ sufficiently small, if $C$ is a set of circles with $\delta$-separated radii $1\leq r\leq 2$, then
\begin{equation}\label{wolffBd}
\Big\Vert \sum_{c\in C}{\bf 1}_{c^\delta}\Big\Vert_{p}\leq \delta^{-\eps}(\delta\# C)^{1/p},\quad p = 3/2,
\end{equation}
where $c^\delta$ denotes the $\delta$-neighborhood of the circle $c$. Wolff's proof was quite technical, but it has been simplified in several follow-up works \cite{Wolff2000, Schlag}. The main ingredient in these later proofs was a bound on the number of ``tangency rectangles'' determined by the arrangement $C$. Informally, a tangency rectangle is a curvilinear rectangle $R$ of dimensions $\delta^{1/2}\times\delta$ (other dimensions are possible, but we will ignore this for now) that is contained in at least two annuli $c_1^\delta, c_2^\delta$, with  $c_1,c_2\in C$. Heuristically, we may suppose that every pair of circles $c_1,c_2\in C$ either intersect transversely, in which case $|c_1^\delta\cap c_2^\delta|\sim \delta^2$, or they intersect tangentially, in which case $|c_1^\delta\cap c_2^\delta|\sim \delta^{3/2}$ and the intersection is localized to a tangency rectangle.

The contribution to the LHS of \eqref{wolffBd} from the first type of intersection is negligible. Indeed, if every pair of circles intersected transversely, then \eqref{wolffBd} would hold for $p=2$, which is a stronger estimate. Thus the goal is to control the number of tangential intersections, or more precisely, the number of tangency rectangles where these intersections occur. Let $\mathcal{R}$ be a set of tangency rectangles that cover all the tangential intersections determined by $C$. For each $R\in\mathcal{R}$, let $m(R)$ be the number circles $c\in C$ with $R\subset c^\delta$. If the rectangles in $\mathcal{R}$ were disjoint and if the LHS of \eqref{wolffBd} was concentrated on these rectangles, then we would have a $\ell^{3/2}(\mathcal{R})\to L^{3/2}(\RR^2)$ bound of the form
\begin{equation}\label{heuristicRectangleBound}
\Big\Vert \sum_{c\in C}{\bf 1}_{c^\delta}\Big\Vert_{3/2} \leq \delta \Big(\sum_{\mathcal{R}} m(R)^{3/2}\Big)^{2/3}.
\end{equation}
We caution that reader that in general, \eqref{heuristicRectangleBound} is not true as stated. However, the heuristic \eqref{heuristicRectangleBound} was made precise by Schlag \cite{Schlag}, who used induction on scales and a bilinear $\implies$ linear argument to prove that \eqref{wolffBd} would follow from a bilinear version of the weak $\ell^{3/2}$ bound
\begin{equation}\label{boundNumberTangencyRectangles}
\#\{R \in \mathcal{R}\colon m(R)\geq k\} \leq (\# C / k)^{3/2}.
\end{equation}
Previously in \cite{Wolff2000}, Wolff had proved a suitable bilinear version of \eqref{boundNumberTangencyRectangles} using the ``cuttings'' method developed by researchers in computational geometry in the late 1980s (see \cite{CEGSW, SharirAgarwal}), and these two results combined to give a new proof of \eqref{wolffBd}.

The cuttings method makes crucial use of the property that circles are algebraic curves. In \cite{Zahl2012b}, the third author extended Wolff's arguments and the cuttings method to obtain an analogue of \eqref{wolffBd} where circles are replaced by algebraic curves of bounded degree that satisfy Sogge's cinematic curvature condition (see also \cite{KolasaWolff} for earlier work in this direction). In \cite{Zahl2012}, the third author used the newly developed polynomial partitioning technique of Guth and Katz \cite{GuthKatz} to extend \eqref{wolffBd} to smooth curves with cinematic curvature. All of these techniques, however, require that the curves are algebraic (or in the case of smooth curves, that they can be accurately approximated by moderately low degree algebraic curves). These methods do not work, and new ideas are needed, for curves that are merely $C^k$.

The main difficulty in proving Theorem \ref{thm_wolffAnalogue_maximal_general} is establishing an analogue of \eqref{boundNumberTangencyRectangles} for $C^2$ curves. We do this by adapting a result due to Marcus and Tardos \cite{MarcusTardos} from topological graph theory, which concerns objects called lenses. In brief, if $\mathcal{C}$ is a set of Jordan curves in the plane, then a \emph{lens} is a pair of curve-segments $\lambda_1 \subset\gamma_1;$ $\lambda_2\subset\gamma_2$; $\gamma_1,\gamma_2\in\mathcal{C},$ so that $\lambda_1$ and $\lambda_2$ intersect at their endpoints, and $\lambda_1\cup\lambda_2$ is homeomorphic to $S^1$. For example, if two circles $c_1,c_2$ with well-separated centers share a common tangency rectangle, then after a small perturbation we expect to find arcs of $c_1$ and $c_2$ of length roughly $\delta^{1/2}$ that create a lens.

Lemma 10 from \cite{MarcusTardos} says that an arrangement of $n$ closed Jordan curves that satisfy a certain ``pseudo-circle'' property determines $O(n^{3/2}\log n)$ pairwise non-overlapping lenses. We use this result to obtain an analogue of \eqref{boundNumberTangencyRectangles} by carefully perturbing our arrangement of curves, so that the number of tangency rectangles in $\mathcal{R}$ is comparable to the number of lenses. See Proposition \ref{thm_number_estimate} for a precise statement, and Section \ref{lensCountingBoundSection} for the details of this argument.

With Proposition \ref{thm_number_estimate} in hand, it remains to prove Theorem \ref{thm_wolffAnalogue_maximal_general}. We encounter several difficulties. First, our bilinear analogue of \eqref{boundNumberTangencyRectangles} is more general than the corresponding estimates from \cite{Wolff2000, Zahl2012b, Zahl2012} because it applies to $C^2$ curves rather than algebraic (or smooth) curves. However, our bilinear analogue of \eqref{boundNumberTangencyRectangles} requires slightly stronger assumptions, which do not work well with induction on scales. Thus we need a new bilinear $\implies$ linear argument that does not use induction on scales. Closely related to this difficulty is the annoyance that intersections between curves need not be completely tangential or transverse---there are many intermediate amounts of tangency or transversality. We handle these difficulties by combining a $L^2$ argument at fine scales with a suitable bilinear version of \eqref{boundNumberTangencyRectangles} at coarse scales. In brief, we define a parameter $\Delta$ that measures the typical amount of tangency between curves; if $\Delta = \delta$ then the main contribution to the LHS of \eqref{wolffBd} comes from pairs of curves that intersect completely tangentially, while if $\Delta = 1$ then the main contribution comes from pairs of curves that intersect completely transversely. We then divide $\RR^2$ into rectangles of dimensions $\Delta\times\Delta^{1/2}$. If two curves intersect, then typically this intersection is localized to one such rectangle, and the curves are tangent at scale $\Delta$. Thus we can use a bilinear analogue of \eqref{boundNumberTangencyRectangles} to control the number of tangency rectangles at scale $\Delta$. Within each rectangle, most pairs of curves intersect with a precisely prescribed amount of transversality. Thus we can use a $L^2$ argument to control the contribution from each rectangle. This is the most technical part of the paper, and the details are in Section \ref{proofOfMaximalThmSection}.

%%%%%%%%%
%%%%%%%%%
%%%%%%%%%
\subsection{Notation}\label{notationSection}
To reduce clutter, we will write $A\lesssim B$ or $A = O(B)$ if there is a constant $C>0$ so that $A\leq CB$. The constant $C$ will be allowed to depend on several quantities, which are context-dependent. Based on context, $C$ may depend on any of the quantities $|I|$, $\gamma$, and $\dim Z$ from the statement of Theorem \ref{thm_KaufmanRestrictedProjThm}, and also the quantites $K$, $D$, and $\eps$ from the statement of Theorem \ref{thm_wolffAnalogue_maximal_general}.

In the arguments that follow, we will frequently introduce a small parameter $\delta>0$, and we will employ dyadic pigeonholing arguments that replace a set $X$ with a subset of size roughly $|\log\delta|^{-1}(\#X)$. We will write $A\lessapprox B$ or $B\gtrapprox A$ if $A\lesssim|\log\delta|^{C_0}B$ for some absolute constant $C_0$ (in practice, we will always have $C_0\leq 100$). When we write $A\lessapprox B$, the relevant quantity $\delta$ will always be apparent from context.

\subsection{Thanks}
The authors would like to thank Orit Raz for many helpful comments and suggestions throughout the preparation of this manuscript. We would also like to thank Alan Chang, Jongchon Kim, and the anonymous referee for comments, suggestions, and corrections to an earlier version of this manuscript. All authors were supported by NSERC Discovery grants.

%%%%%%%%%%%%%%%%%%%%%%%%
%%%%%%%%%%%%%%%%%%%%%%%%
%%%%%%%%%%%%%%%%%%%%%%%%

\section{From Theorem \ref{thm_wolffAnalogue_maximal_general} to Theorem \ref{thm_KaufmanRestrictedProjThm}}\label{sec_proj_thm}
In this section, we will expand upon the sketch from Section \ref{maximalToProjSubSec}. The following version of Theorem \ref{thm_wolffAnalogue_maximal_general} is a main ingredient of the proof.
\begin{prop}\label{thm_wolff_forGamma}
Let us fix $\eps>0$ and dimensional parameters $0<\alpha\leq\zeta\leq 1$. Let $I$ be a compact interval and let $\gamma\colon I\to S^2$ be a curve satisfying \eqref{eqn_torsion}. Then there exists a positive constant $\delta_0$ depending only on these quantities, such that the following is true for all $0 < \delta \leq \delta_0$.

Let $E$ be a $(\delta,\alpha;\delta^{-\eps})_1\times(\delta,\alpha;\delta^{-\eps})_1$ quasi-product. Let $Z_\delta\subset B(0,1)\subset\RR^3$ be a $\delta$-separated set that satisfies the Frostman-type non-concentration condition
\begin{equation}\label{eqn_frostmanNonConcentrationRequirementInProp}
\#(Z_\delta \cap B) \leq  \delta^{-\eps}(r/\delta)^\zeta\quad\textrm{for all balls}\ B\subset \RR^3\ \textrm{of radius}\ r\geq\delta.
\end{equation}
Then
\begin{equation}\label{eqn_L32BdGamma}
\int_E \Big(\sum_{z\in Z_\delta}{\bf 1}_{\Gamma^\delta_z} \Big)^{3/2}\leq \delta^{2-\alpha/2-\zeta/2-C\eps}(\# Z_\delta),
\end{equation}
where $C>0$ is an absolute constant, and $\Gamma^\delta_z$ is the $\delta$-neighborhood of the graph $\Gamma_z$ defined in \eqref{defnGammaZ}.
\end{prop}
\begin{proof}
Proposition \ref{thm_wolff_forGamma} follows from Theorem \ref{thm_wolffAnalogue_maximal_general}, applying the latter to the family of functions
\[ 
F = \bigl\{f_z: z \in Z_{\delta} \bigr\} \subset \mathcal F = \bigl\{f_z: z \in B(0;1) \bigr\}, 
\] 
where $f_z(\theta) = \gamma(\theta)\cdot z$ and $\gamma\colon I\to S^2$ is a curve satisfying \eqref{eqn_torsion}. In order for Theorem \ref{thm_wolffAnalogue_maximal_general} to be applicable, we need to verify that $\mathcal F$ is cinematic (with cinematic constant that depends only on $\gamma$) and that $F$ satisfies a Frostman-type non-concentration bound \eqref{frostmanConditionOnF}. The verification proceeds as follows. First, we have
\[ 
\sup_{f\in \mathcal F}\norm{f}_{C^2(I)} =  \sup_{z\in B(0;1)}\norm{f_z}_{C^2(I)}  \leq \sup_{\theta \in I} \bigl[ |\gamma(\theta)| + |\gamma'(\theta)| + |\gamma''(\theta)| \bigr] |z| \leq ||\gamma||_{C^2(I)} < \infty,
\]
and depends only on $\gamma$. Next, if $z_1,z_2\in B(0,1)$ then $\Vert f_{z_1} - f_{z_2}\Vert_{C^2(I)}\leq ||\gamma||_{C^2(I)} |z_1-z_2|$. Furthermore, the non-degeneracy condition \eqref{eqn_torsion} implies that
\[
\inf_{\theta\in I}\Big[|f_{z_1}(\theta)-f_{z_2}(\theta)| + |f'_{z_1}(\theta)-f'_{z_2}(\theta)| + |f''_{z_1}(\theta)-f''_{z_2}(\theta)|\Big]\gtrsim  \Vert f_{z_1} - f_{z_2}\Vert_{C^2(I)},
\]
where the implicit constant depends only on $\gamma$. As a consequence,
\begin{equation} \label{C^2-Euclidean}
\Vert f_{z_1}-f_{z_2}\Vert_{C^2(I)}\sim |z_1-z_2| \text{ for } z_1, z_2 \in B(0;1).
\end{equation}
 Thus $\mathcal{F}$ is a doubling metric space (with the same doubling constant as $\R^3$). Finally, identifying $F$ with $Z_{\delta}$, we observe from \eqref{C^2-Euclidean} that the Frostman-type non-concentration condition \eqref{frostmanConditionOnF} for $F$ follows from the analogous condition \eqref{eqn_frostmanNonConcentrationRequirementInProp} for the Euclidean set $Z_{\delta}$.
\end{proof}
%Finally, we remark that replacing $\eps$ by $2\eps$ in our hypotheses on $E$ in the statement of Proposition \ref{thm_wolff_forGamma} is harmless, and it will streamline our notation below.

%Let us recall the definition of the graph $\Gamma_z$ from \eqref{defnGammaZ}. As in Remark \ref{Special_case_circle} on page \pageref{Special_case_circle}, one can check that a family $\mathfrak Z$ of graphs $\Gamma_z$ is cinematic if the corresponding parameter family $Z_{\delta} = \{ z \in \mathbb R^3 : \Gamma_z \in \mathfrak \Gamma \}$ is $\delta$-separated and obeys a Frostman-type non-concentration condition in three-dimensional Euclidean space. Henceforth identify the family of curves $\mathfrak Z$ with the corresponding family of parameters $Z_{\delta}$.

We are now ready to prove Theorem \ref{thm_KaufmanRestrictedProjThm}. Let $\gamma\colon I\to S^2$ satisfy \eqref{eqn_torsion} and let $Z\subset\RR^3$ be Borel. After a harmless re-scaling, we can suppose $Z\subset B(0,1)$ and after possibly replacing $Z$ by a subset, we may suppose that $\dim Z\leq 1$. Let $0\leq s<\dim Z$. Aiming for a contradiction, suppose if possible that \eqref{eqn_boundOnDimTheta} fails, i.e.~$\dim\Theta_I > s$, where \[ \Theta_I = \bigl\{\theta \in I: \dim \bigl(Z \cdot \gamma(\theta) \bigr) < s \bigr\}. \]  Then $\Theta_I$ admits a subset $\Theta$ with the propery that $s < \dim\Theta \le \dim Z$. Set $\alpha = \dim \Theta$. We will show that the existence of such a set $\Theta \subset \Theta_I$ contradicts the validity of Proposition \ref{thm_wolff_forGamma}.

Define
\begin{equation}\label{defnEps}
\eps = \frac{\alpha-s}{2C+11} > 0,
\end{equation}
where $C$ is the constant from \eqref{eqn_L32BdGamma}.

%%%%%%%%%
%%%%%%%%%
%%%%%%%%%

\subsection{Initial discretization}
Bourgain \cite{Bourgain_sum_product} proved a discretized projection theorem, which bounded the number of exceptional directions where the projection of a $\delta$-discretized set can be small. Bourgain then showed how this discretized theorem implies an analogous Hausdorff dimension bound. The next set of arguments are similar to the reduction from Theorem 4 to Theorem 3 in \cite{Bourgain_sum_product}, but with a twist. The location of the twist will be noted below.

Let $\zeta =\dim Z$ and let $\mathbb{P}$ be a probability measure supported on $Z$ that satisfies the Frostman condition
\begin{equation}\label{FrostmanOnP}
\mathbb{P}(B)\lesssim r^{\zeta-\eps}\qquad\textrm{for all balls}\ B\subset\RR^3\ \textrm{of radius}\ r>0.
\end{equation}

Let $\nu$ be a probability measure supported on $\Theta$ with
\begin{equation}\label{FrostmanOnNu}
\nu(J)\lesssim |J|^{\alpha-\eps}\ \quad\textrm{for all intervals}\ J\subset\RR.
\end{equation}

Let $k_0\in\ZZ$ be chosen sufficiently large, depending on $\gamma$, $\zeta$, $s$ and $\alpha$. The choice of $k_0$ will be determined as follows: our arguments below will select a number $0<\delta\leq 2^{-\lfloor(1+\eps)^{k_0}\rfloor}$ (as in \cite[Section 7]{KatzTao}; we do not get to choose $\delta$, but our choice of $k_0$ gives an upper bound for $\delta$). Eventually, we will arrive at the conclusion that a certain inequality involving $\delta$ is impossible, provided $\delta>0$ is sufficiently small (depending on $\gamma$, $\zeta$, $s$ and $\alpha$). In order to arrive at a contradiction, we will select $k_0\in\ZZ$ sufficiently large so that any positive $\delta \leq 2^{-\lfloor(1+\eps)^{k_0}\rfloor}$ obeys Proposition \ref{thm_wolff_forGamma}.

For each $\theta\in I$, define $\rho_\theta\colon \RR^3\to\RR$ by $\rho_\theta(z)=z \cdot \gamma(\theta)$. For each $\theta\in\operatorname{supp}(\nu)\sub \Theta$, cover
%\todo{Katz-Tao only for compact sets}
$\rho_\theta(Z)$ by a union of sets $\{ G_{\theta,k} \}_{k=k_0}^\infty$, where  $G_{\theta,k}$ is a $(2^{-\lfloor(1+\eps)^{k}\rfloor},s; 2^{\eps\lfloor(1+\eps)^{k}\rfloor})_1$-set (in Definition \ref{defnDeltaAlphaCSet} a $(\delta,\alpha;C)_1$-set must be contained in $[0,1]$, but it is harmless to apply a translation); see i.e.~\cite[Lemma 7.5]{KatzTao} for details on how to do this. This is the ``twist'' alluded to above; in \cite{Bourgain_sum_product}, Bourgain simply covers $\rho_\theta(A)$ by unions of $2^{-k}$ intervals; he does not ask that these unions satisfy the non-concentration condition \eqref{nonConcentrationKatzTao}.

For each $\theta\in\operatorname{supp}(\nu)$ we have
\[
1=\mathbb{P}(Z)=\mathbb{P}\big(Z \cap \rho_\theta^{-1}(\rho_\theta(Z) \big) = \mathbb{P}\Big(Z \cap \bigcup_{k\geq k_0}\rho_\theta^{-1}(G_{\theta,k}) \Big).
\]
Thus by integrating over $I$ with respect to $\nu$  we conclude that
\[
\int_I \sum_{k\geq k_0} \mathbb{P}\big(Z \cap \rho_\theta^{-1}(G_{\theta,k}\big)\big)d\nu(\theta) \geq \nu(\Theta)= 1.
\]
In particular, there exists $k\geq k_0$ so that
\begin{equation}\label{selectionOfK}
\int_I \mathbb{P}\big(Z \cap \rho_\theta^{-1}(G_{\theta,k})\big)d\nu(\theta) \gtrsim \frac{1}{k^2}.
\end{equation}
We will fix this choice of $k$, and for the remiander of this proof, define $\delta = 2^{-\lfloor(1+\eps)^{k}\rfloor}$ and $G_{\theta}=G_{\theta,k}$; thus the right hand side of \eqref{selectionOfK} becomes essentially $\left(\frac{\log (1 + \eps)}{\log \log (1/\delta)}\right)^{2}$, which is $\geq |\log\delta|^{-2}$ if $\delta$ is sufficiently small relative to $\eps$.

Recalling the notation introduced in Section \ref{notationSection}, we will write $A\gtrapprox B$ to mean $A\gtrsim |\log\delta|^{-C}B$, where $C$ is an absolute constant (in practice, we will always have $C\leq 100$).

%%%%%%%%%
%%%%%%%%%
%%%%%%%%%

\subsection{Replacing $\mathbb{P}$ by a sum over curves}\label{replacePBySumOverCurvesSec}
Given the choice of $\delta$ above, let us consider a covering of $Z$ using cubes of sidelength $\delta/2$. After dyadic pigeonholing, we can find a set $\mathcal{Q}$ of such cubes and a ``weight'' $w$, so that $w \leq\mathbb{P}(Q)<2w$ for each $Q\in\mathcal{Q}$, and
\begin{equation}\label{integralUnionOfQ}
\int_I \mathbb{P}\Big(\bigcup_{Q\in\mathcal{Q}} Q \cap \rho_\theta^{-1}(G_{\theta})\Big)d\nu(\theta) \gtrapprox 1.
\end{equation}
Let $ G_{\theta}' =N_{\delta}(G_{\theta})$; here and henceforth, $N_\delta(X)$ denotes the $\delta$-neighbourhood of the set $X$. Then $G_{\theta}'$ is a $(3\delta,s; 3\delta^{-\eps})_1$-set, and if $Q\in\mathcal{Q}$ is a cube that intersects $\rho_\theta^{-1}(G_{\theta})$, then $Q\subset \rho_\theta^{-1}(G_{\theta}')$. Thus \eqref{integralUnionOfQ} implies that
\[
\int_I w\#\big\{Q \in \mathcal{Q}\colon Q\subset \rho_\theta^{-1}(G_{\theta}')\big\} d\nu(\theta)\gtrapprox 1,
\]
or
\begin{equation}\label{lowerBdOnCubesInGThetaPrime}
\int_I \#\big\{Q \in \mathcal{Q}\colon \rho_\theta(Q) \subset G_{\theta}'\big\} d\nu(\theta)\gtrapprox w^{-1}.
\end{equation}
Since $Q$ is a cube of side-length $\delta/2$ and $\gamma(\theta)\in S^2$, we have that $\rho_\theta(Q)$ is an interval in $\RR$ of length $\sim\delta$. Thus $\int_{G_{\theta}'}{\bf 1}_{\rho_\theta(Q)}(y)dy\sim\delta$ for each $Q\in\mathcal{Q}$ with $\rho_\theta(Q) \subset G_{\theta}'$. Hence \eqref{lowerBdOnCubesInGThetaPrime} implies that
\begin{equation}\label{sumOverQ}
\int_I \int_{G_{\theta}'} \sum_{Q\in\mathcal{Q}} {\bf 1}_{\rho_\theta(Q)}(y)dy d\nu(\theta)\gtrapprox \delta w^{-1}.
\end{equation}

Note that $\delta^3\lesssim w\lesssim \delta^{\zeta-\eps}$, and $1\lessapprox w\#\mathcal{Q}\leq 1$. Let $A=\delta^{\zeta-\eps}w^{-1}\gtrsim 1$, so that $\#\mathcal{Q}\gtrapprox A\delta^{-\zeta+\eps}$ and $w = \delta^{\zeta - \eps} A^{-1}$. Thus the right hand side of \eqref{sumOverQ} becomes $A\delta^{1-\zeta+\eps}$.

Let $Z_\delta'$ denotes the set of centers of the cubes in $\mathcal{Q}$; then the Frostman condition \eqref{FrostmanOnP} implies that
\begin{equation}\label{nonConcentrationOfZ0}
\#(Z_\delta' \cap B)\lessapprox w^{-1} \mathbb P(B) \lesssim w^{-1} r^{\zeta - \eps} =  A (r/\delta)^{\zeta-\eps} \qquad\textrm{for all balls}\ B\ \textrm{of radius}\ r\geq\delta.
\end{equation}
Note that if $Q\in\mathcal{Q}$ has center $z\in Z_\delta'$, then $Q\subset N_{\delta}(z)$ (this is why we required that the cubes in $\mathcal{Q}$ have side-length $\delta/2$, rather than $\delta$). Thus for each $\theta\in I$, we have the implication
\[ y\in \rho_{\theta}(Q) \implies |\gamma(\theta) \cdot z - y| \leq \delta. \]
This leads to the inequality
\begin{equation}\label{cubesDominatedByCurves}
\sum_{Q\in\mathcal{Q}} {\bf 1}_{\rho_\theta(Q)}(y)\leq \sum_{z\in Z_\delta'} {\bf 1}_{\Gamma^{\delta}_z}(\theta,y).
\end{equation}
%where $\Gamma^{\delta}_z$ is the curve from \eqref{eqn_vertical_delta_nbhd_proj}.
Combining \eqref{sumOverQ} and \eqref{cubesDominatedByCurves}, we obtain
\begin{equation}\label{sumOfCurvesZ0}
\int_I \int_{G_{\theta}'} \sum_{z\in Z_\delta'} {\bf 1}_{\Gamma^{\delta}_z}(\theta,y) dy d\nu(\theta)\gtrapprox A\delta^{1-\zeta + \eps}.
\end{equation}

Let $Z_\delta$ be a subset of $Z'_{\delta}$ obtained by randomly and independently selecting each point of $Z'_{\delta}$ with probability $p = A^{-1}$; we will select the set $Z_\delta$ so that the following properties hold.

 % (A): $\#\mathcal{Z}\approx \delta^{-\zeta}$. (B):
\begin{equation}\label{nonConcentrationOfBallsInZ}
\#(Z_\delta \cap B)\leq |\log \delta|A^{-1} \#(Z_\delta' \cap B) \le \delta^{-\eps} (r/\delta)^{\zeta} \qquad\textrm{for all balls}\ B\ \textrm{of radius}\ r\geq\delta,
\end{equation}
and
\begin{equation}\label{sumOverCalZ}
\int_I \int_{G_{\theta}'} \sum_{z\in Z_\delta} {\bf 1}_{\Gamma^{\delta}_z}(\theta,y) dy d\nu(\theta)\gtrapprox \delta^{1-\zeta+\eps}.
\end{equation}
Let us briefly discuss the selection of $Z_\delta$. Apriori, \eqref{nonConcentrationOfBallsInZ} appears to be an assertion about an uncountable set of balls $\{B\subset\RR^3\colon \operatorname{radius}(B)\geq \delta\}$. However, it suffices to establish \eqref{nonConcentrationOfBallsInZ} (with the RHS replaced by $\frac{1}{2}\delta^{-\eps} (r/\delta)^{\zeta}$) for a set of roughly $\delta^{-3}$ balls: For each $j=0,\dots,|\log\delta|$, let $\mathcal{B}_j$ be a set of about $(2^j\delta)^{-3}$ balls of radius $2^{j+1}\delta$, so that every ball of radius $2^j\delta$ is contained in at most $O(1)$ balls from $\mathcal{B}_j$. Then it suffices to prove \eqref{nonConcentrationOfBallsInZ} for the balls in $\mathcal{B} = \bigcup_{j=0}^{|\log\delta|}\mathcal{B}_{j}$; this set has cardinality $\lesssim\delta^{-3}$. If $\delta$ is small enough,
%the probability $p\approx A^{-1}$ is chosen appropriately,
then using Chernoff's inequality and \eqref{nonConcentrationOfZ0} we can ensure that the probability that \eqref{nonConcentrationOfBallsInZ} fails for any particular ball in $\mathcal{B}$ is smaller than $\delta^4$. Applying the union bound over $B \in \mathcal B$, we deduce that with probability at least $1 - O(\delta)$, the estimate \eqref{nonConcentrationOfBallsInZ} holds for all balls $B \in \mathcal B$ and hence in general for all balls of radius $r \geq \delta$. Choosing such an outcome with the additional requirement (also derivable from Chernoff's inequality) that
\[ \# Z_{\delta} > \frac{1}{100A} \# Z_{\delta}', \] the property \eqref{sumOverCalZ} follows from \eqref{sumOfCurvesZ0} and an expected value calculation (expectation is linear) and Markov's inequality.

%%%%%%%%%
%%%%%%%%%
%%%%%%%%%

\subsection{Replacing $\nu$ by an integral over a quasi-product}
%So far, we have managed to replace the probability measure $\mathbb{P}$ by (a suitably normalized) Lebesgue measure on each vertical strip $\{(\theta,y)\colon y\in\RR\}$. Our next task is to replace the measure $d\nu(\theta)$ by (a suitably normalized) Lebesgue measure on some $\delta$-discretized set.
Let $\mathcal{J}_0$ be the set of intervals in $I$ (recall $I$ is the domain of $\gamma(\theta)$) of the form $[n\delta, (n+1)\delta)$, for $n\in\mathbb{Z}$. For each $J\in\mathcal{J}_0$, let $\theta(J)\in J$ be a value of $\theta$ that gets within a factor of 2 of maximizing the quantity
\[
\int_{G_{\theta}'} \sum_{z\in Z_\delta} {\bf 1}_{\Gamma^{\delta}_z}(\theta,y) dy.
\]
For each $J\in\mathcal{J}_0,$ we have
\begin{equation}\label{isolatingThetaJ}
\begin{split}
\nu(J)\int_{G_{\theta(J)}'} \sum_{z\in Z_\delta} {\bf 1}_{\Gamma^{\delta}_z}(\theta(J),y) dy & =  \int_J \int_{G_{\theta(J)}'} \sum_{z\in Z_\delta} {\bf 1}_{\Gamma^{\delta}_z}(\theta(J),y) dy d\nu(\theta)\\
&\geq \frac{1}{2} \int_J \int_{G_{\theta}'} \sum_{z\in Z_\delta} {\bf 1}_{\Gamma^{\delta}_z}(\theta,y) dy d\nu(\theta).
\end{split}
\end{equation}
The first equality is just the statement that $\nu(J)=\int_J d\nu(\theta)$, while the second inequality follows from the definition of $\Theta(J)$.
To simplify notion, define
\[
K(J)=\int_{G_{\theta(J)}'} \sum_{z\in Z_\delta} {\bf 1}_{\Gamma^{\delta}_z}(\theta(J),y) dy,
\]
so the left hand side of \eqref{isolatingThetaJ} is $\nu(J)K(J)$.
Summing \eqref{isolatingThetaJ} over $J\in\mathcal{J}_0$, we obtain
\begin{equation*}
\begin{split}
\sum_{J\in\mathcal{J}_0}\nu(J) K(J) & \geq \frac{1}{2} \sum_{J\in\mathcal{J}_0}\int_J \int_{G_{\theta}'} \sum_{z\in Z_\delta} {\bf 1}_{\Gamma^{\delta}_z}(\theta,y) dy d\nu(\theta)\\
& = \frac{1}{2}\int_I \int_{G_{\theta}'} \sum_{z\in Z_\delta} {\bf 1}_{\Gamma^{\delta}_z}(\theta,y) dy d\nu(\theta)\\
&\gtrapprox \delta^{1-\zeta+\eps}.
\end{split}
\end{equation*}
Observe that $K(J)\lesssim \delta\# Z_\delta \lessapprox \delta^{1-\zeta+\eps}$ for each $J\in\mathcal{J}_0$, and $\sum_{J\in\mathcal{J}_0}\nu(J)=1$. Thus by dyadic pigeonholing, there is a set $\mathcal{J}_1\subset\mathcal{J}_0$ and a weight $\delta\lesssim w'\leq \delta^{\alpha - \eps}$ so that $w'\leq \nu(J)<2w'$ for each $J\in\mathcal{J}_1$, and $1\lessapprox w'\#\mathcal{J}_1\leq 1$. We can write $A'=\delta^{\alpha - \eps} w'^{-1}\gtrsim 1$, and thus $\#\mathcal{J}_1\approx A'\delta^{-\alpha + \eps}$, and we have
\[
\sum_{J\in\mathcal{J}_1} \delta^{\alpha - \eps}(A')^{-1} K(J)\gtrapprox \delta^{1-\zeta+\eps},
\]
i.e.
\[
\sum_{J\in\mathcal{J}_1} K(J)\gtrapprox A'\delta^{1-\zeta-\alpha+2\eps}.
\]

We will now perform a random sampling argument analogous to how we selected $ Z_\delta\subset Z_\delta'$ at the end of Section \ref{replacePBySumOverCurvesSec}. Let $p = (A')^{-1}$, and select $\mathcal{J}\subset\mathcal{J}_1$ by randomly and independently choosing each $J\in\mathcal{J}_1$ with probability $p$. By \eqref{FrostmanOnNu}, with positive probability, we may choose $\mathcal{J}$ so that the following two properties hold.

\begin{equation}\label{nonConcentrationOfintervalsInZ}
| \Lambda \cap \bigcup_{J\in\mathcal{J}}J| \leq \delta|\log \delta|   (|\Lambda|/\delta)^{\alpha-\eps}\leq \delta^{1-\eps}(|\Lambda|/\delta)^{\alpha-\eps}\qquad\textrm{for all intervals}\ \Lambda,
\end{equation}
(i.e.~$\bigcup_{J\in\mathcal{J}''}J$ is a $(\delta,\alpha-\eps;\delta^{-\eps})_1$-set), and 
\begin{equation}\label{stillLargeMass}
\sum_{J\in\mathcal{J}} K(J)\gtrapprox \delta^{1-\zeta-\alpha+2\eps}.
\end{equation}
\eqref{nonConcentrationOfintervalsInZ} and \eqref{stillLargeMass} are the analogues of \eqref{nonConcentrationOfBallsInZ} and \eqref{sumOverCalZ}, respectively, and the argument is the same.

% i.e.
% \begin{equation}\label{lowerBdDeltaKJSum}
% \sum_{J\in\mathcal{J}''} \delta K(J)\gtrapprox \delta^{2-\zeta-\alpha}
% \end{equation}

Next, we claim that if $(\theta,y)\in \Gamma^{\delta}_z$, and if $|\theta-\theta'|\leq \delta$, then $\theta'\in \Gamma^{L\delta}_z$, where
\[
L = \sup_{z \in Z_\delta}\sup_{\theta\in I}\left| \frac{d}{d\theta} \rho_\theta(z)\right| \lesssim 1.
\]
%where the implicit constant depends on the curve $\gamma$.

Thus for each $z\in Z_\delta$, each $y\in\RR$, and each $J\in\mathcal{J}$, we have
\begin{equation}\label{thickDom}
\int_J  {\bf 1}_{\Gamma^{\delta}_z}(\theta(J),y)d\theta \leq \int_J  {\bf 1}_{\Gamma^{L\delta}_z}(\theta,y)d\theta.
\end{equation}

Integrating \eqref{thickDom} over $G_{\theta(J)}'$ and summing over $z\in Z_\delta$ (and interchanging the order of integration and summation), we obtain
\begin{equation*}
\begin{split}
\delta K(J) = & \int_J \int_{G_{\theta(J)}'} \sum_{z\in Z_\delta} {\bf 1}_{\Gamma^{\delta}_z}(\theta(J),y) dyd\theta\\
& \leq \int_J \int_{G_{\theta(J)}'} \sum_{z\in Z_\delta}{\bf 1}_{\Gamma^{L\delta}_z}(\theta,y)dyd\theta.
\end{split}
\end{equation*}
Combining this with \eqref{stillLargeMass}, we conclude that
\begin{equation}\label{lowerBdSumOverJ}
\begin{split}
\delta^{2-\zeta-\alpha+2 \eps} & \lessapprox \sum_{J\in\mathcal{J}} \delta K(J)\\
& \leq \sum_{J\in\mathcal{J}''}\int_J \int_{G_{\theta(J)}'} \sum_{z\in Z_\delta}{\bf 1}_{\Gamma^{L\delta}_z}(\theta,y)dyd\theta.
%& = \int_H \int_{G_{\theta(J)}'} \sum_{z\in Z_\delta}{\bf 1}_{\Gamma^{L\delta}_z}(\theta,y)dyd\theta.
\end{split}
\end{equation}
Finally, define
\[
E = \bigcup_{J\in\mathcal{J}}J\times G_{\theta(J)}'.
\]
Then $E$ is a $(\delta,\alpha-\eps;\delta^{-\eps})_1\times(3 \delta,s;3 \delta^{-\eps})_1$ quasi-product, and \eqref{lowerBdSumOverJ} can be re-written as
\begin{equation}\label{massOverF}
\int_E \sum_{z\in Z_\delta}{\bf 1}_{\Gamma^{L\delta}_z} \gtrapprox \delta^{2-\zeta-\alpha+2\eps}.
\end{equation}
Since
\[ |E|\lesssim (\delta^{1 - \alpha + \eps} \delta^{-\eps}) \times  \bigl((3 \delta)^{1 - s} 3 \delta^{-\eps}\bigr) \sim \delta^{2-\alpha-s-2\eps}, \]
the relation \eqref{massOverF} together with H\"older's inequality implies that
\begin{equation}\label{lowerBoundL32NormOverF}
\int_E \Big(\sum_{z\in Z_\delta}{\bf 1}_{\Gamma^{L\delta}_z} \Big)^{3/2} \gtrapprox \delta^{2+\frac{1}{2}s - \alpha - \frac{3}{2}\zeta+4\eps}.
\end{equation}

%%%%%%%%%
%%%%%%%%%
%%%%%%%%%

\subsection{Applying Proposition \ref{thm_wolff_forGamma}}
Since $E$ is a $(\delta,\alpha-\eps;\delta^{-\eps})_1\times(3\delta,s;3\delta^{-\eps})_1$ quasi-product and $s<\alpha$, $E$ can be expressed as a union of $O(1)$ many $(\delta, \alpha, \delta^{-\eps})_1 \times (\delta, \alpha, \delta^{-\eps})_1$ quasi-products.
%also a $(\delta,\alpha;\delta^{-2\eps})_1\times (\delta,\alpha;\delta^{-2\eps})_1$ quasi-product.
Thus we can apply Proposition \ref{thm_wolff_forGamma} with $\# Z_\delta\lessapprox \delta^{-\zeta-\eps}$ and use the triangle inequality to conclude that
\begin{equation}\label{upperBoundL32NormOverF}
\int_E \Big(\sum_{z\in Z_\delta}{\bf 1}_{\Gamma^{L\delta}_z} \Big)^{3/2}\lessapprox \delta^{2-\frac{1}{2}\alpha-\frac{3}{2}\zeta-C\eps-\eps}.
\end{equation}
Comparing \eqref{lowerBoundL32NormOverF} and \eqref{upperBoundL32NormOverF}, we conclude that $\delta^{s-\alpha}\lessapprox \delta^{-(2C+10)\eps}$. By selecting $\delta>0$ sufficiently small (depending on $\eps$), we conclude that $\alpha-s\leq (2C+10)\eps$. But comparing with \eqref{defnEps}, we obtain a contradiction. This finishes the proof of Theorem \ref{thm_KaufmanRestrictedProjThm}.

\endproof

%%%%%%%%%%%%%%%%%%%%%
%%%%%%%%%%%%%%%%%%%%%
%%%%%%%%%%%%%%%%%%%%%

\section{The geometry of cinematic functions}\label{geomOfCinematicSection}
In this section, we will explore some geometric properties of families of cinematic functions. Many of the results in this section are inspired by analogous results in \cite{KOV2021, KolasaWolff, Wolff2000}. For the definitions that follow, we will fix a family of cinematic functions $\mathcal{F}$ over a compact interval $J$, with cinematic constant $K$ and doubling constant $D$.

For $f\in\mathcal{F}$, we write $f^\lambda$ to denote the closed $\lambda$-vertical neighborhood of the graph of $f$, i.e. the set
$
\{(\theta,y)\in J\times\RR\colon |f(\theta)-y|\leq\lambda\}.
$
This definition differs slightly from the definition of $f^\delta$ given in the statement of Theorem \ref{thm_wolffAnalogue_maximal_general}. This distinction does not matter, however, since $f^\lambda$ is comparable to the $\lambda$-neighborhood of the graph of $f$.

\begin{defn}[Curvilinear rectangles]\label{defn_delta_t_rectangle}
Let $0<\delta\leq 1,$ $\delta\leq t\leq 1$, and $f\in\mathcal{F}$. A $(\delta,t)$-rectangle is defined to be a set of the form
$$
f^\delta(I)=\{(\theta,y)\in I\times \RR\colon |y-f(\theta)|\le \delta\},
$$
where $I\subset J$ is an interval of length $\sqrt{\delta/t}$.
\end{defn}

\begin{defn}[Comparability]
We say two $(\delta,t)$-rectangles $R,R'$ are $\Gl$-comparable, if there is another $(\Gl\delta,t)$-rectangle $R''$ that contains $R\cup R'$. Otherwise we say $R,R'$ are $\Gl$-incomparable.
\end{defn}

\begin{defn}[Tangency]\label{defn_tangency}
We say a $(\delta,t)$-rectangle $R$ is $\Gl$-tangent to the function $f\in\mathcal{F}$ (or $f$ is $\lambda$-tangent to $R$) if $R\sub f^{\Gl\delta}$. In most cases we take $\lambda=5$, and we will simply refer to this as $R$ is tangent to $f$.
\end{defn}

\begin{defn}[Separation]\label{defn_bipartite}
Let $\mathcal W,\mathcal B\subset\mathcal{F}$. For $t>0$, we say $(\mathcal W,\mathcal B)$ is $t$-separated if $\Vert w-b\Vert \geq t$ for all $w\in \mathcal W$, $b\in \mathcal B$.
\end{defn}

%%%%%%%%%
%%%%%%%%%
%%%%%%%%%

\subsection{Preliminary reductions}
We start with the following lemma, which is similar to \cite[Lemma 3.1]{KOV2021}.

\begin{lem}\label{lem_KOV_3.1}
If $I\sub J$ is an interval of length $|I|\le (6K)^{-1}$, $f,g\in \mathcal F$, and $k$ is either of the functions $f-g$ or $f'-g'$, then either of the following alternatives hold (depending on the choice of $I$ and $k$):
\begin{enumerate}
    \item [(S)] $|k(\theta)|<(3K)^{-1}\Vert f-g\Vert$ for all $\theta\in I$.
    \item [(L)] $|k(\theta)|\ge (6K)^{-1} \Vert f-g\Vert$ for all $\theta\in I$.
\end{enumerate}
Moreover, if for $I$ the alternative (S) holds for both $k=f-g$ and $k=f'-g'$, then the alternative (L) must hold for $k=f''-g''$. 
\end{lem}
\begin{proof}
Let $\overline B$ be the closed unit ball of $C^2(J)$. Then if $|\theta-\theta'|<(6K)^{-1}$ and $h\in \overline B$, we have $|h(\theta)-h(\theta')|<(6K)^{-1}$ and $|h'(\theta)-h'(\theta')|<(6K)^{-1}$.

Next, fix $I\sub J$ with $|I|<(6K)^{-1}$, and suppose that the alternative (L) fails for $k=f-g$. So there is $\theta_0\in I$ such that $|k(\theta_0)|< (6K)^{-1}\norm{f-g}$. Then, if $\theta\in I$ is arbitrary, we have $|\theta-\theta_0|\le (6K)^{-1}$, and so if we define $l=k/\norm{f-g}\in \overline B(0,1)$, then
\begin{align*}
    \frac{|k(\theta)|}{\norm{f-g}}\le \frac{|k(\theta_0)|}{\norm{f-g}}+\left|\frac{k(\theta)}{\norm{f-g}}-\frac{k(\theta_0)}{\norm{f-g}}\right|<\frac 1{6K}+|l(\theta)-l(\theta_0)|<\frac 1{3K}.
\end{align*}
Thus the alternative (S) holds for $k$. The proof of the case $k=f'-g'$ is the same. The last assertion follows from the cinematic curvature assumption \eqref{cinematicFunctionCondition}.
\end{proof}
To prove Theorem \ref{thm_wolffAnalogue_maximal_general}, it suffices to prove \eqref{eqn_L32Bd_general} for each individual subinterval $I$ of length $\le (6K)^{-1}$, since the restriction of the functions in $\mathcal{F}$ to each such sub-interval $I$ is still family of cinematic functions with cinematic constant $K$ and doubling constant $D$. By an abuse of notation, we continue to write $I=J$, and assume without loss of generality that $|J|\le (6K)^{-1}$.

%Considering each sub-interval in turn, we may assume without loss of generality that for each pair of functions $f,g\in\mathcal{F}$, either (S) or (L) hold on the entire interval $J$, when $k$ is one of the functions $f-g$, $f'-g'$ or $f''-g''$. A last important observation is that since $\mathcal{F}$ is a family of cinematic functions, the alternative (L) holds for at least one of the functions $f-g$, $f'-g'$ or $f''-g''$.

The following lemma is an analogue of \cite[Lemma 3.2]{KOV2021}.
\begin{lem}\label{lem_KOV_3.2}
For distinct $f,g\in \mathcal F$, the map $\theta\mapsto f(\theta)-g(\theta)$ has at most two zeros in $J$. Moreover, if $\theta\mapsto f'(\theta)-g'(\theta)$ has two zeros in $J$, then the alternative (L) holds for $\theta\mapsto f(\theta)-g(\theta)$.
\end{lem}
\begin{proof}
We start with the second claim. Assume $\theta\mapsto f'(\theta)-g'(\theta)$ has two zeros in $J$. This implies, by Rolle's theorem, that $\theta\mapsto f''(\theta)-g''(\theta)$ has a zero in $J$, so in particular the alternative (L) cannot hold for $\theta\mapsto f''(\theta)-g''(\theta)$. Now Lemma \ref{lem_KOV_3.1} implies that the alternative (S) holds for $J$ and $\theta\mapsto f'(\theta)-g'(\theta)$. Thus by the last sentence of Lemma \ref{lem_KOV_3.1}, we must have that (L) holds for $\theta\mapsto f(\theta)-g(\theta)$.

The first claim follows from the second one: if $\theta\mapsto f(\theta)-g(\theta)$ had three zeros in $J$, then $\theta\mapsto f'(\theta)-g'(\theta)$ would have two zeros in $J$ again by the Rolle's theorem. But then, by the second claim, $\theta\mapsto f(\theta)-g(\theta)$ satisfies the alternative (L) in $J$, and hence cannot have zeros in $J$.
\end{proof}

%%%%%%%%%
%%%%%%%%%
%%%%%%%%%

\subsection{The tangency parameter $\Delta$}
Next, we shall define a quantity that measures how close two functions $f,g\in\mathcal{F}$ are to being tangent. Our definition is similar to the quantity (3.1) from \cite{KOV2021}, which in turn was inspired by \cite{KolasaWolff, Wolff2000}.
\begin{defn}
For $f,g\in \mathcal{F}$, we define
\begin{equation}
    \Delta(f,g)=\min_{\theta\in J/2}|f(\theta)-g(\theta)|+|f'(\theta)-g'(\theta)|,
\end{equation}
where $J/2$ is the interval of length $|J|/2$ with the same midpoint as $J$.
\end{defn}
In particular, we have
\begin{equation}\label{eqn_delta_bounded_by_d}
    \Delta(f,g)\leq \Vert f-g\Vert.
\end{equation}
$\Delta(f,g)$ measures the tangency between the graphs of functions $f,g$ over the interval $J/2$. If $\Delta(f,g)=0$, then $f$ and $g$ are tangent at some point $\theta\in J/2$. Note that $\Delta(f,g)$ is a pseudo-metric, since it is symmetric and satisfies the triangle inequality. The restriction to $J/2$ but not $J$ is needed for technical reasons, as is also the case of \cite{KolasaWolff} and \cite{KOV2021}.

The next proposition is a close cousin of \cite[Lemma 3.3]{KolasaWolff}. It describes the key geometric information about cinematic functions.

\begin{lem}\label{prop_two_zeros}
Let $f,g\in \mathcal F$, and define $h=f-g$, $t= \Vert f-g\Vert$, and $\Delta=\Delta(f,g)$. %All implicit constants only depend on $K$.
\begin{enumerate}
    \item\label{item_1} There exists $c_1=c_1(K)>0$ so that if $\Delta\le c_1t$, then the following are true.
         \begin{enumerate}
         \item\label{item_1a} There is a unique $\theta_0\in \frac {3} 5 J$ such that $h'(\theta_0)=0$. We also have $|h(\theta_0)|\lesssim \Delta$.

         \item \label{item_1b}We have $|h(\theta)|\gtrsim t>0$ for every $\theta\in J\bs (\frac {4} 5 J)$. If in addition $h$ changes sign in $J$, then $h$ has exactly 2 zeros $\theta_1<\theta_2$, both within $\frac {4} 5 J$, such that $|\theta_i-\theta_0|\lesssim \sqrt{\Delta/t}$, $i=1,2$.
\end{enumerate}

    \item \label{item_2}For each $0<\delta\le t/(6K)$, the set
    \begin{equation}\label{eqn_defn_E_delta}
        E_\delta=\{\theta\in J/4:|h(\theta)|\le\delta\}
    \end{equation}
    satisfies the following:
    \begin{enumerate}
       \item \label{item_2a} $E_\delta$ is contained in an interval of length $O(\sqrt{(\Delta+\delta)/t})$.
       \item \label{item_2b} $E_\delta$ is either a closed interval of length $O(\delta/\sqrt{(\Delta+\delta)t})$ or the union of two closed intervals of length $O(\delta/\sqrt{(\Delta+\delta)t})$.

        \item \label{item_2c}If in addition there is some $\tilde\theta\in J/4$ such that $|h(\tilde\theta)|\le \delta/2$, and $I$ is the interval constituting $E_\delta$ that contains $\tilde \theta$, then $|I|\gtrsim \delta/\sqrt{(\Delta+\delta)t}$.
    \end{enumerate}

\end{enumerate}
\end{lem}

\begin{proof}
Let $\theta_\Delta\in J/2$ be such that
\begin{equation*}
    \Delta=|h(\theta_\Delta)|+|h'(\theta_\Delta)|.
\end{equation*}
Thus $|h(\theta_\Delta)|\le \Delta$ and $|h'(\theta_\Delta)|\le \Delta$.

We first prove Part \eqref{item_1}. If we choose $c_1<(3K)^{-1}$, then using the definition of $\Delta(f,g)$, we have the alternative (S) holds for both $h,h'$. Thus the alternative (L) must hold for $h''$. Without loss of generality we may assume $h''(\theta)\gtrsim t$.

For Part \eqref{item_1a}, using $h''(\theta)\gtrsim t$, if $c_1$ is chosen to be small enough, then there is a unique $\theta_0\in \big(\frac 1 2+O(\Delta/t)\big)J\sub \frac{3} 5 J$ with $h'(\theta_0)=0$. Also, $|\theta_0-\theta_\Delta|\lesssim \Delta/t$, and thus
\begin{align*}
    |h(\theta_0)-h(\theta_\Delta)|\lesssim t|\theta_0-\theta_\Delta|\lesssim \Delta.
\end{align*}
Since $|h(\theta_\Delta)|\leq \Delta$, we also have $|h(\theta_0)|\lesssim \Delta$.

For Part \eqref{item_1b}, write $J=[a,b]$. Using $|h(\theta_0)|\lesssim \Delta$, for any $\theta\in [\theta_0,b]$ we have
\begin{align}
    |h(\theta)|
    &\ge \left|\int_{\theta_0}^{\theta} h'(s)ds\right|-C\Delta=\int_{\theta_0}^{\theta} \left|h'(s)\right|ds-C\Delta\nonumber\\
    &=\int_{\theta_0}^\theta \int_{\theta_0}^s h''(u)duds-C\Delta\geq ct(\theta-\theta_0)^2-C\Delta.\label{eqn_f_theta}
\end{align}
Here $C,c$ are absolute constants. The same is true for $\theta\in [a,\theta_0]$. If $c_1$ is small enough, we then have $|h(\theta)|\gtrsim t$ for every $\theta\in J\bs (\frac 4 5 J)$, since $\theta_0\in \frac 3 5 J$ and $|J|\sim 1$. The statements about the zeros follow from the intermediate value theorem and a similar application of \eqref{eqn_f_theta}.

Now we come to Part \eqref{item_2}. First, for Part \eqref{item_2a}, if $\Delta>c_1t $ then the result is trivial. If $\Delta\le c_1t$, then the argument of Part \eqref{item_1} applies. Let $\alpha=\inf E_\delta$ and $\beta=\sup E_\delta$, and hence $E_\delta\sub [\Ga,\beta]$. Using \eqref{eqn_f_theta} we have $\beta-\theta_0\le C\sqrt{(\Delta+\delta)/t}$ and $\theta_0-\Ga\le C\sqrt{(\Delta+\delta)/t}$. Thus the result follows.

For Part \eqref{item_2b}, we first note that the alternative (L) cannot hold for $h$, otherwise it will contradict the assumption $\delta<t/(6K)$. Thus (L) holds for $h'$ or $h''$. In either case, we see $E_\delta$ is a closed interval or the union of two closed intervals.

For the length estimate, we first consider the easier case $\Delta>c_1t$. In this case, if $\delta>c_1t/2$, then the result is trivial since then $\delta\sim \Delta\sim t$. If $\delta\le c_1t/2$, then for $\theta\in E_\delta$ we have
\begin{equation*}
    c_1t<\Delta\le \min_{\theta\in E_\delta}\big(|h(\theta)|+|h'(\theta)|\big)\le \delta+\min_{\theta\in E_\delta}|h'(\theta)|,
\end{equation*}
and thus $|h'(\theta)|\gtrsim t$ on $E_\delta$, and thus each interval constituting $E_\delta$ has length $\lesssim\delta/t=\delta/\sqrt{(\Delta+\delta)t}$.

For the harder case $\Delta\le c_1t$, the argument of Part \eqref{item_1} is applicable. We also have two subcases. First, if $\delta\ge c_1\Delta$, then the result follows from Part \eqref{item_2a}, since then $\sqrt{(\Delta+\delta)/t}\sim \delta/\sqrt{(\Delta+\delta)t}$.

If $\delta<c_1\Delta$, then we may further assume $\theta_0\in J/2$. Otherwise, using \eqref{eqn_f_theta} we have $h(\theta)\gtrsim t$ for every $\theta\in J/4$ if $c_1$ is small enough. In particular, $E_\delta=\varnothing$ and we have nothing to prove. Thus we have $\theta_0\in J/2$.

But by the definition of $\Delta$ and $\theta_0$, we have $|h(\theta_0)|=|h(\theta_0)|+|h'(\theta_0)|\ge \Delta$. Thus for every $\theta\in E_\delta\cap [\theta_0,b]$, we have
\begin{align*}
    \delta
    &\ge |h(\theta)|\ge \Delta-\int_{\theta_0}^\theta |h'(s)|ds\\
    &\ge \Delta-\int_{\theta_0}^\theta\int_{\theta_0}^s |h''(u)|duds\ge \Delta-C't (\theta-\theta_0)^2.
\end{align*}
Here $C'$ is an absolute constant. For $\theta\in E_\delta\cap [a,\theta_0]$ the argument is the same. This implies that
\begin{equation}
    |\theta-\theta_0|\gtrsim \sqrt {\Delta/t},\quad \forall \theta\in E_\delta.
\end{equation}
Since $h''(\theta)\gtrsim t$ in $J$, this further implies that for $\theta\in E_\delta$
\begin{equation}
    |h'(\theta)|=|h'(\theta)-h'(\theta_0)|\gtrsim t|\theta-\theta_0|\gtrsim \sqrt{t\Delta}.
\end{equation}
Thus each interval constituting $E_\delta$ has length $O(\delta/\sqrt{t\Delta})=O(\delta/\sqrt{(\Delta+\delta)t})$.

Now we come to Part \eqref{item_2c}. Let $I$ be the interval constituting $E_\delta$ that contains $\tilde\theta$. If $I=J/4$ then the result is trivial, so we assume $I\ne J/4$. In this case, there exists an endpoint $\theta_I$ of $I$ so that $|h(\theta_I)|=\delta$.

If $\Delta>c_1 t$, then we have
\begin{equation*}
    \delta/2\le |h(\theta_I)|-|h(\tilde \theta)|\lesssim |\theta_I-\tilde \theta|t,
\end{equation*}
and so $|\theta_I-\tilde \theta|\gtrsim \delta/t\sim \delta/\sqrt{(\Delta+\delta)t}$.

If $\Delta\le c_1 t$, then the argument of Part \eqref{item_1} is applicable. For every $\theta\in E_\delta$, the argument of Part \eqref{item_2a} implies that $|\theta-\theta_0|\lesssim \sqrt{(\Delta+\delta)/t}$. Thus
$$
|h'(\theta)|=|h'(\theta)-h'(\theta_0)|\lesssim t|\theta-\theta_0|\lesssim \sqrt{(\Delta+\delta)t},\quad \forall \theta\in E_\delta.
$$
Hence
\begin{equation*}
    \delta/2\le |h(\theta_I)|-|h(\tilde \theta)|\lesssim |\theta_I-\tilde \theta|\sqrt{(\Delta+\delta)t},
\end{equation*}
which implies that $|\theta_I-\tilde \theta|\gtrsim \delta/\sqrt{(\Delta+\delta)t}$. In particular, $|I|\gtrsim \delta/\sqrt{(\Delta+\delta)t}$.
\end{proof}

%%%%%%%%%
%%%%%%%%%
%%%%%%%%%

\subsection{Further geometric estimates}
We conclude this section with some simple arguments that relate the quantities $\Vert f-g\Vert$, $\Delta(f,g)$, the volume of $f^\delta\cap g^\delta$, and the number of tangency rectangles associated to the intersection. %Unless otherwise specified, all implicit constants below only depend on $K$.

\begin{lem}\label{prop_tangent1}
Let $f,g\in \mathcal F$ and let $t>0$. If there is a $(\delta,t)$-rectangle $R$ over the interval $J/4$ (i.e. the interval $I$ from Definition \ref{defn_delta_t_rectangle} is contained in $J/4$) that is tangent to both $f$ and $g$, then $(\Delta(f,g)+\delta)\Vert f-g\Vert\lesssim\delta t$.
\end{lem}

\begin{proof}
Let $R$ be a $(\delta,t)$-rectangle over an interval $I\sub J/4$. By Definition \ref{defn_tangency} and the triangle inequality, for each $\theta\in I$ we have $|f(\theta)-g(\theta)|\leq 10\lambda \delta$. Since $I\sub J/4$, by Part \eqref{item_2b} of Lemma \ref{prop_two_zeros}, we have
\[
\sqrt{\delta/t}=|I|\lesssim\frac{\delta}{\sqrt{(\Delta(f,g)+\delta)\Vert f-g\Vert}}.\qedhere
\]
\end{proof}

%For an interval $I$ and $\lambda>0$, we denote by $\lambda I$ the dilation of $I$ by $\lambda$ with respect to the centre of $I$.
\begin{lem}\label{prop_roughly_const}
Let $I\sub J/4$ be an interval, let $\lambda\geq 1$, let $f,g\in \mathcal F,$ and let $0<\delta\le \Vert f-g\Vert/(6K)$. If $|f(\theta)-g(\theta)|\leq\delta$ on $I$, then $|f(\theta)-g(\theta)|\lesssim \lambda^2\delta$ on $\lambda I\cap J$.
\end{lem}
\begin{proof}
Let $k = f-g$, let $t=\Vert f-g\Vert$, and let $\Delta=\Delta(f,g)$. Since $I\sub J/4$, by Part \eqref{item_2b} of Lemma \ref{prop_two_zeros} we have $|I|\lesssim \delta /\sqrt{(\Delta+\delta)t}\le \sqrt{\delta/t}$.

We will first show that $|k'(\theta)|\le C\lambda \delta|I|^{-1}$ on $\lambda I\cap J$, for some $C\lesssim 1$. Suppose not. Then there is some $\theta_1\in \lambda I\cap J$ with $|k'(\theta_1)|>C\lambda \delta|I|^{-1}$. We have $|k''(\theta)|\leq t$ in $J$. Since $|I|\lesssim \sqrt{\delta/t}$, for any $\theta\in \lambda I\cap J$ we have
$$
|k'(\theta)|\geq C\lambda \delta|I|^{-1}-\lambda t|I|\ge C\lambda \delta|I|^{-1}/2,
$$
provided $C\sim 1$ is chosen appropriately. But since $k\in C^2(I)$ and $\lambda>1$, if $C>4$ then it is impossible for $|k'(\theta)|\ge C\lambda \delta|I|^{-1}/2$ and $|k(\theta)|\leq\delta$ to simultaneously hold for all $\theta\in I$. Hence $|k'(\theta)|\lesssim \lambda \delta|I|^{-1}$ on $\lambda I\cap J$. By the triangle inequality, we thus have $|k(\theta)|\lesssim \lambda^2\delta$ on $\lambda I\cap J$.
\end{proof}

\begin{defn}
For $R=f^\delta(I)$ where $I\sub J$, we define the $\lambda$-dilation of $R$ to be
\begin{equation}\label{eqn_dilation_rectangle}
    \lambda R=f^{\lambda\delta}(\lambda I\cap J).
\end{equation}
\end{defn}

We will record several technical lemmas related to comparability.
\begin{lem}\label{prop_comparable_rectangle}
Let $f,g\in \mathcal F$ and $\lambda\ge 1$. Let $R=f^\delta(I)$ and $R'=g^\delta(I')$ be two $\lambda$-comparable $(\delta,t)$-rectangles over $J/4$. Let $I''$ be the convex hull of $I\cup I'$. Then $|I''|\leq  \sqrt{\lambda\delta/t}$, and
\begin{equation}\label{fVsg}
|f(\theta)-g(\theta)|\lesssim \lambda^3\delta \quad \textrm{for all}\ \theta\in I''.
\end{equation}
In particular, $R'\sub C\lambda R$ for some $C\sim 1$.%, where $C\lambda R$ is as defined in \eqref{eqn_dilation_rectangle}.
\end{lem}

\begin{proof}
By definition of comparability, there is some $(\lambda\delta,t)$-rectangle $R''=h^{\lambda\delta}(I_0)$ such that $R\cup R'\sub R''$ and $|I_0|= \sqrt{\lambda\delta/t}$. Since $I''\subset I_0$, we have $|I''|\leq \sqrt{\lambda\delta/t}$.

It remains to prove \eqref{fVsg}. If $\Vert f-g\Vert<12K\lambda \delta$ then \eqref{fVsg} is immediate and we are done.

Next, suppose $\Vert f-g\Vert\geq 12K\lambda\delta$. We have
$$
|f(\theta)-h(\theta)|\leq \lambda\delta \quad \textrm{for all}\ \theta\in I.
$$
Using Lemma \ref{prop_roughly_const} (with $2\lambda\delta$ in place of $\delta$; note that  $2\lambda\delta\leq \Vert f-g\Vert/(6K)$), we have
\[
|f(\theta)-h(\theta)|\lesssim \lambda^3\delta \quad \textrm{for all}\ \theta\in I''.
\]
The same is true with $f$ replaced by $g$. Then the result follows from the triangle inequality.
\end{proof}

%We record a corollary of Lemma \ref{prop_comparable_rectangle}.
\begin{cor}\label{cor_comparable_rectangle}%$\phantom{1}$
%\begin{enumerate}
    %\item
    Let $R_1,R_2,R_3$ be $(\delta,t)$ rectangles over $J/4$. Suppose that $R_1$ is $\lambda_1$-comparable to $R_2$ and $R_2$ is $\lambda_2$-comparable to $R_3$. Then $R_1$ is $\lambda_3$-comparable to $R_3$, for some $\lambda_3$ depending on $K$, $\lambda_1,$ and $\lambda_2$.
    %\item Let $R_1=f^\delta(I)$ and $R_2=g^\delta(I')$ be $\lambda$-comparable $(\delta,t)$-rectangles over $J/4$. Then $R_1$ is $\lambda'$-tangent to $g$ and $R_2$ is $\lambda'$-tangent to $f$, where $\lambda'\sim_K \lambda$.
    %\item Let $R_1$ and $R_2$ be $(\delta,t)$-rectangles over $J/4$, and suppose $R_1$ is $\lambda_1$-comparable to $R_2$. Let $f\in\mathcal{F}$ be $\lambda_2$-tangent to $R_1$. Then $f$ is $\lambda_3$-tangent to $R_2$, for some $\lambda_3$ depending on $K$, $\lambda_1,$ and $\lambda_2$.
%\end{enumerate}

\end{cor}

\begin{lem}\label{prop_incomparable}
Let $0<c<1$. If $C$ is large enough depending on $c$ and $K$, then the following holds. Let $f,g\in \mathcal F$ and $R=f^\delta(I)$ and $R'=g^\delta(I')$ be two $C$-incomparable $(\delta,t)$-rectangles over subintervals of $J/4$. For $0<c<1$, let $\tilde R=f^\delta(\tilde I)$ and $\tilde R'=g^\delta(\tilde I')$ where $\tilde I=cI$, $\tilde I'=cI'$. Then $\tilde R$ and $\tilde R'$ are $100$-incomparable.
\end{lem}

\begin{proof}
Suppose towards contradiction that $\tilde R$ and $\tilde R'$ are $100$-comparable. Then
$$
\dist(c_I,c_{I'})\leq  \sqrt{100c\delta/t},
$$
and thus $\diam (I\cup I')\lesssim \sqrt{\delta/t}$. Also, we have $|f(\theta)-g(\theta)|\le 100\delta$ for $\theta\in \tilde I$, and thus by Lemma \ref{prop_comparable_rectangle}, we have $|f(\theta)-g(\theta)|\lesssim \delta$ for $\theta\in I\cup I'$. Thus $R,R'$ are $C$-comparable for a sufficiently large $C$, from which a contradiction arises.
\end{proof}

\begin{lem}\label{prop_number_R=O(1)}
Let $\mathcal R$ be a family of pairwise $100$-incomparable $(\delta,t)$-rectangles $R=f^\delta(I)$ such that all $f\in\mathcal{F}$ are contained in a ball $B\subset C^2(I)$ of radius $3t$. Suppose on the other hand that all $R\in \mathcal R$ are contained in a $(\lambda\delta,t)$-rectangle $\tilde R$ where $\lambda\ge 100$. Then $\#\mathcal R\lesssim \lambda^{5/2}$.
\end{lem}
\begin{proof}
We will show that every point in $\R^2$ can be contained in $\lesssim \lambda$ many rectangles $R/2$ where $R\in \mathcal R$. (Here $R/2:=f^\delta(I/2)$.) Then the lemma follows from the simple inequality
$$
(\#\mathcal R)\delta^{3/2}t^{-1/2}\lesssim \int_{\tilde R}\sum_{R\in \mathcal R}{\bf 1}_{R/2}\lesssim \lambda|\tilde R|\lesssim \lambda^{5/2} \delta^{3/2}t^{-1/2}.
$$
Suppose there is some $(\theta_0,y_0)\in \R^2$ that is contained in $N$ pairwise $100$-incomparable $(\delta,t)$-rectangles of the form $R_i/2$, $i=1,\ldots,N$. We want to show $N\lesssim \lambda$.

It suffices to show that for every pair $1\le i\ne j\le N$,
\begin{align}
    &|f'_i(\theta_0)-f'_j(\theta_0)|\le 10\lambda \sqrt{\delta t},\label{eqn_upper_bound_Wolff_Lem1.2}\\
    &|f'_i(\theta_0)-f'_j(\theta_0)|>  \sqrt{\delta t}.\label{eqn_lower_bound_Wolff_Lem1.2}
\end{align}
Suppose $(\theta_0,y_0)\in (R_i/2)\cap (R_j/2)$. In particular, $\theta_0\in (I_i/2)\cap (I_j/2)$.

Denote $h=h_{i,j}=f_i-f_j$ and $I=I_i\cap I_j$. Then $\theta_0\in I$, $\sqrt{\delta/t}\le |I|\le 2\sqrt{\delta/t}$ and $|h(\theta_0)|\le \delta$.

For \eqref{eqn_upper_bound_Wolff_Lem1.2}, suppose $|h'(\theta_0)|>10\lambda  \sqrt{\delta t}$. Then since $\Vert f_i-f_j\Vert\le 6t$, we have in particular $\norm{h''}_\infty\leq 6t$, and thus for every $\theta\in I$, we have
$$
|h'(\theta)|>10\lambda \sqrt{\delta t}-12t\sqrt{\delta/t}\geq 5\lambda \sqrt{\delta t},
$$
since $\lambda \ge 100$. Since $|h(\theta_0)|\le \delta$, this implies that for some $\theta\in I$ we have
$$
|h(\theta)|\geq 5\lambda \sqrt{\delta t}\sqrt{\delta/ t}/2-\delta>\lambda \delta,
$$
contradicting the assumption that all rectangles $R_i\sub \tilde R$. This establishes \eqref{eqn_upper_bound_Wolff_Lem1.2}.

For \eqref{eqn_lower_bound_Wolff_Lem1.2}, we argue similarly. Suppose that $|h'(\theta_0)|\le \sqrt{\delta t}$. Then using $\norm{h''}_\infty\leq 6t$, we have for every $\theta'\in I$
$$
|h'(\theta)|\le \sqrt{\delta t}+12t\sqrt{\delta/t}\leq 13\sqrt{\delta t}.
$$
Thus for every $\theta\in I$ we have
$$
|h(\theta)|\le |h(\theta_0)|+26\sqrt{\delta t}\sqrt{\delta/t}\leq 100\delta,
$$
contradicting the assumption that $R_i,R_j$ are $100$-incomparable. This establishes \eqref{eqn_lower_bound_Wolff_Lem1.2}.
\end{proof}

\begin{lem}\label{prop_C_s_comparable}
Let $\mathcal R$ be a finite family of pairwise $100$-incomparable $(\delta,t)$-rectangles $R=f^\delta(I)$ such that all $f\in \mathcal{F}$ are contained in a ball of radius $3t$. If $C\ge 100$, then there is a $C$-incomparable subcollection $\mathcal R'\sub \mathcal R$ such that $\#\mathcal R'\sim \#\mathcal R$.
\end{lem}
\begin{proof}
Start with $R_1\in \mathcal R$. Choose a maximally $C$-incomparable subcollection $\mathcal R'$ containing $R_1$. For each $R\in \mathcal R'$, let $S(R)$ be the set of rectangles in $\mathcal R$ which are $C$-comparable to $R$. Then $\mathcal R= \cup_{R\in \mathcal R'}S(R)$ by maximality. To prove that $\# \mathcal R\sim \# \mathcal R'$, it suffices to show that $\#S(R)\lesssim 1$ for every $R\in \mathcal R'$. This follows from Lemma \ref{prop_comparable_rectangle} with $\lambda=C$ and Lemma \ref{prop_number_R=O(1)} with $\lambda=C^3$.

%We select $\mathcal R'$ in a greedy way. Start with any $R_1\in \mathcal R$, and choose a maximally $C$-incomparable subcollection $\mathcal R_1$ containing $R_1$. If $\mathcal R_1=\mathcal R$ then we are done. Otherwise, start with any $R_2\in \mathcal R\backslash \mathcal R_1$, and choose a maximally $\lambda$-incomparable subcollection $\mathcal R_2$ containing $R_2$. Note that each $R\in \mathcal R_2$ is $C$-comparable to every $R\in \mathcal R_1$. If $\mathcal R_1\cup\mathcal R_2=\mathcal R$, then either $\mathcal R_1$ or $\mathcal R_2$ will have cardinality at least $\#\mathcal R/2$.

%Continue in this fashion until it stops at some finite time $N$, which must occur since $\#\mathcal R$ is finite. Our goal is to show that $N=O(1)$, from which it follows that at least one of $\mathcal R_i$, $1\le i\le N$ has cardinality $\sim \#\mathcal R$. Call this $\mathcal R_i=\mathcal R'$ and we are done.

%Pick any $R_i=f_i^\delta(I_i)\in \mathcal R_i$ for each $1\le i\le N$. By construction, they must be mutually $C$-comparable. This implies that the $\cup_i I_i$ is contained in an interval $I$ of length at most $\sqrt{C\delta/t}$, and by Lemma \ref{prop_comparable_rectangle} there is an $(O(\delta),t)$-rectangle $R$ that contains all $R_i$. Applying Lemma \ref{prop_number_R=O(1)} with $\lambda=C$ gives $N=O(1)$, as required.
\end{proof}

We end this section with a coarse bound that will be used in some dyadic pigeonholing argument later.
\begin{lem}\label{prop_coarse_bound_R}
Let $\lambda\ge 100$. If $\mathcal R$ is a collection of pairwise $\lambda$-incomparable $(\delta,t)$-rectangles of the form $R=f^\delta(I)$ where $f\in \mathcal F$, then we have a coarse bound $\#\mathcal R\leq \delta^{-C}$ for some $C=C(D)$ for $\delta$ sufficiently small (depending on $K$).
\end{lem}
\begin{proof}
Let $R=f^\delta(I)$ and $R'=g^\delta(I')$ where $f,g\in \mathcal F$. It is easy to see that if the centres of $I,I'$ are no more than $\sqrt{\delta/t}$-separated and $\Vert f-g\Vert \le \delta$, then $R,R'$ must be $\lambda$-comparable. The result now follows from the fact that $\mathcal{F}$ has diameter $K$ and doubling constant $D$, and choosing $\delta\le K^{-1}$ sufficiently small.
\end{proof}

%%%%%%%%%%%%%%%%%%%%%%%%
%%%%%%%%%%%%%%%%%%%%%%%%
%%%%%%%%%%%%%%%%%%%%%%%%

\section{Lens counting and a bipartite curve-rectangle tangency bound}\label{lensCountingBoundSection}
In this section, we will use Marcus and Tardos' lens-counting theorem from \cite{MarcusTardos} to prove a bipartite version of the weak $\ell^{3/2}$ bound \eqref{boundNumberTangencyRectangles} from Section \ref{thmthm_wolffAnalogue_maximal_generalSketchSection}. The precise statement is as follows. We remark that all propositions in this section make no assumption on the doubling property in the definition of cinematic functions.

\begin{prop}\label{thm_number_estimate}
Let $\mathcal{F}$ be a family of cinematic functions with cinematic constant $K$. Then there is a constant $C=C(K)>0$ so that the following holds.

Let $A\geq 1$. Let $\delta,t>0$ with $0<\delta\leq At$. Let $\mathcal W,\mathcal B\subset \mathcal F$ be finite sets. Suppose $\mathcal{W}\cup\mathcal{B}$ has diameter at most $6t,$ and the pair $(\mathcal{W},\mathcal{B})$ is $t/A$-separated.

Let $\mathcal R$ be a collection of pairwise $100$-incomparable $(\delta,t)$-rectangles of the form $f^\delta(I)$ where $f\in \mathcal F$ and $I\sub J/4$. Suppose each rectangle in $\mathcal{R}$ is tangent to at least $\mu$ elements of $\mathcal W$ and at least $\nu$ elements of $\mathcal B$. Then
\begin{equation}\label{boundOnMuNuR}
\#\mathcal R\leq C A^C \left(\frac{\#\mathcal W}{\mu}+\frac{\#\mathcal B}{\nu}\right)^{3/2}\log \left(\frac{\#\mathcal W}{\mu}+\frac{\#\mathcal B}{\nu}\right).
\end{equation}
\end{prop}
%\noindent In what follows, all implicit constant in this section may depend on $K$.
%
%
%For the remainder of this section, we will fix the family $\mathcal{F}$ (and hence $K$), and we write $A\lesssim B$ to mean $A\leq CB$, where the constant $C$ may depend on $K$.
\subsection{Preliminary reductions}
\begin{lem}\label{lem_c_2}
Let $c>0$. Suppose there exists a constant $C_1 = C_1(K)>0$ so that Proposition \ref{thm_number_estimate} holds whenever $\delta\leq ct$. Then there is a constant $C = C(C_1, K, c)$ so that  Proposition \ref{thm_number_estimate} holds for all $\delta\leq At$.
\end{lem}
\begin{proof}
If $\delta\leq ct$ then we are done. Suppose instead that $\delta>ct$, and let $\mathcal{W}$, $\mathcal{B}$, $\mathcal{R}$, $\mu$, and $\nu$ be as in the statement of the Proposition. Let $C_1\sim 1$ be a constant to be specified below. Apply Lemma \ref{prop_C_s_comparable} to find a subcollection $\mathcal R'\sub \mathcal R$ that is pairwise $C_1$-incomparable, with $\#\mathcal R'\sim \#\mathcal R$. (By Lemma \ref{prop_coarse_bound_R} we may suppose that $\mathcal R$ is finite.) Next, shrink the length of each $R\in \mathcal R'$ to be $\sqrt{\delta/t'}$ where $t'=c^{-1}\delta>t$, and so $\delta=ct'$. Denote by $\mathcal R''$ the resulting collection of $(\delta,t')$-rectangles. If $C_1$ is chosen large enough, then using Lemma \ref{prop_incomparable}, we have that $\mathcal R''$ is $100$-incomparable. Applying Proposition \ref{thm_number_estimate} with $t'$ in place of $t$, we are done.
\end{proof}

%In view of Lemma \ref{lem_c_2}, we may now assume that $\Delta\le t/(6K)$, so in particular Part \eqref{item_2} of Lemma \ref{prop_two_zeros} holds.
Similar arguments yield the following.
\begin{lem}\label{lem_r=t}
Suppose there exists a constant $C_2 = C_2(K)>0$ so that Proposition \ref{thm_number_estimate} holds whenever $(\mathcal W,\mathcal B)$ is $2t$-separated. Then there is a constant $C = C(C_2, K)$ so that  Proposition \ref{thm_number_estimate} holds when $(\mathcal W,\mathcal B)$ is $t/A$-separated.
\end{lem}

% \begin{proof}
% If $A>1$, we shrink the height of each $(\Delta,t)$-rectangle $R=f^\Delta(I)$ to a $(r\Delta/t,r)$ rectangle $R'=f^{\frac{r\Delta}t}(I)$. It is easy to note that the resulting collection $\mathcal R'$ of $R'$ is $C/2$-incomparable. Also, using the weak Besicovitch covering property in the definition of cinematic functions, we may partition $B$ into finitely overlapping smaller balls $B'$ of radius $r$. Losing a factor of $O((t/r)^D)$, it suffices to prove the result for each $B'$ in place of $B$. Applying the case of Proposition \ref{thm_number_estimate} with $r\Delta/t$ in place of $\Delta$ and $r$ in place of $t$, we obtain the desired conclusion.
% \end{proof}

\begin{lem}
Suppose there exists a constant $C_3=C_3(K)$ so that Proposition \ref{thm_number_estimate} holds whenever $\mathcal R$ is $2C_c$-incomparable, where $C_c$ is as in Lemma \ref{prop_comparable} below. Then there is a constant $C=C(C_3,K)$ so that Proposition \ref{thm_number_estimate} holds when $\mathcal R$ is $100$-incomparable.
\end{lem}

From now on we will assume $0<\delta\le c_2t$ where $c_2=c_2(K)$ will be determined in Lemma \ref{prop_comparable} below. We also assume $(\mathcal W,\mathcal B)$ is $2t$-separated, and that $\mathcal R$ is $2C_c$-incomparable. % let $r=t$.

\medskip

\noindent \emph{Remark.} In the arguments below, we will choose a few constants depending on $K$. The order of dependence is as follows: $c_1$ as in Lemma \ref{prop_two_zeros};  $\lambda$ as in Lemma \ref{lem_random_1}; $c_2$ and $C_c$ as in Lemma \ref{prop_comparable}.

\subsection{Pseudo-circles and lenses}
Now we introduce some terminology from Marcus-Tardos \cite{MarcusTardos} that will be used in the proof of Proposition \ref{thm_number_estimate}.

\begin{defn}
Let $\mathcal C$ be a family of closed Jordan plane curves (i.e.~the homeomorphic image of $S^1$). We say $\mathcal C$ forms a family of pseudo-circles if for every pair of distinct $c,c'\in \mathcal C$, the following conditions are satisfied:
\begin{enumerate}
    \item $\#(c\cap c')\le 2$.
    \item If $p\in c\cap c'$, then $c$ and $c'$ intersect properly at $p$. That is, for all sufficiently small circles $C(p)$ centered at $p$, the points in $C(p)\cap c$ and $C(p)\cap c'$ appear alternatively cyclically counterclockwise. (See Figure \ref{figProperInt} below).
\end{enumerate}
\end{defn}

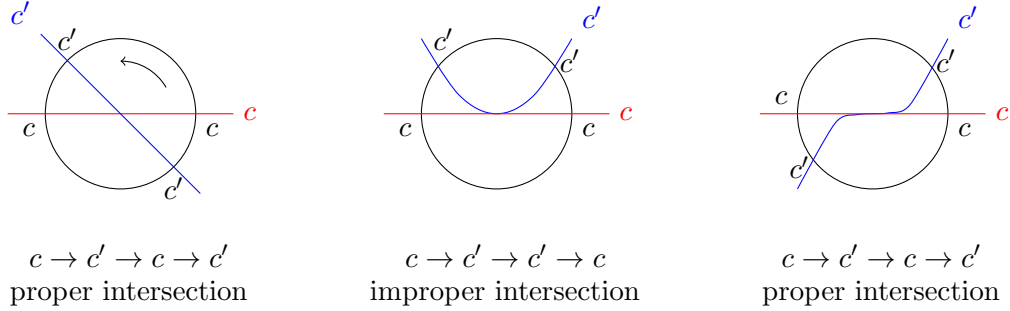
\begin{figure}
\centering
	\begin{tikzpicture}
		\def\r{1}
		%% case 1
		\draw (0,0) circle (\r);
		\draw [color=red] plot [smooth, tension=0.6] coordinates { (-1.5,0) (0,0) (1.5,0)} node[right] {$c$};
		\draw [color=blue] plot [smooth, tension=0.6] coordinates { (-1.5/1.414,1.5/1.414) (0,0) (1.5/1.414,-1.5/1.414)};
		\draw [color=blue] (-1.5/1.414,1.5/1.414) node [above left] {$c'$};
	
		\def\rs{0.7}
		\draw[->] (\rs*1.7321/2,\rs/2) arc [start angle = 30, end angle = 90, radius = \rs ];
	
		\draw (1,0) node[below right] {$c$};
		\draw (-1,0) node[below left] {$c$};
		\draw  (-1/1.414,1/1.414) node[above] {$c'$};
		\draw  (1/1.414,-1/1.414) node[below] {$c'$};
	
		\node at (0.5,-1.5,) {$c \rightarrow c' \rightarrow c \rightarrow c'$};
		\node at (0.5,-2,) {\text{proper intersection}};

		%% case 2
		\draw (5,0) circle (\r);
		\draw [color=blue] plot [smooth, tension=0.6] coordinates { (4,1) (4.5,0.25) (5,0) (5.5, 0.25) (6,1)} node[above right] {$c'$};
		\draw [color=red] plot [smooth, tension=0.6] coordinates { (3.5,0) (5,0)  (6.5,0)} node[right] {$c$};
	
		\draw (6,0) node[below right] {$c$};
		\draw (4,0) node[below left] {$c$};
		\draw  (5-1/1.414,1/1.414) node [above] {$c'$};
		\draw  (5 + 1/1.414,1/1.414) node[right] {$c'$};

		\node at (5.5,-1.5,) {$c \rightarrow c' \rightarrow c' \rightarrow c$};
		\node at (5.5,-2,) {\text{improper intersection}};
	
		%% case 3
		\draw (10,0) circle (\r);
		\draw [color=blue] plot [smooth, tension=0.5] coordinates { (9,-1) (9.5,-0.125)  (9.8,-0.01) (10.2,0.01)     (10.5,0.125) (11,1)} node[above right] {$c'$};
		\draw [color=red] plot [smooth, tension=0.6] coordinates { (8.5,0) (10,0) (11.5,0)} node[right] {$c$};
	
		\draw (11,0) node[below right] {$c$};
		\draw (9,0) node[above left] {$c$};
		\draw  (10+1/1.414,1/1.414) node[right] {$c'$};
		\draw  (10-1/1.414,-1/1.414) node[left] {$c'$};
	
		\node at (10.5,-1.5,) {$c \rightarrow c' \rightarrow c \rightarrow c'$};
		\node at (10.5,-2,) {\text{proper intersection}};
	\end{tikzpicture}
\caption{Proper and improper intersection.}
\label{figProperInt}
\end{figure}

%\begin{figure}[h]
%\centering
%\includegraphics[scale=0.4]{proper_intersection}
%\caption{Proper and improper intersection.}
%\label{figProperInt}
%\end{figure}

The first condition is key. The second condition was not explicitly stated in \cite{MarcusTardos}, so we included it here for completeness. It is a technical assumption included to make our settings fit rigorously to the discrete geometry scenario in \cite{MarcusTardos}. Of course, in the Kakeya type maximal inequality, if two functions $f,g$ are exactly tangent at some point, then the intersection may be improper. To avoid this issue, we will slightly perturb
the cinematic functions $f$ by $O(\delta)$, such that no exact tangency could occur. This will do no harm, since if $f$ and $g$ satisfy \eqref{cinematicFunctionCondition} and also $\Vert f-g\Vert \geq C\delta$ for a sufficiently large constant $C$, then $f+O(\delta)$ and $g$ will continue to satisfy \eqref{cinematicFunctionCondition} (with $K$ replaced by $2K$).  See Section \ref{sub_perturb} for the details of this perturbation process. %A last remark is that it only makes sense to say $\mathcal C$ is a family of pseudo-circles; it makes no sense to say a single function $f$ is a pseudo-circle.

%\todo{insert a picture}

Now we recall the definition of lenses from \cite{MarcusTardos}.
\begin{defn}
Let $\mathcal C$ be a set of pseudo-circles. Let $c,c'\in \mathcal C$ be two curves that intersect exactly twice, at the points $p,p'$. This divides $c,c'$ into four closed segments, that intersect only at their endpoints. We say $\ell\subset c\cup c'$ is a \emph{lens} if it is a simple closed curve that is the union of two of these segments.
\end{defn}
Note that there are exactly $\binom 4 2=6$ lenses formed by $c,c'$; two of these are the curves $c$ and $c'$ themselves. See \cite[Figure 1]{MarcusTardos}. As a technical aside, we only consider ``first-generation lenses", that is, we ignore lenses formed by lenses. Thus the total number of lenses determined by an arrangement $\mathcal C$ is at most $4\binom{\#\mathcal{C}}{2} + \#\mathcal{C}$.

\begin{defn}[Overlapping lenses]\label{defn_overlap_lense}
Let $\mathcal C$ be a family of pseudo-circles, and let $\ell,\ell'$ be two lenses formed by $\mathcal C$. We say $\ell,\ell'$ \emph{overlap} if they share a segment of positive length (see Figure \ref{figOverlappingLens}).
\end{defn}
Note that if two lenses intersect (properly) at only finitely many points, then they do not overlap.

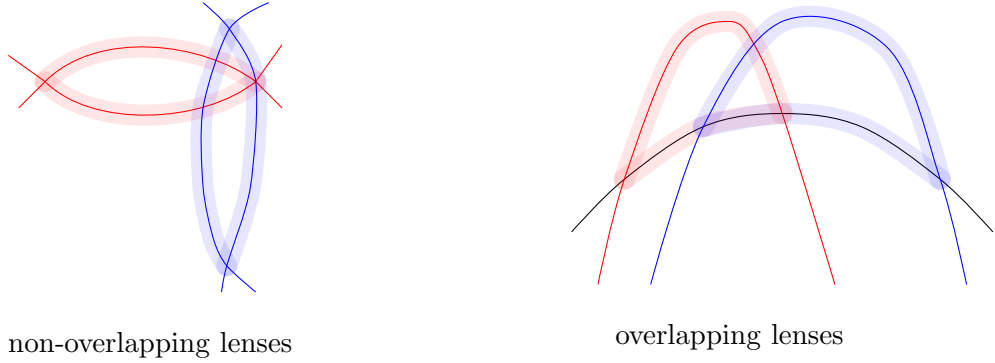
\begin{figure}
\centering
      \begin{tikzpicture}
		\begin{scope}[scale=0.7]
		% horizontal (no smoothing)	
		\draw [color=red] (-0.5,-0.5)
		.. controls (0,0) .. (0,0)
		.. controls (1,1) and (3,0.75) .. (4,0)
		.. controls (4.5,-0.5) .. (4.5,-0.5);
		
		\draw [line width=8,color=red,opacity=0.1, line cap=round]  (0,0)
		.. controls (1,1) and (3,0.75) .. (4,0);
		
		\draw [color=red] (-0.7,0.5)
		.. controls (0,0) .. (0,0)
		.. controls (0.9,-0.9) and (3,-0.8) .. (4,0)
		.. controls (4.5,0.7) .. (4.5,0.7);
	
		\draw [line width=8,color=red,opacity=0.1, line cap=round] (0,0)
		.. controls (0.9,-0.9) and (3,-0.8) .. (4,0);

		% vertical (with smoothing)
		\draw [color=blue] plot [smooth, tension=0.5] coordinates {(3,1.5) (3.5,1) (4,0) (3.9, -2) (3.45,-3.5) (3.35,-4)};
		\draw [color=blue, line width=8, opacity=0.1, line cap = round] plot [smooth, tension=0.5] coordinates { (3.5,1) (4,0) (3.9,-2) (3.45,-3.5) };
	
		\draw [color=blue] plot [smooth, tension=0.5] coordinates {(4.25,1.5) (3.5,1) (3,-0.5) (3, -2) (3.2,-3.0) (3.45,-3.5) (4,-4)};
		\draw [color=blue, line width=8, opacity=0.1, line cap=round] plot [smooth, tension=0.5] coordinates {(3.5,1) (3,-0.5) (3, -2) (3.2,-3.0) (3.45,-3.5)};
		
		\node at (2,-5) {$\text{non-overlapping lenses}$};
		\end{scope}
		
		\begin{scope}[shift={(7,-2)}, scale=0.7]	
		\draw [color=black] plot [smooth, tension = 0.6] coordinates {(0,0)  (1,1)  (2.5,2)  (4,2.25)  (5.5,2)  (7,1)  (8,0)};

		\draw [color=red] plot [smooth, tension = 0.6] coordinates {(0.5,-1) (1,1)   (2,3.5)  (3,4) (3.5, 3.5)  (4,2.25) (5,-1)};
		\draw [color=red, line width=8, opacity=0.1, line cap=round] plot [smooth, tension = 0.6] coordinates { (1,1)   (2,3.5)  (3,4) (3.5, 3.5)  (4,2.25) };
		\draw [color=red, line width=8, opacity=0.1, line cap=round] plot [smooth, tension = 0.6] coordinates {  (4,2.25) (2.5,2) (1,1) };

		\draw [color=blue] plot [smooth, tension = 0.6] coordinates {(1.5,-1) (2.5,2) (4,4) (6,3.5) (7,1) (7.5,-1) };
		\draw [color=blue, line width=8, opacity=0.1, line cap=round] plot [smooth, tension = 0.6] coordinates {(2.5,2) (4,4) (6,3.5) (7,1) };
		\draw [color=blue, line width=8, opacity=0.1, line cap=round] plot [smooth, tension = 0.6] coordinates{  (2.5,2)  (4,2.25)  (5.5,2)  (7,1)};
	
		\node at (3,-2) {$\text{overlapping lenses}$};
		\end{scope}
	
	\end{tikzpicture}
\caption{Overlapping and non-overlapping lenses.}
\label{figOverlappingLens}
\end{figure}

%\begin{figure}[h]
%\centering
%\includegraphics[scale=0.4]{overlapping_lens}
%\caption{Overlapping and non-overlapping lenses.}
%\label{figOverlappingLens}
%\end{figure}

We can now state Lemma 10 from \cite{MarcusTardos}.
\begin{thm}\label{thm_main_MT_bound}
Let $\mathcal{C}$ be a family of $n$ pseudo-circles. Then every set of non-overlapping lenses has cardinality $O(n^{3/2} \log n)$. %{\color{red}Apart from the logarithmic loss, the exponent $3/2$ is sharp.}
\end{thm}
%As a simple example, if $\#\mathcal C=2$, then a non-overlapping family of lenses formed by $\mathcal C$ has cardinality at most $2$, instead of the total number of lenses which is $6$.

\subsection{Lenses and pseudo-parabolas}\label{subsec_extension_pseudo_circle}
We return to our case of a finite family $F\subset \mathcal F$ of cinematic functions. By Lemma \ref{lem_KOV_3.2}, their graphs intersect at most twice, but they are not yet pseudo-circles because they are not closed curves. However, we can extend them to be pseudo-circles in the following way.

We need a mild technical assumption first. Denote $J=[\alpha,\beta]$. We require that both sets
$$
\{f(\alpha):f\in  F\}, \quad \{f(\beta):f\in  F\}
$$
consist of distinct numbers, respectively. We also need that a pair of different cinematic functions $f,g$ never intersect exactly tangentially at any point in $J$. As mentioned before, this can be guaranteed by a careful perturbation; see Section \ref{sub_perturb} for details. Also, we note that all curves are confined in $J\times [-M,M]$ for a finite $M$. We index the curves by $f_{j},1\le j\le \# F$, and rearrange so that $f_j(\alpha)$ is increasing. We first extend each $f_j(\theta)$ to be equal to $f_j(\alpha)$ for $\theta\in [\alpha-j,\alpha]$ and equal to $f_j(\beta)$ for $\theta\in [\beta,\beta+j]$. We then extend by adding the line segments $\{\alpha-j\}\times  [-M-j,f_j(\alpha)]$ and $\{\beta+j\}\times [-M-j,f_j(\beta)]$. Finally, close the loop by adding the line segment $[\alpha-j,\beta+j]\times \{-M-j\}$. In this way we have extended the graphs of the functions in $F$ to a family $\mathcal C$ of pseudo-circles, with all intersections proper. See Figure \ref{figExtension} below.

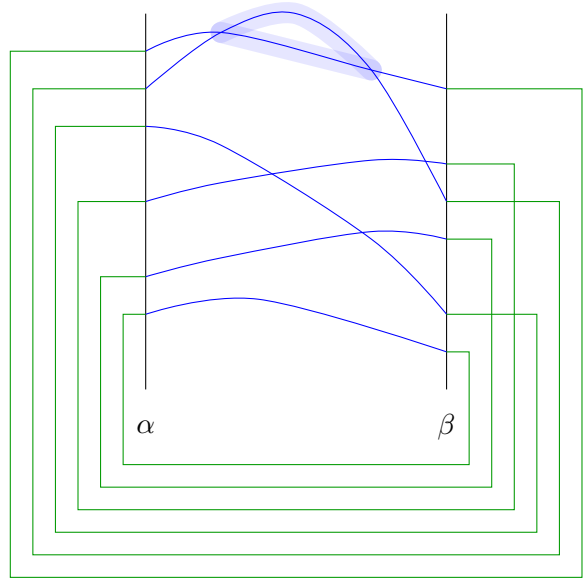
\begin{figure}
	\centering
	\begin{tikzpicture}
	
		\draw [color=black] plot (0,0) -- (0,5);
		\node at (0,-0.5) {$\alpha$};
	
		\draw [color=black] plot (4,0) -- (4,5);
		\node at (4,-0.5) {$\beta$};
	
		\draw [color=blue] plot [smooth, tension=0.6] coordinates { (0,1) (1.5,1.2) (4,0.5) };
		\draw [color=blue] plot [smooth, tension=0.6] coordinates { (0,1.5) (1,1.75) (3,2.1) (4,2) };
		\draw [color=blue] plot [smooth, tension=0.6] coordinates { (0,2.5) (1,2.75)  (3,3.05) (4,3) };
		\draw [color=blue] plot [smooth, tension=0.6] coordinates { (0,3.5) (1,3.25) (3,2)  (4,1) };
		\draw [color=blue] plot [smooth, tension=0.6] coordinates { (0,4) (1,4.75) (2,5) (3,4.25) (4,2.5) };
		\draw [color=blue] plot [smooth, tension=0.6] coordinates { (0,4.5) (1,4.75)  (3,4.25)  (4,4) };
	
		\draw [color=blue, line width=8, opacity=0.1, line cap=round] plot [smooth, tension=0.6] coordinates { (1,4.75) (2,5) (3,4.25)  };
		\draw [color=blue, line width=8, opacity=0.1, line cap=round] plot [smooth, tension=0.6] coordinates {  (1,4.75)  (3,4.25)  };
	
		\def\x{0.3}
		\draw [ color=darkgreen] (0,1) -- (-\x,1) -- (-\x,-1) -- (4+\x,-1) -- (4+\x,0.5) -- (4,0.5);
		\draw [ color=darkgreen] (0,1.5) -- (-2*\x,1.5) -- (-2*\x,-1-\x) -- (4+2*\x,-1-\x) -- (4+2*\x,2) -- (4,2);
		\draw [ color=darkgreen] (0,2.5) -- (-3*\x,2.5) -- (-3*\x,-1-2*\x) -- (4+3*\x,-1-2*\x) -- (4+3*\x,3) -- (4,3);
		\draw [ color=darkgreen] (0,3.5) -- (-4*\x,3.5) -- (-4*\x,-1-3*\x) -- (4+4*\x,-1-3*\x) -- (4+4*\x,1) -- (4,1);
		\draw [ color=darkgreen] (0,4) -- (-5*\x,4) -- (-5*\x,-1-4*\x) -- (4+5*\x,-1-4*\x) -- (4+5*\x,2.5) -- (4,2.5);
		\draw [ color=darkgreen] (0,4.5) -- (-6*\x,4.5) -- (-6*\x,-1-5*\x) -- (4+6*\x,-1-5*\x) -- (4+6*\x,4) -- (4,4);
	
	\end{tikzpicture}
	\caption{Extending the graph of $f \in F$ to a closed curve.}
	\label{figExtension}
\end{figure}

%\begin{figure}[h]
%\centering
%\includegraphics[scale=0.3]{extension}
%\caption{Extending the graph of $f\in F$ to a closed curve.}
%\label{figExtension}
%\end{figure}

We observe that if two functions $f,g$ intersect exactly twice in the strip $J\times \R$, then they give rise to exactly one lens $\ell$ that lies entirely in the strip $J\times \R$. See Figure \ref{figExtension}. Denote the collection of these lenses by $\mathcal L$. They are the only lenses we are interested in.

Now we explain the connection between overlapping lenses and curve tangencies.
\begin{lem}\label{prop_comparable}
If $\delta/t\le c_2\le c_1$ (where $c_1$ is as in Part \eqref{item_1} of Lemma \ref{prop_two_zeros}) for some small enough constant $c_2=c_2(K)$, then the following holds. Let $w,b_1,b_2\in \mathcal F$. For $i=1,2$, let $t_i=\Vert w-b_i\Vert$. Suppose $t_i\geq t$, and that $w$ and $b_i$ intersect exactly twice in $J$, thus forming a lens $\ell_i$ that lies over some subinterval of $J$. Let $R_i$ be a $(\delta,t)$-rectangle that is $2\lambda$-tangent to both $w$ and $b_i$, where $\lambda=\lambda(K)$ is to be determined in Lemma \ref{lem_random_1}.

Suppose that $\ell_1$ and $\ell_2$ overlap. Then $R_1$ and $R_2$ are $C_c$-comparable, where $C_c$ ($c$ stands for ``comparable") depends on $\lambda,K$.
\end{lem}
\begin{proof}
For $i=1,2$, define the interval $I_i=[\alpha_i,\beta_i]\subset J$ so that $R_i$ is a curvilinear rectangle over $I_i$. Since $R_1$, $R_2$ are both tangent to $w$, it suffices to show that $\diam(I_1\cup I_2)=O(\sqrt{\delta/t})$.

By Lemma \ref{prop_tangent1}, we have $\Delta(w,b_i)\lesssim \delta$. If $\delta/t$ is small enough, then $\Delta(w,b_i)\le c_1 t$. Thus there is a unique $\theta^{(i)}_0\in \frac 3 5 J$ such that $|\alpha_i-\theta^{(i)}_0|=O(\sqrt{\delta/t_i)}=O(\sqrt{\delta/t)}$ and $|\beta_i-\theta^{(i)}_0|=O(\sqrt{\delta/t})$. Since $(\alpha_1,\beta_1)\cap (\alpha_2,\beta_2)\neq \varnothing$, we have $\theta^{(2)}_0-\theta^{(1)}_0=O(\sqrt{\delta/t})$.

Now we study $I_i$. Using the triangle inequality, for $\theta\in I_i$ we have $|w(\theta)-b_i(\theta)|\lesssim \delta$. By Part \eqref{item_2a} of Lemma \ref{prop_two_zeros}, $I_i$ is contained in an interval of length $O(\sqrt{\delta/t})$ centered at $\theta^{(i)}_0$. But since $\theta^{(2)}_0-\theta^{(1)}_0=O(\sqrt{\delta/t})$, we see that $\diam(I_1\cup I_2)=O(\sqrt{\delta/t})$, as required.
\end{proof}

\subsection{Rectangle tangency bounds for properly intersecting curves}
We first prove a special case of Proposition \ref{thm_number_estimate}, where the curves intersect properly.
\begin{lem}\label{prop_number_estimate0}
Let $\mathcal{F}$ be a family of cinematic functions, with cinematic constant $K$. Then there is a constant $C=C(K)>0$ so that the following holds.

Let $\delta,t>0$ with $0<\delta\leq  c_2 t$ (here $c_2$ is the quantity from Lemma \ref{prop_comparable}).  Let $\mathcal W,\mathcal B\subset \mathcal F$ be a pair of $t$-separated finite sets. %(Note that we have no assumption on the diameter of $\mathcal W\cup\mathcal B$.)

Let $\mathcal R$ be a collection of pairwise $C_c$-incomparable $(\delta,t)$-rectangles of the form $f^\delta(I)$, where $f\in \mathcal F$, $I\sub J/4$ and $C_c$ is as in Lemma \ref{prop_comparable}. Suppose each rectangle $R\in \mathcal{R}$ is tangent to at least one $w_R\in \mathcal W$ and at least one $b_R\in\mathcal B$. Suppose furthermore that the following holds.

\begin{enumerate}
    \item[A1.]  For each $R\in \mathcal R$, we have that $w_R,b_R$ intersect exactly twice in $J$.
    \item[A2.] If $J=[\alpha,\beta]$, then the sets
$$
\{f(\alpha):f\in \mathcal W\cup \mathcal B\}, \quad \{f(\beta):f\in \mathcal W\cup \mathcal B\}
$$
respectively consist of distinct elements.
\item[A3.] For all $f\neq g\in \mathcal W\cup \mathcal B$, $f,g$ never intersect tangentially in $J$, that is, the two quantities $f(\theta)-g(\theta)$ and $f'(\theta)-g'(\theta)$ are never simultaneously $0$ for any $\theta\in J$.
\end{enumerate}

Then
\begin{equation}\label{boundOn11R}
\#\mathcal R\leq C \left(\#\mathcal W +\#\mathcal B \right)^{3/2}\log \left(\#\mathcal W +\#\mathcal B\right).
\end{equation}
\end{lem}

\begin{proof}
With the additional assumptions A1--A3, each $R\in \mathcal R$ is tangent to a pair $(w_R,b_R)$ intersecting exactly twice in $J$. After extending them to pseudo-circles in the way described in Section \ref{subsec_extension_pseudo_circle}, we see the pair $(w_R,b_R)$ gives rise to a lens $\ell_R$ in the collection $\mathcal L$.

By Lemma \ref{prop_comparable}, the lenses in $\mathcal L=\{\ell_R:R\in \mathcal R\}$ are non-overlapping. Since the mapping $R\mapsto \ell_R$ is clearly a bijection, using Theorem \ref{thm_main_MT_bound}, we conclude that
\[
\#\mathcal R=\#\mathcal L\lesssim (m+n)^{3/2}\log (m+n).\qedhere
\]
\end{proof}

\subsection{Perturbing the curves}\label{sub_perturb}
Our next step is to remove the additional assumptions A1--A3 in Lemma \ref{prop_number_estimate0} by perturbing the functions in $\mathcal W$ and $\mathcal B$. Namely, we would like to prove the following version of Proposition \ref{thm_number_estimate}.

\begin{lem}\label{prop_number_estimate1}
Let $\mathcal{F}$ be a family of cinematic functions, with cinematic constant $K$. Then there is a constant $C=C(K)>0$ so that the following holds.

Let $\delta,t>0$ with $0<\delta\leq  c_2 t/2$ (here $c_2$ is the quantity from Lemma \ref{prop_comparable}). Let $\mathcal W,\mathcal B\subset \mathcal F$ be a pair of $2t$-separated finite sets.

Let $\mathcal R$ be a collection of pairwise $2C_c$-incomparable $(\delta,t)$-rectangles of the form $f^\delta(I)$ where $f\in \mathcal F$ and $I\sub J/4$. Suppose each rectangle in $\mathcal{R}$ is tangent to at least one element of $\mathcal W$ and at least one element of $\mathcal B$.

Then
\begin{equation}\label{boundOn11RPerturbed}
\#\mathcal R\leq C \left(\#\mathcal W +\#\mathcal B \right)^{3/2}\log \left(\#\mathcal W +\#\mathcal B\right).
\end{equation}
\end{lem}
The reduction from Lemma \ref{prop_number_estimate0} to Lemma \ref{prop_number_estimate1} requires two perturbations. The first perturbation (which has size roughly $\delta$) ensures that property A1 holds for a considerable fraction of the tangency rectangles in $\mathcal{R}$. The second perturbation (which is infinitesimally small) ensures that properties A2 and A3 hold.

\begin{lem}\label{lem_random_1}
Let $0<\delta<c_1t$ and let $f,g\in\mathcal{F}$ with $\Vert f-g\Vert \geq 2t$. Suppose $f$ and $g$ are tangent to a common $(\delta,t)$-rectangle $R$ over $J/4$. If $\lambda=\lambda(K)\ge 1$ is chosen to be large enough, then there exists $\eta \in \{-1,0,1\}$ so that the function $h = f + \eta \lambda\delta$ is $2\lambda$-tangent to $R$, and the functions $g$ and $h$ intersect exactly twice in $J$.
\end{lem}

\begin{proof}
By Part \eqref{item_1b} of Lemma \ref{prop_two_zeros}, it suffices to show that for a suitable $\eta\in \{-1,0,1\}$, we have $k=h-g$ changes sign within $J$. We have three cases. If $f-g$ changes sign in $J$, then we just take $\eta=0$. Otherwise, we must have $f-g\ge 0$ on $J$ or $f-g\le 0$ on $J$. In the former case, we take $\eta=-1$, and so $k=f-\lambda\delta-g$. Also, by Lemma \ref{prop_tangent1}, we have $\Delta(f,g)\lesssim \delta$, and thus $\min_{\theta\in J}f(\theta_0)-g(\theta_0)\leq C' \delta$ for some $C'=C'(K)$, using Part \eqref{item_1a} of Lemma \ref{prop_two_zeros}. If we choose $\lambda>C'$, then we have $\min_{\theta\in J}k(\theta_0)<0$, and hence $k$ changes sign in $J$.

In the case $f-g\le 0$ on $J$, we take $\eta=1$ and the argument is similar.
\end{proof}

\begin{cor}\label{choiceOfShift}
Let $\delta,t>0$ with $0<\delta\leq  c_2 t/2$.  Let $\mathcal W,\mathcal B\subset \mathcal F$ be a pair of $2t$-separated finite sets.

Let $\mathcal R$ be a collection of pairwise $2C_c$-incomparable $(\delta,t)$-rectangles of the form $f^\delta(I)$ where $f\in \mathcal F$ and $I\sub J/4$. Suppose each rectangle in $\mathcal{R}$ is tangent to at least one $w_R\in \mathcal W$ and at least one $b_R\in\mathcal B$.

Then there exists $\eta \in \{-1,0,1\}$ and a set $\mathcal{R}'\subset \mathcal{R}$ with the following properties.
\begin{itemize}
\item $\#\mathcal{R}'\geq \frac{1}{3}\#\mathcal{R}$.
\item If we define $\mathcal{W}'=\{w+\eta \lambda\delta\colon w\in\mathcal{W}\}$, then each rectangle in $\mathcal{R}'$ is $2\lambda$-tangent to at least one $w_R\in \mathcal W'$ and at least one $b_R\in\mathcal B$.
\item For each $R\in \mathcal R'$, we have that $w_R,b_R$ intersect exactly twice in $J$.
\end{itemize}
\end{cor}

We are now ready to prove Lemma \ref{prop_number_estimate1}. Applying Corollary \ref{choiceOfShift} to $\mathcal{W}$ and $\mathcal{B}$, we obtain a perturbed set $\mathcal{W}'$ and a subset $\mathcal{R}'\subset\mathcal{R}$. We claim that $\mathcal{W}'\cup\mathcal{B}$ is a finite subset of a family $\mathcal{F}'$ that has cinematic constant at most $2K$. Indeed, we can simply define $\mathcal{F}' = (\mathcal W - \eta \lambda\delta)\cup \mathcal B$ (recall that  $\mathcal W\subset C^2(J)$, so $\mathcal W - \eta \lambda\delta=\{f - \eta \lambda\delta\colon f\in \mathcal W\}$. It is easy to check that since $\mathcal W$ and $\mathcal{B}$ are $2t$-separated, \eqref{cinematicFunctionCondition} continues to hold with $2K$ in place of $K$ (the only interesting case to check is when $f\in \mathcal W - \eta \lambda\delta$ and $g\in \mathcal B$, or vice-versa). Also, $\mathcal W'\cup \mathcal B$ is at least $t$ separated, if $c_2$ is chosen to be small enough.

Our collection $\mathcal{W}'$, $\mathcal{B},$ and $\mathcal{R}'$ satisfies property A1 from the statement of Lemma \ref{prop_number_estimate0}. After applying a second harmless infinitesimal perturbation, we can ensure that $\mathcal{W}'$, $\mathcal{B},$ and $\mathcal{R}'$  satisfy Properties A2 and A3, while maintaining property A1. Lemma \ref{prop_number_estimate1} now follows from Lemma \ref{prop_number_estimate0}.

Finally, in light of the preliminary reductions given by Lemmas \ref{lem_c_2} and \ref{lem_r=t}, Proposition \ref{thm_number_estimate} follows from Lemma \ref{prop_number_estimate1} via a standard random sampling argument. See i.e.~the proof of Lemma 1.4 from \cite{Wolff2000}.

%%%%%%%%%%%%%%%%%%%%%%%%
%%%%%%%%%%%%%%%%%%%%%%%%
%%%%%%%%%%%%%%%%%%%%%%%%

\section{Proof of Theorem \ref{thm_wolffAnalogue_maximal_general}}\label{proofOfMaximalThmSection}
We have now assembled the necessary ingredients to execute the proof sketch described in Section \ref{thmthm_wolffAnalogue_maximal_generalSketchSection}. We will begin with a few preliminary reductions. First, if $\mathcal{F}\subset C^2(J)$ is a family of cinematic functions, the results from Sections \ref{geomOfCinematicSection} and \ref{lensCountingBoundSection} allow us to control the functions $f\in\mathcal{F}$ on a slightly shorter interval $J/4$. This is inconvenient, since the domain of integration in \eqref{eqn_L32Bd_general} might involve the entire interval $J$. Our first task is to show that Theorem \ref{thm_wolffAnalogue_maximal_general} will follow from the following (superficially) weaker version, in which the domain of integration has been restricted to a slightly smaller set.

\begin{lem}\label{thm_wolffAnalogue_maximal_generalShortCurve}
Let $\eps>0$, $0<\alpha\leq\zeta\leq 1,$ and $D,K\geq 1$. Then the following is true for all $\delta>0$ sufficiently small (depending on $D,K,$ and $\eps$). Let $\mathcal{F}\subset C^2(J)$ be a family of cinematic functions, with cinematic constant $K$ and doubling constant $D$. Let $E$ be a $(\delta,\alpha;\delta^{-\eps})_1\times(\delta,\alpha;\delta^{-\eps})_1$ quasi-product.

Let $F\subset\mathcal{F}$ be a set of functions that satisfies the Frostman-type non-concentration condition
\[
\#(F \cap B) \leq  \delta^{-\eps}(r/\delta)^\zeta\quad\textrm{for all balls}\ B\subset C^2(I)\ \textrm{of radius}\ r\geq\delta.
\]

Let $J' = \frac{1}{16}J$ and let $E' = E\cap (J'\times\RR)$. Then
\begin{equation}\label{eqn_L32Bd_generalEprime}
 \int_{E'}\Big( \sum_{f\in F}{\bf 1}_{f^\delta}\Big)^{3/2} \leq \delta^{2-\alpha/2-\zeta/2-C_0\eps}(\# F),
\end{equation}
where $C_0=C_0(D)>0$ depends on $D$.
\end{lem}

Using Lemma \ref{thm_wolffAnalogue_maximal_generalShortCurve} we can prove Theorem \ref{thm_wolffAnalogue_maximal_general} as follows. Let $J_1,\ldots,J_N$, $N\sim|\log\delta|$ be a sequence of sub-intervals of $J$, so that $J \backslash \bigcup_{i=1}^N \frac{1}{16}J_i$ consists of two intervals of length $\delta$ (these intervals will be at the two endpoints of $J$). Note that if $J_0$ is an interval of length $\delta$ and if $E_0 = E \cap (J_0\times\RR)$, then
\begin{equation}\label{trivialStubBd}
 \int_{E_0}\Big( \sum_{f\in F}{\bf 1}_{f^\delta}\Big)^{3/2} \lesssim \delta^2(\#F)^{3/2} \leq \delta^{2-\alpha/2-\zeta/2-C_0\eps}(\# F).
\end{equation}

For each index $i$, apply Lemma \ref{thm_wolffAnalogue_maximal_generalShortCurve} to the restriction $\mathcal{F}|_{J_i}\subset C^2(J_i)$; this is still a family of cinematic functions, with cinematic constant $K$ and doubling constant $D$. Using the triangle inequality, we conclude that
\[
\int_{E}\Big( \sum_{f\in F}{\bf 1}_{f^\delta}\Big)^{3/2} \lesssim |\log\delta| \delta^{2-\alpha/2-\zeta/2-C_0\eps}(\# F).
\]
Theorem \ref{thm_wolffAnalogue_maximal_general} now follows, provided we select $C>C_0$ and restrict $\delta>0$ sufficiently small.

Next, we remark that when proving Lemma \ref{thm_wolffAnalogue_maximal_generalShortCurve}, we may suppose that $F$ is $\delta$-separated. This follows from the triangle inequality and observation that by \eqref{frostmanConditionOnF} each $\delta$-ball contains at most $\delta^{-\eps}$ elements of $F$.

Finally, we claim that Lemma \ref{thm_wolffAnalogue_maximal_generalShortCurve} follows from the following restricted weak-type estimate. For each $x\in\RR^2$, let $F(x) = \{f\in F\colon x\in f^\delta\}.$ For each $x\in\RR^2$, $\#F(x)\leq\#F\lesssim \delta^{-\eps-\zeta}$. Thus by dyadic pigeonholing, there is a set $E_0\subset E$ and an integer $\mu$ so that $\mu\leq \#F(x)<2\mu$ for all $x\in E_0$, and $\mu^{3/2}|E_0|\approx\textrm{LHS of}\ \eqref{eqn_L32Bd_generalEprime}$. Thus in order to prove Lemma  \ref{thm_wolffAnalogue_maximal_generalShortCurve} (and hence Theorem \ref{thm_wolffAnalogue_maximal_general}), it suffices to establish the estimate
\begin{equation}\label{E0Bd}
    |E_0|\lessapprox \delta^{2-\frac\Ga 2-\frac \zeta 2-C_1\eps}\# F \mu^{-3/2},
\end{equation}
for some $C_1=C_1(D)$.

\subsection{Two ends reductions}
In this section, we will find quantities called $t$ and $\Delta$, so that for a typical point $x\in E_0,$ the functions in $F(x)$ are localized to a ball in $C^2(J)$ of radius $t$, and for a typical pair of functions $f,g\in F(x)$, we have $\Delta(f,g)\sim\Delta$. Our main tool will be the ``two-ends'' reduction. See \cite{Twoends} for an introduction to the topic.

\subsubsection{Finding the diameter of $F(x)$}\label{subsub_two_ends}
Since $F\subset C^2(J)$ is bounded, so is the set $F(x)$ for each $x\in E_0$. However, it is possible that $F(x)$ (or a large fraction thereof) has diameter much smaller than $\operatorname{diam}(F)$ for a typical $x\in E_0$. In this section, we will find the ``true'' diameter of $F(x)$ for a typical $x\in E_0$.

For $t>0$ and $x\in \R^2$, define
\begin{equation}\label{eqn_defn_f}
    n_1(t,x)=\sup_{g\in C^2(J)}\#  F(x)\cap B(g,t),
\end{equation}
where $B(g,t)\subset C^2(J)$ is the ball centered at $g\in C^2(J)$ of radius $t$. With this, we further define
\begin{equation}\label{eqn_defn_g}
    n_2(x)=\sup_{t\in [\delta,K]}n_1(t,x)t^{-\Ge}.
\end{equation}
Hence, given $x$, there is some $t(x)\in [\delta,K]$ such that $n_2(x)\le 2n_1(t(x),x)t(x)^{-\Ge}$. There is also some $g_x$ such that
\begin{equation}
    \#  F(x)\cap B(g_x,t(x))\ge n_2(t(x),x)/2,
\end{equation}
so we have
\begin{equation}\label{eqn_g_two_ends}
    n_2(x)\le 4 t(x)^{-\Ge}\#  F(x)\cap B(g_x,t(x)).
\end{equation}
For simplicity we denote $B_x=B(g_x,t(x))$.

Since $t(x)\in [\delta,K]$, by a dyadic pigeonholing, there is a Borel set $E_1\sub E_0$ and a number $t\in [\delta,2K]$ such that $t(x)\in (t/2,t]$ for every $x\in E_1$, and $|E_1|\approx |E_0|$. It then suffices to show
\begin{equation}
    |E_1|\lessapprox \delta^{2-\frac\Ga 2-\frac \zeta 2-C_1\eps} \#  F \mu^{-3/2}.
\end{equation}

Since $\mathcal{F}$ is doubling, we can cover it by finitely overlapping balls $B$ of radius $t$, and denote
\begin{equation*}
     F_B= F\cap (3B),
\end{equation*}
so $\{ F_B\}$ covers $ F$ and is finitely overlapping. Also, let $E_B$ be the set of $x\in E_1$ such that $B_x$ intersects $B$, and thus $\{E_B\}$ is also finitely overlapping. Thus it suffices to show for each $B$
\begin{equation}\label{eqn_E_B}
    |E_B|\lessapprox \delta^{2-\frac\Ga 2-\frac \zeta 2-C_1\eps} \#  F_B \mu^{-3/2}.
\end{equation}

Now fix a ball $B$ of radius $t$ and let $x\in E_B$. By \eqref{eqn_g_two_ends}, \eqref{eqn_defn_g} and \eqref{eqn_defn_f}, we have
\begin{align*}
    \# F (x)\cap B_x\gtrsim t^{\Ge} n_2(x)\gtrsim t^{\Ge}n_1(K,x)=t^{\Ge}\# F (x)\sim t^\Ge \mu.
\end{align*}
By another dyadic pigeonholing, we may find a subset $E'_B\sub E_B$ and a scale $\mu_1$ with
\begin{equation}\label{eqn_defn_mu_1}
    t^\Ge\mu\lesssim\mu_1\le \mu,
\end{equation}
such that $|E'_B|\approx |E_B|$ and for every $x\in E'_B$ we have $\# F (x)\cap B_x\sim \mu_1$. It suffices to show \eqref{eqn_E_B} with $E_B$ replaced by $E'_B$.

Now let $x\in E'_B$. If $B$ intersects $B_x$ then we have $B_x\sub 3B$. Hence
\begin{equation}\label{eqn_lower_bound_Z_B}
    \# F_B(x)=\# F (x)\cap (3B)\cap B_x=\# F (x)\cap B_x\sim \mu_1.
\end{equation}

From now on we fix such a ball $B$ and aim to prove \eqref{eqn_E_B}. We end with a useful lemma.
\begin{lem}\label{lem_|z-z'|}
Fix $x\in E'_B$ and $g\in C^2(J)$. If $\delta t^{-1}<\Gl <1$, then for at most $4\cdot (2\Gl)^\Ge$ fraction of $f\in  F_B(x)$ we have $\Vert f-g\Vert \le \Gl t$.
\end{lem}

\begin{proof}
Let $g\in C^2(J)$. By \eqref{eqn_g_two_ends} and \eqref{eqn_lower_bound_Z_B}, we have
\begin{equation*}
    n_2(x)\le 4 t(x)^{-\Ge}\#  F (x)\cap B_x\le 4 \cdot 2^\Ge t^{-\Ge}\#  F_B(x).
\end{equation*}
But by the definition of $n_2$ in \eqref{eqn_defn_g} and $\lambda t\in [\delta,K]$, we also have
\begin{align*}
  n_2(x)
  &\ge n_1(\Gl t,x)(\Gl t)^{-\Ge}
  \ge\# F (x)\cap B(g,\Gl t ) \Gl^{-\Ge} t^{-\Ge}
\end{align*}
Combining the lower and upper bounds of $n_2(x)$ gives
\begin{equation*}
    \# F (x)\cap B(g,\Gl t )\le  4\cdot (2\Gl)^\Ge\#  F_B(x),
\end{equation*}
from which the result follows.
\end{proof}

\subsubsection{Finding the typical tangency of curve-curve intersection}
As discussed in Section \ref{thmthm_wolffAnalogue_maximal_generalSketchSection}, a typical pair of intersecting curves in $F$ might intersect transversely, tangentially, or something in-between. In this section, we will find the typical degree of tangency when two curves intersect. We will do this by using a two-ends reduction argument similar to that in Section \ref{subsub_two_ends}. This two-ends reduction will use a much smaller quantity in place of $\eps$, so that the resulting refinement retains the two-ends property from  Section \ref{subsub_two_ends}.

For $x\in E_B$ and $\Delta>0$, define
\begin{equation}\label{eqn_defn_h}
    n_3(\Delta,x)=\max_{k\in  F_B(x)}\#\{f\in F_B(x):\Delta(f,k)\le \Delta\}.
\end{equation}
Since $ F_B$ is contained in $3B$ which has diameter $6t$, invoking \eqref{eqn_lower_bound_Z_B} and the relation $\Delta(f,k)\leq \Vert f-g\Vert$, we have
\begin{equation}\label{eqn_h(t)}
    n_3(6t,x)=\# F_B(x).
\end{equation}
We further define
\begin{equation}\label{eqn_defn_k}
  n_4(x)= \sup_{\Delta\in [\delta,6t]}n_3(\Delta,x)\Delta^{-\Ge^2}.
\end{equation}
Given $x$, there is some $\Delta_x\in [\delta,6t]$ such that
\begin{equation}\label{eqn_defn_Delta_x}
    n_4(x)\le 2n_3(\Delta_x,x)\Delta_x^{-\Ge^2}.
\end{equation}
There is also some $k_x\in  F_B(x)$ attaining the maximum in \eqref{eqn_defn_h}, that is,
\begin{equation}\label{eqn_defn_vx}
    \#\{f\in F_B(x):\Delta(f,k_x)\le \Delta_x\}= n_3(\Delta_x,x),
\end{equation}
and hence
\begin{equation}\label{eqn_k}
    n_4(x)\le 2\Delta_x^{-\Ge^2}\#\{f\in F_B(x):\Delta(f,k_x)\le \Delta_x\}.
\end{equation}
Since $\Delta_x\in [\delta,6t]$, by a dyadic pigeonholing, there is a Borel set $E''_B\sub E'_B$ and a number $\Delta\in [\delta,12t]$ such that $\Delta_x\in (\Delta/2,\Delta]$ for every $x\in E''_B$, and $|E''_B|\approx |E'_B|$. It then suffices to show
\begin{equation}\label{eqn_E''_B_bound}
    |E''_B|\lessapprox \delta^{2-\frac\Ga 2-\frac \zeta 2-C_1\eps} \#  F_B \mu^{-3/2}.
\end{equation}

\subsubsection{Shadings}
At this point, the set $E_0$ from \eqref{E0Bd} has been divided into (mostly) disjoint pieces $E_B$. For each $x\in E_B$, most of the functions $f\in F(x)$ are contained in the ball $B$, and only a small fraction of these functions are contained in any ball substantially smaller than $B$. For a typical pair of functions $f,g\in F(x)$, $\Delta(f,g)$ has size roughly $\Delta$. For technical reasons, it will be convenient to restrict the sets $f^\delta$ to a slightly smaller ``shading'' $Y_1(f),$ so that the above properties hold whenever $x\in Y_1(f)$. More precisely, we define
\begin{equation}\label{eqn_defn_Y_1}
    Y_1(f)=\{x\in E''_B\cap f^\delta:\Delta(f,k_x)\le \Delta\}.
\end{equation}
%The term ``shading" comes from the Kakeya conjecture; see for instance \cite{GuthZahl}.
%Intuitively, $Y_1(f)$ is an important subset of $f^\delta$ since this is where the $\delta$-neighbourhoods of the graphs of functions in $ F_B$ overlap greatly.

In view of \eqref{eqn_k}, \eqref{eqn_h(t)} and $\Delta_x\le \Delta$, for every $x\in E''_B$ we have
\begin{equation}\label{eqn_Z'_B_lower_bound}
    \#\{f\in  F_B(x):\Delta(f,k_x)\le \Delta\}\gtrsim (\Delta/t)^{\Ge^2}\# F_B(x)\sim (\Delta/t)^{\Ge^2} \mu_1.
\end{equation}
On the other hand, we have a trivial bound
\begin{equation*}
    \#\{f\in  F_B(x):\Delta(f,k_x)\le \Delta\}\leq \# F_B(x)\sim \mu_1.
\end{equation*}
Now we apply another dyadic pigeonholing to find a subset $E_2\sub E''_B$ with $|E_2|\approx |E''_B|$ and a scale $\mu_2$ with
\begin{equation}\label{eqn_defn_mu_2}
    (\Delta/t)^{\Ge^2} \mu_1\lesssim \mu_2\le \mu_1,
\end{equation}
such that for every $x\in E_2$ we have
\begin{equation}\label{eqn_mu_2_number}
    \#\{f\in  F_B(x):\Delta(f,k_x)\le \Delta\}\sim \mu_2.
\end{equation}
Equivalently, we have
\begin{equation}\label{eqn_Z'_B_integral_bound}
    \#\{f\in F_B:x\in Y_1(f)\}=\sum_{f\in F_B}{\bf 1}_{Y_1(f)}(x)\sim \mu_2,\quad \forall x\in E_2.
\end{equation}

It then suffices to prove
\begin{equation}\label{eqn_E_2_bound}
    |E_2|\lessapprox \delta^{2-\frac\Ga 2-\frac \zeta 2-C_1\eps} (\# F_B) \mu^{-3/2}.
\end{equation}

We record a coarse bound that is used in some dyadic pigeonholing argument later. Since $\# F_B(x)\gtrsim t^\Ge\mu$ and $\mu\le\# F\lesssim \delta^{-1-\Ge}$, the right hand side of \eqref{eqn_E_2_bound} is trivially bounded below by $\delta^{3}$, for $\delta$ small enough. If $|E_2|\le \delta^{3}$ then we are done, and so in the following we may assume
\begin{equation}\label{eqn_delta^3}
    |E_2|>\delta^{3}.
\end{equation}

\subsection{Fine-scale rectangles}\label{fineScaleRectanglesSection}
In this section, we will show that $E_2$ can be covered by curvilinear rectangles of dimensions roughly $\delta\times \delta / \sqrt{t\Delta}$, so that these rectangles are (mostly) disjoint, and for most pairs $f,g\in F_B(x)$, we have that $Y_1(f)\cap Y_1(g)$ is localized to roughly one of these rectangles.

\begin{lem}\label{prop_Y_2}
There is a constant $C_R=C_R(K)$ ($R$ stands for ``rectangle") such that the following is true. Let
\begin{equation}
 t'=C_R t\Delta/\delta.
\end{equation}
For each $x\in E_2$, there is a $(\delta,t')$-rectangle $R_x$ over $J/8$, such that $x\in R_x$, and every $f\in F_B$ with $x\in Y_1(f)$ is tangent to $R_x$.
\end{lem}
\begin{proof}
Denote $x=(\theta_0,y_0)$ where $x\in J/16$. Let $I$ be the interval of length $\sqrt{\delta/t'}$ centred at $\theta_0$. Since $t\gtrsim\Delta\ge \delta$, we have $I\sub J/8$ if $C_R$ is large enough. With this, define the curvilinear rectangle
\begin{equation}
    R_x=k_x^\delta(I),
\end{equation}
where $k_x$ is as in \eqref{eqn_defn_vx}. We now show that $R_x$ is as required.

First we have $x\in R_x$ since $k_x\in F_B(x)$ by \eqref{eqn_defn_h}. Also, let $f$ be such that $x\in Y_1(f)$, then in particular, $x\in f^{\delta}$. By the triangle inequality, $|f(\theta_0)-k_x(\theta_0)|\le 2\delta$. By Part \eqref{item_2c} of Lemma \ref{prop_two_zeros} at the scale $4\delta$, we have
$$
E_{4\delta}=\{\theta\in J/4:|f(\theta)-k_x(\theta)|\le 4\delta\}
$$
contains an interval $I'$ obeying $\theta_0\in I'$ and
\begin{equation*}
    |I'|\gtrsim \delta d(f,k_x)^{-1/2}(\delta+\Delta(f,k_x))^{-1/2}.
\end{equation*}
Since $x\in Y_1(f)$ we have $\Delta(f,k_x)\le \Delta$. Also $d(f,k(x))\le 6t$. If $C_R$ is chosen to be large enough, then the interval $I$ containing $\theta_0$ will have length $\sqrt{\delta/t'}\le |I'|$, and thus $I\sub I'$ and $R_x$ is tangent to $f$.
\end{proof}

\subsubsection{Shadings of fine-scale rectangles}
For each $(\delta,t')$-rectangle $R$, we define its shading as
\begin{equation}\label{eqn_defn_Y_3}
    Y_2(R)=\{x\in C_s R\cap E_2:R_x \text{ is $C_s$-comparable to }R\},
\end{equation}
where $C_s=C_s(K)\ge 100$ ($s$ standing for ``shadings") is a suitable constant to be determined. Intuitively, $Y_2(R)$ is an important subset of $R$ since it is where the $\delta$-neighbourhoods of the curves in $F_B$ overlap greatly. %Note also $C_sR$ is over $J/$ if $\delta$ is small enough.

For technical purposes, we also define a slightly smaller shading of $R$ as
\begin{equation*}
    Y'_2(R)=\{x\in R\cap E_2:R_x \text{ is $100$-comparable to }R\}.
\end{equation*}
By Lemma \ref{prop_Y_2}, we have
\begin{equation}\label{eqn_E_2_shadings}
    E_2=\bigcup_{x\in E_2}Y'_2(R_x).
\end{equation}

We also record an elementary lemma on the shadings of curvilinear rectangles.
\begin{lem}\label{lem_shading_comparable}
If $C_s$ is suitably chosen, then for any pair $(R,R')$ of $100$-comparable $(\delta,t)$-rectangles, we have $Y'_2(R')\sub Y_2(R)$.
\end{lem}
\begin{proof}
If $R,R'$ are $100$-comparable, then by Lemma \ref{prop_comparable_rectangle}, we have $R'\sub C_sR$ if $C_s$ is large enough. Thus if $x\in Y'_2(R')$, then $x\in C_sR\cap E_2$. Also, $R_x$ is $100$-comparable to $R'$, and since $R'$ is $100$-comparable to $R$, Corollary \ref{cor_comparable_rectangle} implies that $R_x$ is $C_s$-comparable to $R$ if $C_s$ is large enough. Thus $Y'_2(R')\sub Y_2(R)$.
\end{proof}

%\begin{lem}\label{lem_shading_intersection_empty}
%If $R,R'$ are $C'_{s}$-incomparable for some suitable constant $C'_{s}\ge C_s$, then $Y_2(R)\cap Y_2(R')=\varnothing$.
%\end{lem}

%\begin{proof}
%Suppose not. Then there is some $x\in Y_2(R)\cap Y_2(R')$. Since $x\in Y_2(R)$, $R_x$ is $C_s$-comparable to $R$. Similarly, $R_x$ is $C_s$-comparable to $R'$. By Corollary \ref{cor_comparable_rectangle}, we have $R,R'$ are $C'_s$-comparable for some suitable constant $C'_s\ge C_s$, which is a contradiction.
%\end{proof}

%\begin{lem}\label{lem_Y_1(z)_intersect_Y_2(R)_implies_tangent}
%Let $R=R_x$, $x\in E_2$ be a $(\delta,t')$-rectangle. If $f\in F_B$ is such that $Y_1(f)\cap Y_2(R)\neq \varnothing$, then $f$ is $C_T$ tangent to $R$, where $C_T$ ($T$ stands for ``tangent") depends in turn on $K$ only.
%\end{lem}
%\begin{proof}
%Let $y\in Y_1(f)\cap Y_2(R)$. Since $y\in Y_1(f)$, by Lemma \ref{prop_Y_2}, we have $f$ is tangent to $R_y$. Since $y\in Y_2(R)$, we have $R_y$ is $C_s$-comparable to $R$. Using Corollary \ref{cor_comparable_rectangle}, we have $f$ is tangent to $R$ for some suitable constant $C_T$.
%\end{proof}

\subsubsection{Finding an incomparable set of fine-scale rectangles}
\begin{lem}\label{lem_maximal_R}
There is a pairwise $100$-incomparable subcollection $\mathcal R$ of $\{R_x:x\in E_2\}$ such that the following holds:
\begin{enumerate}
    \item \label{part_1_lem_maximal_R}There is some $\Lambda\in (0,1)$ such that for each $R_x\in \mathcal R$ we have $\Lambda\leq |Y_2(R_x)|< 2 \Lambda$.
    \item \label{part_2_lem_maximal_R} $Y_2(R)\cap Y_2(R')=\varnothing$ for $R\neq R'\in \mathcal R$, and
    \begin{equation}\label{eqn_E_2_approx}
    |E_2|\lessapprox \sum_{R\in \mathcal R}|Y_2(R)|\sim \Lambda\#\mathcal R.
    \end{equation}
    \item We have the coarse bounds
    \begin{equation}\label{eqn_coarse_number_rectangles}
        \#\mathcal R\leq \delta^{-C},\quad \Lambda\ge \delta^{C'},
    \end{equation}
    for $\delta$ sufficiently small (depending on $K,D$), where $C,C'$ depend on $D$ only.
\end{enumerate}
\end{lem}

\begin{proof}

Before we construct $\mathcal R$, we first construct a larger collection $\mathcal R^*$ in the following greedy way.

Denote $\Lambda_1=\sup_{x\in E_2}|Y'_2(R_x)|$. Since we have assumed $|E_2|>\delta^{3}$ we have $\Lambda_1>0$ by \eqref{eqn_E_2_shadings}. So we pick $R_1=R_{x_1}$ such that $|Y'_2(R_1)|>\Lambda_1/2$.

To choose $R_2$, we consider all $(\delta,t)$-rectangles $R_x$ that are $100$-incomparable to $R_1$; denote this collection by $\mathcal R^{(1)}$. Denote $\Lambda_2=\sup_{R_x\in \mathcal R^{(1)}}|Y'_2(R_x)|$. If $\Lambda_2=0$, then we just let $\mathcal R^*=\{R_1\}$. With this, \eqref{eqn_E_2_shadings} and Lemma \ref{lem_shading_comparable}, we have (here $\sim_{100}$ means $100$-comparable)
\begin{align*}
    |E_2|
    &\le\left|\bigcup_{x\in E_2,R_x\sim_{100}R_1 }Y'_2(R_x)\right|+\left|\bigcup_{x\in E_2,R_x\not\sim_{100}R_1 }Y'_2(R_x)\right|\\
    &\le |Y_2(R_1)|+\sum_{R_x\in \mathcal R^{(1)}}\left|Y'_2(R_x)\right|\\
    &= |Y_2(R_1)|.
\end{align*}

If $\Lambda_2>0$, then we pick $R_2\in \mathcal R^{(1)}$ such that $|Y'_2(R_2)|>\Lambda_2/2$. To choose $R_3$, we consider all $(\delta,t)$-rectangles that are $100$-incomparable to both $R_1$ and $R_2$; denote this collection by $\mathcal R^{(2)}$. Denote $\Lambda_3=\sup_{R_x\in \mathcal R^{(2)}}|Y'_2(R_x)|$. If $\Lambda_3=0$, we just let $\mathcal R^*=\{R_1,R_2\}$. With this, we similarly have
\begin{align*}
    |E_2|
    &\le \left|\bigcup_{x\in E_2,R_x\sim_{100} R_1 }Y'_2(R_x)\right|+\left|\bigcup_{x\in E_2,R_x\sim_{100} R_2 }Y'_2(R_x)\right|
    \le |Y_2(R_1)|+|Y_2(R_2)|.
\end{align*}
We can continue arguing in this fashion. Note that Lemma \ref{prop_coarse_bound_R} implies that the process must stop at some finite time $N<\delta^{-C}$ for some constant $C=C(D)$, proving the first inequality in \eqref{eqn_coarse_number_rectangles}. Thus $\Lambda_{N+1}=0$, $\Lambda_N>0$, $\mathcal R^*=\{R_1,\dots,R_N\}$, and
\begin{equation*}
    |E_2|\leq \sum_{n=1}^N |Y_2(R_n)|.
\end{equation*}
We may assume without loss of generality that $|Y_2(R_n)|\ge \delta^{C+3}$ for all $n$. Indeed, we may throw away all $n$ such that $|Y_2(R_n)|<\delta^{C+3}$, since the sum of all such $|Y_2(R_n)|$ is less than $\delta^{C+3} N < \delta^3<|E_2|/2$, as $|E_2|>\delta^3$ by our assumption.

Now for $j\ge 0$, we denote
\begin{equation*}
    \mathcal R^{(j)}=\{R_n\in \mathcal R^*:2^{j-1}\Lambda_N\le |Y_2(R_n)|<2^j \Lambda_N\}.
\end{equation*}

There are at most $O(|\log\delta|)$ many such $j$'s, since $\delta^{C+3}\le |Y_2(R_n)|\leq 1$. By pigeonholing, there is some $j$ such that
\begin{equation*}
    \sum_{n=1}^N |Y_2(R_n)|\lesssim |\log \delta|\sum_{n:R_n\in \mathcal R^{(j)}}|Y_2(R_n)|.
\end{equation*}

We now define
\begin{equation*}
    \mathcal R=\mathcal R^{(j)},\quad \Lambda=2^{j-1}\Lambda_{N}>0.
\end{equation*}

%We have $\Lambda\ge \delta^{C+3}$, otherwise by \eqref{eqn_E_2_shadings} and $\#\mathcal R\le \delta^{-C}$ we will have $|E_2|<\delta^{3}$, a contradiction to the assumption $|E_2|>\delta^3$ (recall \eqref{eqn_delta^3}).
Thus
\begin{equation*}
    |E_2|\le\sum_{n=1}^N |Y_2(R_n)|\lessapprox \Lambda \#\mathcal R.
\end{equation*}
Also, we have $\Lambda\ge \delta^{C+4}$, since $|E_2|>\delta^3$ and $\#\mathcal R\le N<\delta^{-C}$.

%Lastly, since $\mathcal R^{**}$ is $100$-incomparable, using Lemma \ref{prop_C_s_comparable}, we can find a subcollection $\mathcal R\sub \mathcal R^{**}$ of $C'_s$-incomparable rectangles such that $\#\mathcal R\sim \#\mathcal R^{**}$. Using Lemma \ref{lem_shading_intersection_empty}, we thus have $Y_2(R)\cap Y_2(R')=\varnothing$ for $R\ne R'\in \mathcal R$. Thus
%\begin{equation*}
    %|E_2|\ge \left|\bigcup_{R\in \mathcal R}Y_2(R)\right|\sim\Lambda\# \mathcal R,
%\end{equation*}
%and thus we have \eqref{eqn_E_2_approx}. Lastly, $\Lambda$
\end{proof}
We now observe the following. Since $Y_2(R)\sub C_sR\cap E_2$, $R$ is over an interval of length $\sim \sqrt{\delta/t'}$, and $E_2$ is further contained in a $(\delta,\Ga;\delta^{-\eps})_1\times (\delta,\Ga;\delta^{-\eps})_1$ quasi-product, we have
\begin{equation*}
    \Lambda\sim|Y_2(R)|\lesssim\delta^{2-\eps} (\sqrt{\delta/t'}/\delta)^\Ga\sim \delta^{2-\eps}\Delta^{-\Ga/2}t^{-\Ga/2}.
\end{equation*}
With this, we write
\begin{equation}\label{eqn_defn_tau}
    \Lambda=\tau \delta^{2-\eps}\Delta^{-\Ga/2}t^{-\Ga/2},
\end{equation}
where
\begin{equation}\label{upperBdTau}
0<\tau\lesssim 1.
\end{equation}
To establish \eqref{eqn_E_2_bound}, it thus suffices to show
\begin{equation}\label{eqn_number_R_bound}
    \#\mathcal R\lessapprox \tau^{-1}\delta^{-\frac\Ga 2-\frac \zeta 2-C_2\eps}\Delta^{\Ga/2}t^{\Ga/2} \#  F_B \mu^{-3/2},
\end{equation}
for some $C_2=C_2(D)$.

\subsection{Coarse-scale rectangles}\label{sec_coarse_scale_rectangles}
In this section, we will show that the fine-scale rectangles from Section \ref{fineScaleRectanglesSection} are contained in coarse-scale rectangles of dimensions roughly $\Delta\times \sqrt{\Delta/t}$. If a fine-scale rectangle $R$ is contained in a coarse-scale rectangle $\tilde R$, then $\tilde R$ is (roughly) a dilate of $R$.

% We would like to

% Since $R$ is a $(\delta,t')$-rectangle and $t'$ may be significantly larger than $t$, Theorem \ref{thm_number_estimate} may not be directly applied to $\mathcal R$. To solve this problem, we will enlarge our rectangles $R$ as follows.
For each $R\in \mathcal R$, we define $\tilde R(R)$ as follows. If $R=f^\delta(I),$ then $\tilde R(R)=f^\Delta(\tilde I)$, where $\tilde I$ is the $\Delta/\delta$ dilation of $I$, and hence it obeys $|\tilde I|=\sqrt{\frac{\Delta}{C_Rt}}$ where $C_R$ is as in Lemma \ref{prop_Y_2}. Thus $\tilde R(R)$ is a $(\Delta,C_R t)$-rectangle, and trivially we have $R\sub \tilde R(R)$. Also, since $I\sub J/8$, for $C_R$ large enough, we have $\tilde I \sub J/4$.

\begin{lem}\label{lem_coarse_tangent}
Let $R=R_x$ and $f\in F_B$ with $x\in Y_1(f)$, such that $f$ is tangent to $R$ by Lemma \ref{prop_Y_2}. If $C_R$ was chosen to be sufficiently large, then $f$ is also tangent to $\tilde R$.
\end{lem}
\begin{proof}
It follows from the same proof as that of Lemma \ref{prop_Y_2}, with $\Delta$ in place of $\delta$.
\end{proof}

Next, we pick a maximally $100$-incomparable subcollection $\tilde{\mathcal R}$ of $\{\tilde R(R):R\in \mathcal R\}$, and we will denote each member of $\tilde {\mathcal R}$ as $\tilde R$. In this way we have partitioned $\mathcal R$ into subcollections of the form $\mathcal R(\tilde R)$ consisting some of those $R\in \mathcal R$ such that the enlarged $\tilde R(R)$ of $R$ is $100$-comparable to $\tilde R$.

For each $R\in \mathcal R(\tilde R)$, we have $R\sub \tilde R(R)$, and since $\tilde R(R)$ is $100$-comparable to $\tilde R$, Lemma \ref{prop_comparable_rectangle} shows that $\tilde R(R)\sub C\tilde R$ for some suitable absolute constant $C\ge 1$. Thus $R\sub C\tilde R$, and so $R\cap E_2\sub C\tilde R\cap E_2$.

Recall we want to prove \eqref{eqn_number_R_bound}. Using \eqref{eqn_coarse_number_rectangles} and dyadic pigeonholing, we may find a subcollection $\tilde {\mathcal R}_1\sub \tilde {\mathcal R}$ such that for some $M$ we have
\begin{equation}\label{eqn_defn_M}
    \#\mathcal R(\tilde R)\sim M
\end{equation}
for each $\tilde R\in \tilde {\mathcal R}_1$; moreover, this $M$ is chosen such that if we define $\mathcal R_1=\cup_{\tilde R\in \tilde{\mathcal R_1}}\mathcal R(\tilde R)$, then $\#\mathcal R_1\approx \#\mathcal R$. As a result, we have $\#\mathcal R\approx \#\tilde{\mathcal R}_1 M$. Note that by Lemma \ref{lem_maximal_R}, we also have $|E_2|\lessapprox \Lambda \#\mathcal R_1$.

For this reason, in the subsequent argument we will slightly abuse notation and denote $\mathcal R=\mathcal R_1$ and $\tilde{\mathcal R}_1=\tilde {\mathcal R}$. Thus we have
\begin{equation}
    \#\mathcal R\approx \#\tilde{\mathcal R} M,
\end{equation}
and so it suffices to bound $\#\tilde{\mathcal R}$ and $M$.

We will now give an easy upper bound of $M$. Since $E_2$ is contained in a $(\delta,\Ga;\delta^{-\eps})_1\times (\delta,\Ga;\delta^{-\eps})_1$-quasi product and $C\tilde R$ is a the $O(\Delta)$-neighbourhood of a curve of length $\sim\sqrt{\Delta/t}$, we have
\begin{equation*}
    |C\tilde R\cap E_2|\lesssim \delta^{1-\eps}(\Delta/\delta)^\Ga \delta^{1-\eps} (\sqrt{\Delta/t}/\delta)^\Ga=\delta^{2-2\Ga-2\eps}\Delta^{3\Ga/2}t^{-\Ga/2}.
\end{equation*}
Recall $Y_2(R)$ defined as in \eqref{eqn_defn_Y_3}. For $R\in \mathcal R(\tilde R)$, we have $Y_2(R)\sub C_sR\cap E_2\sub C_s\tilde R\cap E_2$. But by Lemma \ref{lem_maximal_R}, $Y_2(R)$'s are disjoint, and thus we have a measure bound
\begin{equation}\label{eqn_M_bound_measure}
    M\lesssim \frac{|C_s\tilde R\cap E_2|}{|Y_2(R)|}\lesssim  \Lambda^{-1}\delta^{2-2\Ga-2\eps}\Delta^{3\Ga/2}t^{-\Ga/2}=\tau^{-1}\delta^{-2\Ga-2\eps}\Delta^{2\Ga},
\end{equation}
where in the last equality we have used \eqref{eqn_defn_tau}. Later we will give a second bound on $M$ using Proposition \ref{thm_number_estimate}.

\subsection{A bilinear $L^2$ bound}
In this section, we will analyze the interaction between fine and coarse-scale rectangles. We will show that for each fine-scale rectangle $R$, there are about $\mu^2$ pairs of functions $f,g$ that are tangent to $R$, so that $f^\delta\cap g^\delta$ is localized to $R$. Thus if $N$ functions $f\in F$ are tangent to a coarse-scale rectangle $\tilde R,$ then $\tilde R$ contains about $N^2/\mu^2$ fine-scale rectangles.  This is the ``$L^2$ argument'' alluded to in the proof sketch from Section \ref{thmthm_wolffAnalogue_maximal_generalSketchSection}.

We turn to the details. By construction, each $R\in \mathcal R$ is indexed by some $x\in E_2$, namely, $R=R_x$. We will write $x(R)$ to denote this value of $x$.
Let $f\in F_B$ and $R\in \mathcal R$. We say the pair $(f,R)$ is {\it good} if $x(R)\in Y_1(f)$. Note that if $(f,R)$ is good, then $f$ is tangent to $R$ by Lemma \ref{prop_Y_2}.

Denote by $\mathcal G$ the collection of all good pairs. For each $f\in F_B$, we define
\begin{equation}\label{eqn_defn_G(z)}
    \mathcal G(f)=\{R\in \mathcal R:(f,R)\in \mathcal G\}.
\end{equation}
For each $R\in \mathcal R$, we define
\begin{equation}\label{eqn_defn_G(R)}
    \mathcal G(R)=\{f\in F_B:(f,R)\in \mathcal G\}.
\end{equation}
For each $x\in E_2$, we have $x\in Y_1(f)$ if and only if $R_x\in \mathcal G(f)$. Thus by \eqref{eqn_Z'_B_integral_bound} and unravelling the definitions above, for each $R\in \mathcal R$ we have
\begin{equation}\label{eqn_number_G(R)}
    \#\mathcal G(R)\sim \mu_2.
\end{equation}

\subsubsection{Pigeonholing the good pairs}
For each $\tilde R\in \tilde {\mathcal R}$, we consider the quantity
\begin{equation*}
    \#\mathcal G(f)\cap \mathcal R(\tilde R)=\#\{R\in \mathcal R(\tilde R):(f,R)\in \mathcal G\}.
\end{equation*}
By the disjointness of $\mathcal R(\tilde R)$, we have
\begin{equation*}
    \#\mathcal G=\sum_{f\in F_B}\sum_{\tilde R\in \tilde {\mathcal R}}\#\mathcal G(f)\cap \mathcal R(\tilde R).
\end{equation*}
In view of \eqref{eqn_coarse_number_rectangles}, we may apply a dyadic pigeonholing to find an integer $q$ and a set $\mathcal{G}'\subset\mathcal{G}$ with $\#\mathcal{G}'\gtrsim |\log\delta|^{-1}\#\mathcal{G}$, so that for each $\tilde R\in\tilde{\mathcal{R}}$ and each $f\in F_B$, the quantity $\#\mathcal{G}'(f) \cap \mathcal{R}(\tilde R)$ is either 0 or $\sim q$. (Here we define $\mathcal G'(f)$ and $\mathcal G'(R)$ similarly as \eqref{eqn_defn_G(z)} and \eqref{eqn_defn_G(R)}.)

For each $\tilde R\in\tilde{\mathcal{R}}$, define
\begin{equation}\label{eqn_F(R)}
 F(\tilde R) = \{f\in F_B \colon \mathcal{G}'(f) \cap \mathcal{R}(\tilde R)\neq\varnothing\}.
\end{equation}
For each $f\in F(\tilde R)$, we have in particular $\mathcal G(f)\cap \mathcal R(\tilde R)\neq \varnothing$, and thus there is some $R\in \mathcal R(\tilde R)$ such that $(f,R)\in \mathcal G$. Hence $f$ is tangent to $R$, and using Lemma \ref{lem_coarse_tangent}, we also have $f$ is tangent to $\tilde R$.

By construction, $\#\mathcal{G}'(f) \cap \mathcal{R}(\tilde R)\sim q$ for each $f\in F(\tilde R)$, and hence
\begin{equation}\label{sideOfCalZTildeREstimate}
\# F(\tilde R) \sim \frac{1}{q} \# \{(f,R)\in \mathcal{G}'\colon R\in\mathcal{R}(\tilde R)\}=\frac{1}{q} \sum_{R\in\mathcal{R}(\tilde R)}\# \mathcal{G}'(R).
% \frac{\mu \#\mathcal{R}(\tilde R)}{m}.
\end{equation}
More generally, for every ball $B(g,r)\sub C^2(J)$ we have
\begin{equation}\label{sideOfCalZTildeREstimate_ball}
    \# F(\tilde R)\cap B(g,r) \sim \frac{1}{q} \sum_{R\in\mathcal{R}(\tilde R)}\# \mathcal{G}'(R)\cap B(g,r).
\end{equation}

By \eqref{eqn_number_G(R)} and \eqref{eqn_defn_M}, for each $\tilde R\in\tilde{\mathcal{R}}'$ we have
\[
\sum_{R\in\mathcal{R}(\tilde R)}\#\mathcal{G}'(R) \leq \sum_{R\in\mathcal{R}(\tilde R)}\#\mathcal{G}(R)\sim\mu_2 M,
\]
and hence
\begin{equation*}
    \#\mathcal G'=\sum_{\tilde R\in \tilde {\mathcal R}}\sum_{R\in\mathcal{R}(\tilde R)}\#\mathcal{G}'(R)\lesssim \mu_2 M\# \tilde{\mathcal{R}}.
\end{equation*}
On the other hand, we have
\begin{equation}
    \#\mathcal G'\gtrsim |\log\delta|^{-1}\#\mathcal G=|\log\delta|^{-1}\sum_{\tilde R\in \tilde {\mathcal R}}\sum_{R\in\mathcal{R}(\tilde R)}\#\mathcal{G}(R)\sim |\log\delta|^{-1}\mu_2 M\# \tilde{\mathcal{R}},
\end{equation}
where the final quasi-inequality used \eqref{eqn_number_G(R)}. Combining the upper and lower bounds, we see that for a $\gtrsim |\log\delta|^{-1}$ fraction of $\tilde R\in \tilde {\mathcal R}$ we have
\begin{equation}\label{eqn_sum_number_G'(R)}
    \sum_{R\in\mathcal{R}(\tilde R)}\#\mathcal{G}'(R)\gtrsim |\log \delta|^{-1} \mu_2 M.
\end{equation}
Denote this subcollection as $\tilde{\mathcal R}'$. Similarly, for each $\tilde R\in \tilde {\mathcal R}'$, for a $\gtrsim |\log \delta|^{-1}$ fraction of $R\in \mathcal R(\tilde R)$ we have
\begin{equation}\label{eqn_number_G'(R)}
    \#\mathcal G'(R)\gtrsim |\log \delta|^{-1} \mu_2.
\end{equation}
Denote this subcollection as $\mathcal R'(\tilde R)$, which has cardinality $\gtrsim |\log\delta|^{-1}M$. Denote
\begin{equation}
    \mathcal R'=\bigsqcup_{\tilde R\in \tilde {\mathcal R}'}\mathcal R'(\tilde R).
\end{equation}
Thus $\#\mathcal R'\gtrsim |\log\delta|^{-2}\#\mathcal R$. Since each $\#\mathcal R'(\tilde R)\approx M$, it suffices to bound $\#\mathcal R'$.

\subsubsection{Each fine-scale rectangle has many good pairs}\label{subsec_number_good_pairs}
\begin{lem}\label{lem_pairs_lower_bound}
If $\delta>0$ is sufficiently small, then for each $R\in \mathcal R'$, there are at least $(\#\mathcal G'(R))^2/3$ pairs $(f,g)\in (\mathcal G'(R))^2$ such that $\delta^{2\eps}t\leq \Vert f-g\Vert \leq 6t $ and $\delta^{\eps}\Delta\leq \Delta(f,g)+\delta\leq 2\Delta$.
\end{lem}

\begin{proof}
Let $R=R_x\in \mathcal R'$ where $x\in E_2$. Fix $g\in \mathcal G'(R)$. We deal with $\Vert f-g\Vert$ first. If $t\leq \delta^{1-\eps}$, then we immediately have $\Vert f-g\Vert \ge \delta\geq \delta^{2\eps}t$ for every $f\in \mathcal G'(R)\backslash \{g\}$, since the elements of $F$ are $\delta$-separated.

Next, suppose $t>\delta^{1-\eps}$, and let $\Gl=\delta^{2\eps}$; by hypothesis we have $\delta t^{-1}<\Gl <1$. Thus we can apply Lemma \ref{lem_|z-z'|} to conclude that
\begin{align*}
&\#\{f\in \mathcal G'(R):\Vert f-g\Vert \le \Gl t\}\\
    &\le \#\{f\in F_B(x):\Vert f-g\Vert\le \Gl t\}\\
    &\lesssim \Gl^\Ge \mu_1\\
    &\lesssim \Gl^\eps \delta^{-\eps^2}\mu_2\\
    & = \delta^{\eps^2}\mu_2\\
    &\lesssim \delta^{\eps^2}|\log \delta|\#\mathcal G'(R),
\end{align*}
where the third inequality used \eqref{eqn_defn_mu_2} and in the last inequality we have used \eqref{eqn_number_G'(R)}. If $\delta>0$ is sufficiently small (depending on $\eps$), then
\begin{equation}\label{fewFWithSmallT}
\#\{f\in \mathcal G'(R):\Vert f-g\Vert \le \Gl t\} \leq \#\mathcal G'(R)/3.
\end{equation}

Next we deal with $\Delta(f,g)$. If $\Delta\le \delta^{1-\eps}$ then we immediately have $\Delta(f,g)+\delta\ge \delta^{\eps} \Delta$ for every $f\in \mathcal G'(R)$. If $\Delta> \delta^{1-\eps}$, then we choose
$$
\Gs=|\log \delta|^{-2},
$$
and thus $\Gs^{\eps^{-2}}\Delta\ge\delta^{\eps}\Delta>\delta$ for $\delta$ small enough. Using \eqref{eqn_defn_k} and \eqref{eqn_defn_Delta_x}, we have
\begin{equation*}
    2n_3(\Delta_x,x)\Delta_x^{-\Ge^2}\ge n_4(x)\ge n_3(\Gs^{\Ge^{-2}}\Delta,x)(\Gs^{\Ge^{-2}}\Delta)^{-\Ge^2},
\end{equation*}
and thus $n_3(\Gs^{\Ge^{-2}}\Delta,x)\lesssim  \Gs n_3(\Delta_x,x)$ since $\Delta_x\sim \Delta$. Now recalling the definition of $n_3$ in \eqref{eqn_defn_h}, taking $k=g\in F_B(x)$ shows that
\begin{align*}
    n_3(\Gs^{\Ge^{-2}}\Delta,x)
    &\geq \#\{f\in F_B(x):\Delta(f,g)\leq \Gs^{\Ge^{-2}}\Delta\}.
\end{align*}
Combining the lower and upper bounds for $n_3(\Gs^{\Ge^{-2}}\Delta,x)$ we have
\begin{equation}\label{fewFWithSmallDelta}
\begin{split}
&\#\{f\in \mathcal G'(R):\Delta(f,g)\leq \Gs^{\Ge^{-2}}\Delta\}\\
&\le\#\{f\in F_B(x):\Delta(f,g)\leq \Gs^{\Ge^{-2}}\Delta\}\\
&\lesssim \Gs n_3(\Delta_x,x)\\
&= \Gs\#\{f\in F_B(x):\Delta(f,k_x)\le \Delta_x\}\\
&\le \Gs\mu_2\\
&\lesssim \Gs|\log \delta|\#\mathcal G'(R)\\
&\le \#\mathcal G'(R)/3,
\end{split}
\end{equation}
for $\delta$ small enough, where in the fourth line we have used \eqref{eqn_defn_vx} and in the sixth line we have used $\Delta_x\le \Delta$,  \eqref{eqn_mu_2_number} and \eqref{eqn_number_G'(R)}.

By \eqref{fewFWithSmallT} and \eqref{fewFWithSmallDelta}, we conclude that at least a $1/3$ fraction of the functions $f\in \mathcal G'(R)$ satisfy $\Vert f-g\Vert \geq\delta^{2\eps}t$ and $\Delta(f,g)+\delta\geq\delta^{\eps}\Delta$. The corresponding upper bounds $\Vert f-g\Vert \leq 6t$ and $\Delta(f,g)+\delta\leq 2\Delta$ hold for every $f\in \mathcal G'(R)$.
\end{proof}

\begin{lem}\label{lem_tangent_at_most_1}
For each pair $f,g\in F_B\times F_B$ such that $\Vert f-g\Vert \geq \delta^{2\eps}t$ and $\Delta(f,g)+\delta\geq \delta^\eps\Delta$, there are $\lesssim\delta^{-4\eps}$ rectangles $R\in \mathcal R'$ such that both $f,g\in \mathcal G'(R)$.
\end{lem}
\begin{proof}
If $f\in \mathcal G'(R)$, then $f\in \mathcal G(R)$ and thus $f$ is tangent to $R$. It then suffices to show that there are at most $\lesssim\delta^{-3\eps}$ many $R\in \mathcal R'$ that are tangent to both $f,g$.

Let $f,g$ with $\Vert f-g\Vert \geq \delta^{2\eps}t$ and $\Delta(f,g)+\delta\geq \delta^{\eps}\Delta$, and suppose $R=k^\delta(I)$ is tangent to both $f$ and $g$. Then over $I$ we have $|f(\theta)- g(\theta)|\leq 10\delta$. By Part \eqref{item_2b} of Lemma \ref{prop_two_zeros} at the scale $10\delta$, the union of all $R$ that are tangent to both $f$ and $g$ projects to the $\theta$-axis into a set of Lebesgue measure
$$
\lesssim \delta/\sqrt{\Vert f-g\Vert (\Delta(f,g)+\delta)}\lesssim \delta^{1-(3/2)\eps}(t\Delta)^{-1/2}\sim \delta^{-(3/2)\eps}\sqrt{\delta/t'}.
$$
Using Lemma \ref{prop_number_R=O(1)} with $\lambda\lesssim \delta^{-(3/2)\eps}$, we see there are $\lesssim \delta^{-(15/4)\eps}\le \delta^{-4\eps}$  $R\in \mathcal R'$ that are tangent to both $f,g$.
\end{proof}

\subsubsection{A tangency bound on $\#\mathcal R(\tilde R)$}

Now we give another bound on $M$ in addition to \eqref{eqn_M_bound_measure},  using the conclusions in Section \ref{subsec_number_good_pairs}. For $\tilde R\in \tilde{\mathcal R}'$, define
\begin{align*}
    \mathcal P(\tilde R)=\bigcup_{R\in \mathcal R'(\tilde R)}\{(f,g)\in \mathcal G'(R)^2:\Vert f-g\Vert \geq \delta^{2\eps}t,\ \Delta(f,g)+\delta\geq\delta^{\eps}\Delta \}.
\end{align*}
By Lemma \ref{lem_pairs_lower_bound}, Lemma \ref{lem_tangent_at_most_1} and \eqref{eqn_number_G'(R)} we have
\begin{equation}\label{lowerBdmathcalP}
    \#\mathcal P(\tilde R)\gtrapprox \delta^{4\eps}\mu_2^2M.
\end{equation}
On the other hand, we have a trivial bound (recall \eqref{eqn_F(R)})
\begin{equation}\label{upperBdmathcalP}
    \#\mathcal P(\tilde R)\leq (\# F(\tilde R))^2.
\end{equation}
Combining \eqref{lowerBdmathcalP} and \eqref{upperBdmathcalP} and recalling \eqref{eqn_defn_mu_2} and \eqref{eqn_defn_mu_1}, we conclude that
\begin{equation}\label{eqn_M_bound_tangency}
    M\lessapprox \delta^{-4\eps}\mu_2^{-2}(\# F(\tilde R))^2\lessapprox \delta^{-6\eps}\mu^{-2}(\# F(\tilde R))^2.
\end{equation}
This is our second bound on $M$. Taking a suitable geometric average of these bounds, we have
%But recall we have a bound of $M$ earlier in \eqref{eqn_M_bound_measure}. Raising \eqref{eqn_M_bound_measure} to the power $1/4$ and \eqref{eqn_M_bound_tangency} to the power $3/4$, we obtain
\begin{equation}\label{boundOnM}
    M \lessapprox \eqref{eqn_M_bound_measure}^{1/4}\eqref{eqn_M_bound_tangency}^{3/4}\lessapprox \tau^{-1/4}\delta^{-\Ga/2-6\eps}\Delta^{\Ga/2}\mu^{-3/2}(\# F(\tilde R))^{3/2}.
\end{equation}
Note that \eqref{boundOnM} holds for every $\tilde R\in \tilde{\mathcal R}'$, and the LHS is independent of the choice of $\tilde R$. Thus if we define
\begin{equation*}
    l=\min_{\tilde R\in \tilde{\mathcal R}'}\# F(\tilde R),
\end{equation*}
Then using \eqref{upperBdTau} and \eqref{boundOnM}, \eqref{eqn_number_R_bound} would follow from the estimate
\begin{equation}\label{neededBoundOnMathcalR}
    \#\tilde{\mathcal R}'\leq \delta^{-C_3\eps} \delta^{-\zeta/2}t^{\Ga/2}l^{-3/2}\# F_B,
\end{equation}
for some $C_3=C_3(D)$.

In the next section we will establish this estimate.

\subsection{Proving Inequality \eqref{neededBoundOnMathcalR}}
Our goal in this section is to bound the size of $\tilde{\mathcal R}'$. Our main tool will be Proposition \ref{thm_number_estimate}. However, before the proposition can be applied we need to construct suitable sets $\mathcal{W}$ and $\mathcal{B}$.

\subsubsection{Non-concentration of functions tangent to $\tilde R$}
In this section, we will show that for each coarse-scale rectangle $\tilde R$, the functions from $F$ tangent to $\tilde R$ cannot be concentrated inside a ball of diameter much smaller than $t$.
\begin{lem}\label{lem_Z(tildeR)}
Let $\tilde R\in \tilde {\mathcal R}'$. For every $r>0$ and every $g\in C^2(J)$, we have
\begin{equation}
    \#F(\tilde R)\cap B(g,r)\lesssim |\log\delta|(r/t)^\Ge (t/\Delta)^{\Ge^2}\#F(\tilde R).
\end{equation}
\end{lem}

\begin{proof}
Recall \eqref{sideOfCalZTildeREstimate_ball} that
\begin{equation*}
    \#F(\tilde R)\cap B(g,r)\sim q^{-1}\sum_{R\in \mathcal R(\tilde R)}\#\mathcal G'(R)\cap B(g,r).
\end{equation*}
For each $R\in \mathcal R(\tilde R)$, write $R=R_x$. using Lemma \ref{lem_|z-z'|} with $\lambda=r/t$ and \eqref{eqn_lower_bound_Z_B} we have
\begin{align*}
    \#\mathcal G'(R)\cap B(g,r)\le \#F_B(x)\cap B(g,r)\lesssim (r/t)^\Ge \#F_B(x)\sim (r/t)^\Ge \mu_1.
\end{align*}
Thus
\begin{equation*}
    \#F(\tilde R)\cap B(g,r)\lesssim q^{-1}\#\mathcal R(\tilde R) (r/t)^\Ge \mu_1\sim q^{-1}M(r/t)^\Ge \mu_1.
\end{equation*}
On the other hand, by \eqref{sideOfCalZTildeREstimate} and \eqref{eqn_sum_number_G'(R)}, we have
\begin{align*}
    \#F(\tilde R)
    &\sim q^{-1}\sum_{R\in \mathcal R(\tilde R)}\#\mathcal G'(R)\\
    &\gtrsim q^{-1}M|\log\delta|^{-1}\mu_2\gtrsim q^{-1}M|\log\delta|^{-1}(\Delta/t)^{\Ge^2}\mu_1.
\end{align*}
Thus the result follows.
\end{proof}

By Lemma \ref{lem_Z(tildeR)}, we now take
\begin{equation}\label{defnR}
    r=t c^{1/\Ge}|\log\delta|^{-1/\Ge}(\Delta/t)^{\Ge},
\end{equation}
where $c=c(D)$ is a sufficiently small constant to be determined, such that for every ball $B(g,r)\sub C^2(J)$ and every $\tilde R\in \tilde {\mathcal R}'$ we have
\begin{equation}\label{eqn_before_finding_bipartite}
    \#F(\tilde R)\cap B(g,r)\le c\#F(\tilde R).
\end{equation}
Note that $t\leq \delta^{-2\eps}r/3$ for $\delta$ small enough.

\subsubsection{Applying Proposition \ref{thm_number_estimate}}
Since $\mathcal{F}$ has doubling constant $D$, there exists $C_D\geq 1$ so that we can cover $3B$ by a set of $(3t/r)^{C_D}\leq \delta^{-2C_D\eps}$ balls of radius $r$, so that these balls are boundedly overlapping (with the bound depending on $D$). Denote this set of balls by $\mathcal O.$

If $c=c(D)$ is chosen to be small enough, then for any $\tilde R\in \tilde {\mathcal R}'$, by pigeonholing and \eqref{eqn_before_finding_bipartite}, there are two balls $O_1(\tilde R),O_2(\tilde R)\in \mathcal O$ with their centres $\ge 10r$ separated, such that for $i=1,2,$ we have
\begin{equation}\label{eqn_tangency_lower_bound}
   \#F(\tilde R)\cap O_i(\tilde R)\gtrsim \delta^{2C_D\eps} \#F(\tilde R)\ge \delta^{2C_D\eps}l.
\end{equation}
Since $(\# \mathcal{O})^2 \lesssim\delta^{-4C_D\eps}$, by pigeonholing, there is a pair $(O_1,O_2)\in\mathcal{O}^2$ so that \eqref{eqn_tangency_lower_bound} holds for $\gtrsim\delta^{-4C_D\eps}\#\tilde{\mathcal R}'$ rectangles in  $\tilde{\mathcal R}'$. Fix this choice of $O_1$ and $O_2$, and let  $\tilde{\mathcal R}''\subset \#\tilde{\mathcal R}'$ be the set of rectangles for which \eqref{eqn_tangency_lower_bound} holds.

To establish \eqref{neededBoundOnMathcalR}, it suffices to prove that
\begin{equation}\label{eqn_final_bound_tilde_R}
   \#\tilde{\mathcal R}''\lesssim \delta^{-\zeta/2-C_3\eps}t^{\Ga/2}l^{-3/2}\#F_B,
\end{equation}
for some $C_3=C_3(D)$.

Let $\mathcal W=F_B\cap O_1(\tilde R)$ and $\mathcal B=F_B\cap O_2(\tilde R)$. Then $(\mathcal W,\mathcal B)$ is $r$-separated. Also, by the discussion right after \eqref{eqn_F(R)}, each $(\Delta,C_R t)$-rectangle $\tilde R\in \tilde{\mathcal R}''$ is in particular tangent to
$\ge l$ many $f\in \mathcal W$ and $\ge l$ many $f\in \mathcal B$.

With this, we can now apply Proposition \ref{thm_number_estimate} to the family $\#\tilde {\mathcal R}''$ with $\mu=\nu=l$, $C_R t$ in place of $t$, and $A\le\delta^{-2\eps}$ to get
\begin{equation}\label{eqn_R''}
    \#\tilde {\mathcal R}''\lesssim  \delta^{-C_4\eps}\left(\# F_B/l\right)^{3/2},
\end{equation}
for some $C_4=C_4(D)$.

By \eqref{frostmanConditionOnF} with $3t$ in place of $r$, we have
\begin{equation}\label{boundOnCalFB}
    \#F_B\leq \delta^{-\eps}(t/\delta)^{\zeta}.
\end{equation}
Combining \eqref{eqn_R''} and \eqref{boundOnCalFB} and using the condition $\Ga\le \zeta$, we obtain \eqref{eqn_final_bound_tilde_R} for some suitable $C_3$. This concludes the proof of Theorem \ref{thm_wolffAnalogue_maximal_general}.

\begin{rem}
For some applications, it may be helpful to replace the Frostman condition \eqref{frostmanConditionOnF} with the weaker inequality
\begin{equation}\label{weakerFrostmanConditionOnF}
\#(F \cap B) \leq  A(r/\delta)^\zeta\quad\textrm{for all balls}\ B\subset C^2(I)\ \textrm{of radius}\ r\geq\delta,
\end{equation}
for some $A\geq 1$. By randomly selecting each $f\in F$ with probability $1/A$ and then applying Theorem \ref{thm_wolffAnalogue_maximal_general}, we obtain the following analogue of \eqref{eqn_L32Bd_general}.
\begin{equation}\label{eqn_L32Bd_general_with_A}
 \int_E\Big( \sum_{f\in F}{\bf 1}_{f^\delta}\Big)^{3/2} \leq A^{1/2}\delta^{2-\alpha/2-\zeta/2-C\eps}(\# F).
\end{equation}
\end{rem}

%%%%%%%%%%%%
%%%%%%%%%%%%
%%%%%%%%%%%%

\appendix
\section{Sogge's cinematic curvature, and cinematic functions}\label{cinematicCurvatureAppendix}
In this section, we will explain the connection between cinematic functions (see Definition \ref{cinematicFunctionsDef}) and Sogge's cinematic curvature condition, as used by Kolasa and Wolff in \cite{KolasaWolff}. First, we will show that if $\Phi\colon \RR^4\to\RR$ satisfies the cinematic curvature condition (for simplicity, we will state this condition in a form that is convenient to compute with), then there are sets $Y\times R\subset\RR^3$ and $I\times J\subset\RR^2$, so that for each $(y,r)\in Y\times R$, the curve $\{x\in I\times J\colon \Phi(x,y) = r\}$ is the graph of a function $f_{(y,r)}\colon I\to J$, and $\{f_{(y,r)}\}$ is a family of cinematic functions. The proof of this statement requires careful computation using the inverse function theorem. These computations are described below.

\subsection{Cinematic curvature for normalized defining functions}
Let $U=I\times J\times Y\subset\RR^4,$ where $I,J\subset\RR$ are open intervals containing 0, and $Y\subset\RR^2$ is a neighborhood of 0. Let $\Phi(x,y)\colon U\to\RR$, with
\begin{equation}\label{sizeC3Norm}
\Vert\Phi\Vert_{C^3(U)}=C_0.
\end{equation}
 We will require
\begin{equation}\label{PhiNear0}
\Phi(x,0)=x_2+O(x^3),
\end{equation}
\begin{equation}\label{cinematicAt0}
d_y \Big(\begin{array}{c} \partial_{x_1}\Phi(x,y) \\  \partial^2_{x_1}\Phi(x,y) / |\nabla_x \Phi(x,y)| \end{array}\Big)\Big|_{(x,y)=(0,0)}\quad\textrm{is invertible}.
\end{equation}
Conditions \eqref{PhiNear0} and \eqref{cinematicAt0} are a formulation of Sogge's cinematic curvature condition for the defining function $\Phi$. See, for example Section 3 or the beginning of Section 5 from \cite{KolasaWolff}.

Shrinking $I$ and $Y$ if necessary, by the inverse function theorem we can find an interval $R$ containing 0 so that for each $(y,r)\in Y\times R$, the curve $\{x\in I\times J\colon \Phi(x,y)=r\}$ is the graph $x_2 = f_{y,r}(x_1)$, where $f_{y,r}\in C^2(I)$. We will show that after possibly further shrinking $I$ and $Y,$ the set
\begin{equation}\label{defnM}
\mathcal{F} =\{ f_{y,r}\colon (y,r)\in Y\times R\}\subset C^2(I)
\end{equation}
is a cinematic family of functions. Specifically, we will show there exists $\eps>0$ so that for all $(y,r),(y',r')\in Y\times R$ and all $x_1\in I$, we have
\begin{equation}\label{cinematicCondition}
| f_{y,r}(x_1)- f_{y',r'}(x_1)| + |\dot f_{y,r}(x_1)-\dot f_{y',r'}(x_1)|+|\ddot f_{y,r}(x_1)-\ddot f_{y',r'}(x_1)|\geq \eps|(y,r)-(y',r')|.
\end{equation}
Clearly the metric space $\mathcal{F}$ has bounded diameter, and doubling.

\medskip
\noindent{\bf Step 1}\\
After shrinking $Y,$ $I,$ and $J$ if necessary (and re-defining $U = I\times J \times Y$), we can suppose that
\begin{equation}\label{sizeOfPartials}
1/2\leq \partial_{x_2}\Phi(x,y)\leq 2\quad\textrm{for all}\ (x,y)\in U,
\end{equation}
and there exists $c>0$ so that
\begin{equation}\label{detLarge}
\Big|\det \Big(d_y \Big(\begin{array}{c} \partial_{x_1}\Phi(x,y) \\  \partial^2_{x_1}\Phi(x,y) / |\nabla_x \Phi(x,y)| \end{array}\Big)\Big)\Big|\geq c\quad\textrm{for all}\ (x,y)\in U.
\end{equation}
\eqref{sizeOfPartials} is a consequence of \eqref{PhiNear0}, while \eqref{detLarge} is a consequence of \eqref{cinematicAt0}.

Let $\eps_1,\eps_2,\eps_3$ be small positive numbers. We will choose $\eps_3$ sufficiently small depending on $c$ and $C_0$. We will choose $\eps_2$ sufficiently small depending on $c,$ $C_0$ and $\eps_3$. We will choose $\eps_1$ sufficiently small depending on $c,$ $C_0$ and $\eps_2,\eps_3$.

After further shrinking $Y,$ $I,$ and $J$ if necessary, we can suppose that
\begin{equation}\label{boundOnPartials}
|\partial_{x_1}\Phi(x,y)|\leq \eps_1,\quad |\nabla^2_x\Phi(x,y)|\leq\eps_1\quad\textrm{for all}\ (x,y)\in U.
\end{equation}
\noindent{\bf Step 2}
Let $(y,r),\ (y',r')\in Y\times R$. Let $x_1\in I$, and let $x_2,x_2'\in J$ so that $\Phi(x_1,x_2,y)=r,\ \Phi(x_1,x_2',y')=r'$. We will define $x=(x_1,x_2)$ and $x'=(x_1,x_2')$, so $|x-x'| = |x_2-x_2'|.$ Finally, we will adopt the notation $\Phi_1(x,y) = \partial_{x_1}\Phi(x,y)$; $\Phi_2(x,y) = \partial_{x_2}\Phi(x,y)$; $\Phi_{11}(x,y) = \partial^2_{x_1}\Phi(x,y),$ etc.

If $|x_2-x_2'|\geq\frac{1}{4}|r-r'|,$ then \eqref{cinematicCondition} holds with $\eps =1/4$ and we are done. Thus we can suppose that
\begin{equation}\label{x2x2pClose}
|x_2-x_2'|<\frac{1}{4}|r-r'|.
\end{equation}
We have
\begin{equation*}
|r-r'| = |\Phi(x,y)-\Phi(x',y')| \leq |\Phi(x,y)-\Phi(x',y)| + |\Phi(x',y) - \Phi(x',y')|\leq 2|x-x'|+C|y-y'|.
\end{equation*}
Using \eqref{x2x2pClose} and re-arranging, we conclude that $|r-r'|\leq 2C|y-y'|$, so in particular
\begin{equation}\label{rrPCloseToyyP}
|(y,r)-(y',r')| \leq 3C|y-y'|.
\end{equation}
%Re-arranging, we have $|x_2-x_2'| = |x-x'|\geq \frac{1}{2}(|r-r'|-C|y-y'|)\geq \frac{1}{4}|r-r'|$.

% \begin{equation}\label{x2x2pClose}

% \end{equation}
% In this case, we have

%  \eqref{x2x2pFar} fails. Suppose not

% We will show that at least one of the following two things must happen:

% or
% \begin{equation}\label{rrPCloseToyyP}
% |r-r'|\leq 2C|y-y'|.
% \end{equation}

% Proof: Suppose \eqref{rrPCloseToyyP} fails. Then

% If \eqref{x2x2pFar} holds, then
% \[
% |f_{y,r}(x_1) - f_{y,r}(x_1)|\geq \frac{1}{4}|r-r'|,
% \]
% and we are done.

\medskip
\noindent{\bf Step 2}\\
Let $\eps_2$ be a small constant to be determined later. If $|x_2-x_2'|\geq\eps_2|y-y'|$, then by \eqref{rrPCloseToyyP}, \eqref{cinematicCondition} holds with $\eps =\frac{\eps_2}{3C}$ and we are done. Thus we can suppose that
\begin{equation}\label{x2x2pClose2}
|x_2-x_2'|<\eps_2|y-y'|.
\end{equation}

Next, let $\eps_3$ be a small constant to be determined later. Suppose that
\begin{equation}\label{largeGrad}
|\Phi_1(x',y)-\Phi_1(x,y')|\geq\eps_3|y-y'|.
\end{equation}
We will show that if $\eps_1$ and $\eps_2$ were selected sufficiently small (depending on $c,C_0$ and $\eps_3$), then
\begin{equation}\label{farDerivatives}
|\dot f_{y,r}(x_1)-\dot f_{y',r'}(x_1)|\geq\frac{\eps_3}{10}|y-y'|,
\end{equation}
and thus by \eqref{rrPCloseToyyP}, \eqref{cinematicCondition} holds with $\eps = \frac{\eps_3}{30C}$.

To establish \eqref{farDerivatives}, we will make use of the following inequality
\begin{equation}\label{ABIneq}
\Big|\frac{A}{B}-\frac{A'}{B'}\Big|\geq |A-A'|\frac{|B|}{|BB'|} - |A|\frac{|B-B'|}{|BB'|}
\end{equation}
In practice, when we apply Inequality \eqref{ABIneq}, $B$ and $B'$ will always have $\frac{1}{2}\leq |B|\leq 2$ and similarly for $B'$; $|A-A'|\gtrsim |y-y'|$; $|B-B'|\lesssim |y-y'|$ (with implicit constants that might depend on $c$ and $C_0$), and $|A|\leq \eps_1$. If $A,A',B,B'$ satisfy these conditions, then we will have $|A/B - A'/B'| \geq \frac{1}{8}|A-A'|$, provided $\eps_1$ is chosen sufficiently small.

We have
\begin{equation}\label{expandFDot}
\begin{split}
|\dot f_{y,r}(x_1)-\dot f_{y',r'}(x_1)|& = \Big|\frac{\Phi_1(x,y)}{\Phi_2(x,y)}-\frac{\Phi_1(x',y')}{\Phi_2(x',y')}\Big|\\
&\geq \Big|\frac{\Phi_1(x',y)}{\Phi_2(x',y)}-\frac{\Phi_1(x',y')}{\Phi_2(x',y')}\Big|-\Big|\frac{\Phi_1(x,y)}{\Phi_2(x,y)}-\frac{\Phi_1(x',y)}{\Phi_2(x',y)}\Big|.
\end{split}
\end{equation}
By \eqref{sizeC3Norm}, \eqref{sizeOfPartials}, and \eqref{x2x2pClose2}, we have
\[
\Big|\frac{\Phi_1(x,y)}{\Phi_2(x,y)}-\frac{\Phi_1(x',y)}{\Phi_2(x',y)}\Big|\leq 4C|x-x'|\leq 4C_0\eps_2|y-y'|.
\]

Applying Inequality \eqref{ABIneq} and using assumption \eqref{largeGrad}, we conclude that
\[
\Big|\frac{\Phi_1(x',y)}{\Phi_2(x',y)}-\frac{\Phi_1(x',y')}{\Phi_2(x',y')}\Big|\geq\frac{1}{8}|\Phi_1(x',y)-\Phi_1(x',y')|\geq \frac{\eps_3}{8}|y-y'|.
\]
Thus if $\eps_2$ is chosen sufficiently small depending on $C_0$ and $\eps_3$, then \eqref{farDerivatives} follows from \eqref{expandFDot}.

\medskip
\noindent{\bf Step 3}\\
Suppose now that \eqref{largeGrad} is false, i.e.
\begin{equation}\label{smallGrad}
|\Phi_1(x',y)-\Phi_1(x,y')|<\eps_3|y-y'|.
\end{equation}
Then by \eqref{detLarge}, we have
\begin{equation}\label{largePhi11Nabla}
\Big|\frac{\Phi_{11}(x',y)}{|\nabla\Phi(x,y)|}-\frac{\Phi_{11}(x',y')}{|\nabla\Phi(x',y')|}\Big|\geq\frac{c}{4}|y-y'|.
\end{equation}
We will show that if $\eps_1,\eps_2,\eps_3$ were chosen sufficiently small, then
\begin{equation}\label{farDoubleDerivatives}
|\ddot f_{y,r}(x_1)-\ddot f_{y',r'}(x_1)|\geq\frac{c}{10}|y-y'|,
\end{equation}
and thus by \eqref{rrPCloseToyyP},  \eqref{cinematicCondition} holds with $\eps = \frac{c}{30C}.$ Thus \eqref{farDoubleDerivatives} implies that set $M$ defined in \eqref{defnM} is a cinematic family of functions. It remains to establish \eqref{farDoubleDerivatives}. The proof of this inequality uses the same ideas as the proof of \eqref{farDerivatives}, though the steps are slightly more complicated.

Define
\[
F(a,b) = \frac{\Phi_{11}(a,b)\Phi_2(a,b)^2
-2\Phi_{12}(a,b)\Phi_1(a,b)\Phi_2(a,b)+\Phi_{22}(a,b)\Phi_1(a,b)^2
}
{\Phi_2(a,b)^3}
\]

Then
\[
|\ddot f_{y,r}(x_1)-\ddot f_{y',r'}(x_1)| = |F(x,y) - F(x',y')| \geq |F(x',y)-F(x,y')| - |F(x,y)-F(x',y)|.
\]
By \eqref{sizeC3Norm}, \eqref{sizeOfPartials}, and \eqref{x2x2pClose2}, we have
\begin{equation}\label{replacexWithXp}
|F(x,y)-F(x',y)|\lesssim C|x-x'|\lesssim C\eps_2|y-y'|.
\end{equation}
Thus our goal is to establish a lower bound for $|F(x',y)-F(x,y')|$. We will estimate
\begin{equation}\label{FEstimate}
\begin{split}
|F(x',y)-F(x,y')| \geq \Big|& \frac{\Phi_{11}(x',y)}{\Phi_2(x',y)}-\frac{\Phi_{11}(x',y')}{\Phi_2(x',y')}\Big|\\
& - 2\Big|\frac{\Phi_{12}(x',y)\Phi_1(x',y)}{\Phi_2(x',y)^2}-\frac{\Phi_{12}(x',y')\Phi_1(x',y')}{\Phi_2(x',y')^2}  \Big|\\
& - \Big|\frac{\Phi_{22}(x',y)\Phi_1(x',y)^2}{\Phi_2(x',y)^3}-\frac{\Phi_{22}(x',y')\Phi_1(x',y')^2}{\Phi_2(x',y')^3}  \Big|.
\end{split}
\end{equation}
For the first term, we will use the inequality $AB-A'B' = (A-A')B' + (B-B')A$ to write
\begin{equation*}
\begin{split}
&\frac{\Phi_{11}(x',y)}{\Phi_2(x',y)}-\frac{\Phi_{11}(x',y')}{\Phi_2(x',y')} =
\frac{\Phi_{11}(x',y)}{|\nabla\Phi(x',y)|}\Big(\frac{\Phi_1^2(x',y)}{\Phi_2^2(x',y)}+1\Big)^{1/2}-
\frac{\Phi_{11}(x',y')}{|\nabla\Phi(x',y')|}\Big(\frac{\Phi_1^2(x',y')}{\Phi_2^2(x',y')}+1\Big)^{1/2}\\
& = \Big(\frac{\Phi_{11}(x',y)}{|\nabla\Phi(x',y)|}-\frac{\Phi_{11}(x',y')}{|\nabla\Phi(x',y')|}\Big)\Big(\frac{\Phi_1^2(x',y')}{\Phi_2^2(x',y')}+1\Big)^{1/2} \\
&+ \Big[\Big(\frac{\Phi_1^2(x',y)}{\Phi_2^2(x',y)}+1\Big)^{1/2}-\Big(\frac{\Phi_1^2(x',y')}{\Phi_2^2(x',y')}+1\Big)^{1/2}\Big]\frac{\Phi_{11}(x',y)}{|\nabla\Phi(x',y)|}
\end{split}
\end{equation*}
Using \eqref{boundOnPartials} and \eqref{sizeOfPartials}, and then \eqref{largePhi11Nabla},  We have
\begin{equation*}
\begin{split}
\Big|\Big(\frac{\Phi_{11}(x',y)}{|\nabla\Phi(x',y)|}-\frac{\Phi_{11}(x',y')}{|\nabla\Phi(x',y')|}\Big)\Big(\frac{\Phi_1^2(x',y')}{\Phi_2^2(x',y')}+1\Big)^{1/2}\Big|
&\geq \frac{1}{2}\Big|\frac{\Phi_{11}(x',y)}{|\nabla\Phi(x',y)|}-\frac{\Phi_{11}(x',y')}{|\nabla\Phi(x',y')|}\Big|\\
&\geq \frac{c}{4}|y-y'|,
\end{split}
\end{equation*}
and by \eqref{boundOnPartials},
\begin{equation*}
\begin{split}
&\Big[\Big(\frac{\Phi_1^2(x',y)}{\Phi_2^2(x',y)}+1\Big)^{1/2}-\Big(\frac{\Phi_1^2(x',y')}{\Phi_2^2(x',y')}+1\Big)^{1/2}\Big]\frac{\Phi_{11}(x',y)}{|\nabla\Phi(x',y)|}\\
&\leq \Big[\Big(\frac{\Phi_1^2(x',y)}{\Phi_2^2(x',y)}+1\Big)^{1/2}-\Big(\frac{\Phi_1^2(x',y')}{\Phi_2^2(x',y')}+1\Big)^{1/2}\Big](2\eps_1)\\
&\leq (2C)|y-y'|(2\eps_1).
\end{split}
\end{equation*}
Thus if $\eps_1$ is selected sufficiently small, we have
\[
\Big|\frac{\Phi_{11}(x',y)}{\Phi_2(x',y)}-\frac{\Phi_{11}(x',y')}{\Phi_2(x',y')}\Big|\geq \frac{c}{5}|y-y'|.
\]
Bounding the second and third term on the RHS of \eqref{FEstimate} is straightforward, using the estimates \eqref{sizeOfPartials}, \eqref{boundOnPartials}, and \eqref{smallGrad}. We conclude that

\[
|F(x',y)-F(x,y')|\geq \frac{c}{6}|y-y'|.
\]
Combining this with \eqref{replacexWithXp}, we obtain \eqref{farDoubleDerivatives}.

\bibliographystyle{plain}
\bibliography{references}
\end{document}